\documentclass[10pt]{article}

\usepackage{amsmath,amsthm} 
\usepackage{amssymb}
\usepackage[dvipsnames, table]{xcolor} 
\usepackage{rotating}
\usepackage[all]{xy} 
\usepackage{comment} 
\usepackage{tikz} 
\usetikzlibrary{arrows, arrows.meta, decorations.markings, decorations.pathmorphing, shapes} 
\usepackage{comment} 
\usepackage{centernot}  
\usepackage{mathtools} %
\usepackage{ stmaryrd } 
\usepackage{multicol} 
\usepackage{mathabx} 
\usepackage[colorlinks=true]{hyperref}
\usepackage{mathrsfs}
 
\numberwithin{equation}{section}

\newtheorem{thm}{Theorem}[section]
\newtheorem{cor}[thm]{Corollary}
\newtheorem{conj}[thm]{Conjecture}
\newtheorem{lem}[thm]{Lemma}
\newtheorem{prop}[thm]{Proposition}
\theoremstyle{definition} 
\newtheorem{defn}[thm]{Definition}
\newtheorem{rem}[thm]{Remark}
\newtheorem{exam}[thm]{Example}

\newcommand{\bA}{\mathbb{A}} 
\newcommand{\bC}{\mathbb{C}}

\newcommand{\bR}{\mathbb{R}}
\newcommand{\bZ}{\mathbb{Z}}
\newcommand{\bN}{\mathbb{N}} 
\newcommand{\cB}{\mathcal{B}}
\newcommand{\cF}{\mathcal{F}}
\newcommand{\cW}{\mathcal{W}} 
\newcommand{\fB}{\mathfrak{B}} 
\newcommand{\fR}{\mathfrak{R}} 
\newcommand{\CE}{\operatorname{CE}}
\newcommand{\red}{\text{red}}
\newcommand{\gr}{\operatorname{gr}}
\newcommand{\ev}{\operatorname{ev}}
\newcommand{\op}{\operatorname{op}}

\newcommand{\Id}{\operatorname{Id}}
\newcommand{\Hom}{\operatorname{Hom}}
\newcommand{\RHom}{\operatorname{RHom}}
\newcommand{\End}{\operatorname{End}}
\newcommand{\Aut}{\operatorname{Aut}}
\newcommand{\Ext}{\operatorname{Ext}}
\newcommand{\REnd}{\operatorname{REnd}}
\newcommand{\Rep}{\text{Rep}}
\newcommand{\Spec}{\operatorname{Spec}}

\newcommand{\Mat}{\operatorname{Mat}}
\newcommand{\GL}{\text{GL}}
\newcommand{\MC}{\text{MC}}
\newcommand{\wMC}{\widehat{MC}}
\newcommand{\fg}{\mathfrak{g}}
\newcommand{\fh}{\mathfrak{h}}
\newcommand{\imm}{\operatorname{im}}

 \newcommand\scalemath[2]{\scalebox{#1}{\mbox{\ensuremath{\displaystyle #2}}}}

\definecolor{forest}{rgb}{.133,.545,.133}

\newcommand{\Addresses}{{
  \bigskip
  \footnotesize

  D.~Kaplan, \textsc{School of Mathematics, University of Birmingham,  Birmingham UK B15 2TT}\par\nopagebreak
  \textit{E-mail address}, D.~Kaplan: \texttt{d.kaplan@bham.ac.uk}

  \medskip

  T.~Schedler, \textsc{Department of Mathematics, Imperial College London,
    London, UK SW7 2AZ}\par\nopagebreak
  \textit{E-mail address}, T.~Schedler: \texttt{t.schedler@imperial.ac.uk}
}}

\begin{document}

\title{Multiplicative preprojective algebras are 2-Calabi--Yau}
\author{Daniel Kaplan and Travis Schedler}
\date{}

\maketitle

\begin{abstract}
We prove that multiplicative preprojective algebras, defined by Crawley-Boevey and Shaw, are 2-Calabi--Yau algebras, in the case of quivers containing  unoriented cycles. If the quiver is not itself a cycle, we show that the center is trivial, and hence the Calabi--Yau structure is unique.  If the quiver is a cycle, we show that the algebra is a non-commutative crepant resolution of its center, the ring of functions on the corresponding multiplicative quiver variety with a type A surface singularity.   We  also  prove  that  the  dg  versions  of  these  algebras (arising as certain Fukaya  categories) are  formal. We  conjecture  that  the  same  properties hold for all non-Dynkin quivers, with respect to any extended Dynkin subquiver (note that  the  cycle  is  the  type  A  case). Finally,  we  prove  that multiplicative quiver varieties---for  all  quivers---are  formally  locally  isomorphic  to  ordinary  quiver varieties.  In particular, they are all symplectic singularities (which implies they are normal and have rational Gorenstein singularities). This includes character varieties of Riemann surfaces with punctures and monodromy conditions. We deduce this from a more general statement about 2-Calabi--Yau algebras (following Bocklandt, Galluzzi, and Vaccarino).
\end{abstract}


\tableofcontents

\section{Introduction}

Multiplicative preprojective algebras have recently gained attention in geometry and topology. These algebras appear in the study of certain wrapped Fukaya categories, see \cite{Lekili}, \cite{Lekili2}, in the study of microlocal sheaves on rational curves, see \cite{Bez}, and in the study of generalized affine Hecke algebras, see Appendix 1 in \cite{Rains}. Their moduli spaces of representations are called multiplicative quiver varieties, and are analogues of Nakajima's quiver varieties. These include character varieties of rank $n$ local systems on closed Riemann surfaces, or on open Riemann surfaces with punctures and monodromy conditions \cite{Shaw, Yamakawa, Schedler}. 
Multiplicative quiver varieties have also been studied from various viewpoints in \cite{VdB_GVMM}, \cite{Boalch}, \cite{CB2}, and \cite{OM}.
A quantization was defined in \cite{Jordan} and further studied in \cite{Pavel}.

Historically, Crawley-Boevey and Shaw defined the multiplicative preprojective algebra, in \cite{Shaw}, to view solutions of the Deligne--Simpson problem as irreducible representations of multiplicative preprojective algebras of certain star-shaped quivers. Their paper establishes the foundations for much of this work. For a fixed field $k$ and a quiver $Q$ with vertex set $Q_{0}$ and arrow set $Q_{1}$ and $q \in (k^\times)^{Q_{0}}$, Crawley-Boevey and Shaw define
$$
\Lambda^{q}(Q) := \frac{L_{Q}}{J_{Q}} := \frac{k \overline{Q} [ (1+ a a^*)^{-1}]_{a \in \overline{Q}}}{\left ( r := \prod_{a \in Q_{1}} (1+ a a^*)(1+ a^* a)^{-1} - q  \right )},
$$
a quotient of the localized path algebra of the double quiver, $L_{Q}$, by the two-sided ideal $J_Q$ generated by the single relation, $r$. 

Many of the desirable properties of the (additive) preprojective algebra seem to hold for the multiplicative preprojective algebra. But establishing this rigorously is difficult, as most proof techniques in the additive case (employing the grading on the algebra) are not available in the multiplicative case. In particular, the multiplicative preprojective algebra is not in general a deformation of the ordinary one, nor does it have a useful Hilbert series for a filtration (due to the localization). 

The goal of this paper is to overcome these difficulties when the quiver contains a cycle, and properly formulate the general expectations. This is sufficient for applications to multiplicative quiver varieties for \emph{every} quiver. 

The main statement is the following:

\begin{conj} \label{conj: 2CY}
$\Lambda^{q}(Q)$ is 2-Calabi--Yau for all $q \in (k^\times)^{Q_{0}}$ and all $Q$ connected and not Dynkin; moreover, it is a prime ring, and the family $\Lambda^q(Q)$ is flat in $q$.\footnote{A prime ring is a non-commutative analogue of an integral domain, being a ring $R$ in which $aRb=0$ implies $a=0$ or $b=0$.}
If $Q$ is furthermore not extended Dynkin, then $Z(\Lambda^q(Q))=k$, and the Calabi--Yau structure is unique. 
\end{conj}
  
Here (extended) Dynkin refers to the underlying unoriented graph being of types A, D, or E.  We explain how one can reduce the conjecture to the case where $Q$ is extended Dynkin in Section \ref{ss: cont cycle}. We carry out this procedure for $Q = \widetilde{A_n}$ and thereby prove the conjecture for all connected quivers containing it:
\begin{thm} \label{thm: 2CY}
$\Lambda^{q}(Q)$ is 2-Calabi--Yau and prime for any $q \in (k^\times)^{Q_{0}}$ and any $k$ a field, and $Q$ connected and containing an unoriented cycle. The family of algebras is flat in $q$. If the quiver properly contains a cycle, then $Z(\Lambda^q(Q))=k$, and the Calabi--Yau structure is unique.
\end{thm} 

This theorem is established in Corollary \ref{cor: 2-CY}, Corollary \ref{cor: flat},  and the results of Section \ref{s:center}: Proposition \ref{prop: prime for quivers containing cycle}, Proposition \ref{prop: general center trivial}, and Corollary \ref{c: unique CY}.  Each relies on technical results proven in Section \ref{s: free-product}. Before outlining the proof techniques, we give four different perspectives on this work:\\
\\
\textbf{(I) Symplectic Topology: Wrapped Fukaya Categories} \\
Multiplicative preprojective algebras arise from studying certain wrapped Fukaya categories. Let $X_{\Gamma}$ be the Weinstein manifold formed by plumbing cotangent bundles of 2-spheres according to the graph $\Gamma$. Ekholm and Lekili \cite{Lekili3} and Etg{\"u} and Lekili \cite{Lekili}, \cite{Lekili2} produced quasi-isomorphisms,
$$
\xymatrix{
\cW(X_{\Gamma}) \ar[rrr]^{\cite{Lekili3}}_{\cong} & & & 
\mathscr{B}_{\Gamma} \ar[rrr]^{\cite{Lekili}, \cite{Lekili2}}_{\cong} & & & \mathscr{L}_{\Gamma}
}
$$
where $\cW(X_{\Gamma})$ is the partially wrapped Fukaya category of $X_{\Gamma}$,  $\mathscr{B}_{\Gamma}$ is the Chekanov--Eliashberg dg-algebra and $\mathscr{L}_{\Gamma}$ is the dg multiplicative preprojective algebra following \cite{Ginzburg}, with $q=1$, see Definition \ref{def: dg_mult_preproj}. 


Since $X_{\Gamma}$ is a Liouville manifold, \cite[Theorem 1.3]{Ganatra} shows that $\cW(X_{\Gamma})$ is a 2-Calabi-Yau category and hence $\mathscr{L}_{\Gamma}$ is 2-Calabi--Yau, as a dg-algebra. We establish this result purely algebraically, in the case $\Gamma$ contains a cycle. In particular, we show in this case that 
$$
\xymatrix{
 \Lambda^{1}(\Gamma) = H^{0}(\mathscr{L}_{\Gamma})  \ar@{^{(}->}[rr]^-{\text{Prop \ref{prop: formal}}}_-{=} & & H^{*}(\mathscr{L}_{\Gamma}) 
}
$$ 
and hence $\mathscr{L}_{\Gamma}$ is formal. By Theorem \ref{thm: 2CY}, the dg multiplicative preprojective algebra $\mathscr{L}_{\Gamma}$ is formal.\footnote{Additionally, since submission of this article, the dg multiplicative preprojective algebra was shown to be 2-Calabi--Yau for all $q$ and all $\Gamma$ in \cite{BCS}.}


Consequently, deformations of the wrapped Fukaya category, $\cW(X_{\Gamma})$, as an $A_{\infty}$-category (respectively Calabi--Yau $A_{\infty}$-category) over a degree zero base, are given by deformations of $\Lambda^{1}(\Gamma)$ as an associative algebra (respectively Calabi--Yau algebra). The infinitesimal deformations can be identified with $\text{HH}^{2}(\Lambda^{1}(\Gamma))$. 
Thanks to Theorem \ref{thm: 2CY}, $\Lambda^1(\Gamma)$ is 2-Calabi--Yau. Hence, Van den Bergh duality identifies
$\text{HH}^{2}(\Lambda^{1}(\Gamma))$
with $\text{HH}_{0}(\Lambda^{1}(\Gamma))$. The techniques in \cite{Schedler_HH} can likely be adapted to compute the latter
using the explicit basis for $\Lambda^{q}(Q)$ computed here. Furthermore, by the 2-Calabi--Yau property, $\text{HH}^3(\Lambda^q(Q))=0$, so there are no obstructions to extending to infinite order deformations. 

We conjecture that the same holds for every connected, non-Dynkin quiver. More precisely, in addition to Conjecture \ref{conj: 2CY}, we expect the following:
\begin{conj}\label{conj:formal}
If $Q$ is connected and not  Dynkin, then the dg multiplicative preprojective algebra $\Lambda^{dg,q}(Q)$ is quasi-isomorphic to $\Lambda^q(Q)$, in degree zero.
\end{conj}
We give the precise definitions and details, as well as proof in the case $Q$ contains a cycle, in Section \ref{s:formal}. 
\\
\\
\textbf{(II) Quiver varieties: Local structure of multiplicative quiver varieties} \\ 
\hspace*{3.73cm} \textbf{and moduli spaces attached to $2$-Calabi--Yau algebras} \\
Given a dimension vector $\alpha \in \bN^{Q_0}$, the affine multiplicative quiver variety is defined as the (coarse) moduli scheme of representations of $\Lambda^q(Q)$. Explicitly, it is the geometric invariant theory quotient of the space of all representations $\Lambda^q(Q) \to \bigoplus_i \Mat_{\alpha_i}(k)$ by the action of $\prod_i \GL(\alpha_i)$ by change of basis. See Section \ref{s:quiver_varieties} for more details (where we also recall a version incorporating a stability condition).

Properties of multiplicative preprojective algebras determine properties of the corresponding multiplicative quiver varieties. 
For instance, in Section 7.5 of \cite{Schedler}, Tirelli and the second author observe, following \cite{Bocklandt}, that the 2-Calabi--Yau property determines the (formal) local structure of the moduli space of representations. Namely, any formal neighborhood can be identified with the formal neighborhood of the zero representation of the moduli space of representations of some (additive) preprojective algebra.  This is proved in more detail here, in Theorems \ref{t:fnt-general} (filling in some more details from \cite[\S 6]{Bocklandt} where a similar result is given).  Among other applications mentioned in \cite{Schedler}, it follows that, when $k$ has characteristic zero, the corresponding multiplicative quiver varieties are normal and are symplectic singularities in the sense of Beauville \cite{Beauss} (in particular, they are normal and have rational Gorenstein singularities): see Corollary \ref{c:symp-sing}.

This includes (as an open subset) character varieties of Riemann surfaces of positive genus with punctures and prescribed monodromy conditions, as explained in \cite[\S 3]{Schedler} (following \cite{Shaw, Yamakawa}). (In the case of closed Riemann surfaces, as pointed out in
\cite{Bellamy-Schedler-character}, this statement does not require our
result, since the group algebra $k[\pi_1(\Sigma)]$ is well-known to be 2-Calabi--Yau.)

One subtle point is that we can describe the local structure of multiplicative quiver varieties for \emph{all} quivers despite the fact that we only prove the 2-Calabi--Yau property for quivers containing a cycle (Theorem \ref{t:fnt-mult}). The key idea is that any quiver can be embedded into a quiver containing a loop and hence any representation of a quiver can be viewed as a representation of a quiver with a cycle. Therefore, its formal neighborhood can be identified with a formal neighborhood of the zero representation of an (additive) preprojective algebra. For detailed definitions, statements, and proofs see Section \ref{s:quiver_varieties}.\\
\\
\textbf{(III) Non-commutative algebraic geometry: Non-commutative resolutions}\\
Although in the non-Dynkin, non-extended Dynkin case, the center is expected to be trivial (Conjecture \ref{conj: 2CY}, proved when the quiver contains a cycle), this is far from true in the extended Dynkin case.  
Indeed, ordinary
preprojective algebras of extended Dynkin quivers have a large center, the spectrum of which is a du Val singularity. The algebra itself is a non-commutative crepant resolution of this center. Moreover, this center is the algebra of functions on a natural quiver variety.  So it is reasonable to ask if multiplicative preprojective algebras also resolve the corresponding multiplicative quiver variety.

In Shaw's thesis, he makes progress towards this question by showing that, for an extended Dynkin quiver $Q$ with extended vertex $v$, the subalgebra $e_{v} \Lambda^{1}(Q) e_{v}$ 
is commutative of dimension 2, with a unique singularity at the origin \cite{Shaw_thesis}; he expects that (for $k$ of characteristic zero) the singularity there has the corresponding du Val type (cf.~Remark \ref{r:wemyss}).

In further analogy to the additive case, it is reasonable to pose the following conjecture:
\begin{conj} \label{conj: NCCR}
Let $Q$ be extended Dynkin. The algebra $\Lambda^1(Q)$ is a 2-dimensional non-commutative crepant resolution (NCCR) of its center, which is the ring of functions on the associated multiplicative quiver variety $\mathcal{M}_{1,0}(Q,\delta)$. Moreover, 
the Satake map\footnote{We use the terminology ``Satake'' following the analogous one for symplectic reflection algebras at $t=0$ of Etingof--Ginzburg,  itself coming from the map for affine Hecke algebras proposed by Lusztig.} $Z(\Lambda^1(Q)) \to e_v \Lambda^1(Q) e_v$, defined by $z \mapsto e_v z$, is an isomorphism.  
\end{conj}
See Section \ref{s:quiver_varieties} for the precise definition of the multiplicative quiver variety. Thanks to our aforementioned results on its local structure, the conjecture implies Shaw's expectation that the singularity of $e_{v} \Lambda^{1}(Q) e_{v}$ is du Val of the corresponding type.

For $Q = \widetilde{A_n}$, we prove the conjecture in 
Section \ref{s:NCCR_proof}. In the process, we obtain an explicit description of the center, $Z(\Lambda^1(\widetilde{A_{n}}))$, which may be of independent interest.\\


\noindent \textbf{(IV) Representation Theory: Kontsevich--Rosenberg Principle} \\
A final perspective on this work involves the Kontsevich--Rosenberg Principle which says: a non-commutative geometric structure on an associative algebra $A$ should induce a geometric structure on the representation spaces $\Rep_{n}(A)$, for all $n \geq 1$. This principle needs adjusting for structures living in the \emph{derived} category of $A$-modules, as the representation functor is not exact. For a d-Calabi--Yau structure on $A$, it is shown in \cite{Brav} and \cite{Wai-kit} 
 that the derived moduli stack of perfect complexes of $A$-modules, $\bR \text{Perf}(A)$, has a canonical $(2-d)$-\emph{shifted} symplectic structure.  Since the dg multiplicative preprojective algebra is 2-Calabi--Yau, this implies that its moduli stack of representations has a $0$-shifted symplectic structure. By 
 Conjecture \ref{conj:formal}, it is the same as the moduli stack of representations of $\Lambda^q(Q)$ itself.  Note that the multiplicative quiver variety can be viewed as a coarse moduli space of semistable representations; so the aforementioned result that this variety locally has the structure of an ordinary quiver variety is a singular analogue of the statement on the moduli stack.
 \\
 \\
 We now give a brief overview of the proof of Conjecture \ref{conj: 2CY} and \ref{conj:formal} for quivers containing a cycle. We prove Theorem \ref{thm: 2CY} using a complex:
$$
P_{\bullet} := \Lambda^{q}(Q) \otimes_{kQ_{0}} kQ_{0} \otimes_{kQ_{0}} \Lambda^{q}(Q) \overset{\alpha}{\longrightarrow}
\Lambda^{q}(Q) \otimes_{kQ_{0}} k\overline{Q}_{1} \otimes_{kQ_{0}} \Lambda^{q}(Q) \overset{\beta}{\longrightarrow}
\Lambda^{q}(Q) \otimes_{kQ_{0}} \Lambda^{q}(Q)
$$
defined originally in \cite{Shaw} (following \cite[Theorem 10.3]{Schofield} and \cite[Corollary 2.11]{Cuntz-Quillen}) and shown to resolve $\Lambda^{q}(Q)$, \emph{except} for the injectivity of the map $\alpha$. We show $\alpha$ is injective and then show the dual complex $P_{\bullet}^{\vee}$ is a resolution of $\Lambda^{q}(Q)[-2]$, which implies $\Lambda^{q}(Q)$ is 2-Calabi--Yau. 

First, we establish a chain of implications to reduce the proof to a presentation of the localization $L_{Q}$ that we call the \emph{strong free product property}, established in Theorem \ref{thm:sfpp}, see Definition \ref{def: sfpp} or see below for a rough definition. The strong free product property is a version of Anick's weak summand property in the ungraded case, see \cite{Anick}. 

To prove the 2-Calabi--Yau property from the strong free product property we establish the chain of implications:
$$
\begin{array}{c}
\text{Strong Free Product Property for } Q: \\ 
\exists \ \sigma': \Lambda^{q}(Q) *_{kQ_{0}} kQ_{0}[t, (q+t)^{-1}] \rightarrow L_{Q} \text{ a linear isomorphism} \\
\textcolor{white}{Section 3.1} \Downarrow \text{ \textcolor{blue}{Section \ref{section: definition of sffp}}} \\
\text{Weak Free Product Property for } Q: \\
\gr(\sigma'): \Lambda^{q}(Q) *_{kQ_{0}} kQ_{0}[t] \rightarrow \gr(L_{Q}) \text{ is an algebra isomorphism} \\
\textcolor{white}{Prop 3.12} \Downarrow \text{ \textcolor{blue}{Prop  \ref{prop: wfpp implies P resolution}}} \\
\gr(\sigma')_{1}:\Lambda^{q}(Q) \otimes_{kQ_{0}} kQ_{0}[t] \otimes_{kQ_{0}} \Lambda^{q}(Q)  \rightarrow  J_{Q}/J_{Q}^{2} \text{ is an isomorphism of } \Lambda^{q}(Q)\text{-bimodules} \\
\textcolor{white}{Prop 3.11 \& Prop 3.12.3} \Downarrow \text{ \textcolor{blue}{Prop  \ref{prop: P is a resolution} \& Prop  \ref{prop: wfpp implies P resolution}}} \\
P_{\bullet} \text{ is a length two projective } \Lambda^{q}(Q)\text{-bimodule resolution of } \Lambda^{q}(Q) \\
\textcolor{white}{Prop 3.3.3} \Downarrow \text{ \textcolor{blue}{\text{Thm } \ref{thm: commuting diagram dual}}}   \\
\Lambda^{q}(Q) \text{ is 2-Calabi--Yau}
\end{array}
$$
Here the isomorphism $\sigma'$ is determined by a choice of $kQ_0$-bimodule section $\Lambda^q(Q) \to L_Q$ of the quotient map $L_{Q} \twoheadrightarrow \Lambda^q(Q)$, but $\gr(\sigma')$ is independent of this choice. The element $t$ maps to the relation, and the filtrations used are the $t$-adic one on the source and the $J_Q$-adic filtration on $L_{Q}$.  

In Section \ref{s:formal}, we show that the strong free product property implies that the dg multiplicative preprojective algebra $\Lambda^{dg,q}(Q)$ is formal. 
Therefore, by the results of this paper, both Conjectures \ref{conj: 2CY} and \ref{conj:formal} would follow from the following more general statement (see Section \ref{s: free-product} for precise details).
\begin{conj} \label{conj: sfpp}
If $Q$ is a connected, non-Dynkin quiver, then $\sigma'$ as above is a linear isomorphism: $(L_Q, r, \sigma, kQ_0[t,(t+q)^{-1}])$ satisfies the strong free product property.
\end{conj}
Proposition \ref{p:mpa-general-q} proves the conjecture for quivers containing a cycle. This is the technical heart of the paper. Our main technique involves reduction systems over the localized ring $kQ_0[t,(t+q)^{-1}]$.  Using the Diamond Lemma \cite{Bergman}, we show these give unique reductions of elements of $L_Q$ to basis elements of the given free product. As a consequence, $\Lambda^q(Q)$ itself obtains the module structure of a free product of the cycle part and the rest of the quiver: see \eqref{e:iso-general-q} for a precise statement.\\

\begin{rem}After submission of this article, in \cite[Theorem 1.1]{CBK-deformed}, Crawley--Boevey and Yuta Kimura proved that a related, more well-studied algebra, the \emph{deformed preprojective algebra} \cite{CBH}, is 2-Calabi--Yau, in the case that the quiver is connected and non-Dynkin. This algebra is a deformation of the usual (additive) preprojective algebra, given as the quotient of the path algebra of the double quiver by the single relation $\sum_{a \in Q_1} a a^* - a^* a - \sum_{i \in Q_0} \lambda_i e_i$, the case $\lambda_i=0$ returning the original preprojective algebra. This can also be proved via the techniques of this article, by deducing  the strong free
product property 
from the known one for the additive preprojective algebra for non-Dynkin quivers. 

Namely, the latter are non-commutative complete intersections \cite{Anick, Etingof-Ginzburg}, shown in 
\cite[Proposition 5.2.1]{Schedler_HH} to be equivalent, in the context of graded algebras, to the (strong or weak) free product property.
More generally, let $A = TV/(r)$ be a graded algebra satisfying the free product property. Briefly, this means that we have a section $\sigma: A \to TV$, which we can take to be graded, so that the induced linear map
 $\sigma': A*_k k[t] \to TV$ sending $t$ to $r$ is a linear isomorphism. Then for every $\lambda \in k$, $\sigma'$ also defines a linear isomorphism by the same formula except sending $t$ to $r+\lambda$. This is because, taking a homogeneous basis of $A$ with degree nondecreasing, we obtain a homogeneous basis of $TV$ via the free product, and the substitution $r \mapsto r+\lambda$ is a strictly triangular change of basis.  Thus, the algebra $TV/(r+\lambda)$ also satisfies the strong free product property. 
Note that the same argument given here applies if we replace $r$ by any filtered deformation $r + r'$, with $r'$ in degrees strictly lower than $r$. They also apply to the quiver context, replacing $k$ by $kQ_0$, hence imply that the deformed preprojective algebra satisfies the strong free product property.
It seems likely this argument can apply to many other interesting algebras.
\end{rem}

An outline of the paper is as follows. 
In Section \ref{s:MPA}, we give elementary background information on multiplicative preprojective algebras and produce an alternative generating set crucial for our approach to the 2-Calabi--Yau property. 
In Section \ref{s: 2CY-property}, we prove the 2-Calabi--Yau property for $\Lambda^{q}(Q)$ assuming the strong free product property.
In Section \ref{s:formal}, using the strong free product property, we show that the dg multiplicative preprojective algebra has homology $\Lambda^{q}(Q)$, concentrated in degree zero. In Section \ref{s:quiver_varieties}, we use the 2-Calabi--Yau property to describe the formal neighborhoods of multiplicative quiver varieties as formal neighborhoods of the zero representation in certain quiver varieties. In Section \ref{s:NCCR} we use the 2-Calabi--Yau and prime properties in the cycle case, together with work of Shaw, to show the multiplicative preprojective algebra is a non-commutative resolution over its center. In Section \ref{s: free-product}, we prove the strong free product property first for multiplicative preprojective algebras of cycles, then for partial multiplicative preprojective algebras. Putting the two together, we deduce the strong free product property for connected quivers containing cycles. The key point of the argument, of independent interest, is a construction of bases of these algebras. Finally, in Section \ref{s:center}, we establish the prime property of $\Lambda^{q}(Q)$ using our explicit bases. We furthermore show that $Z(\Lambda^{q}(Q))=k$ for $Q$ connected and properly containing a cycle. 
This shows that the Calabi--Yau structure in these cases are unique, up to scaling. 
\\
\\
{\bf Acknowledgements} \\ 
The first author was supported by the Roth Scholarship through the Department of Mathematics at Imperial College London. We thank the Max Planck Institute for Mathematics in Bonn for their support and ideal working conditions.  We'd like to thank Yank\i \ Lekili for bringing the problem to our attention and discussing the Fukaya category perspective. We're grateful to Michael Wemyss for explaining the NCCR perspective and to Georgios Dimitroglou Rizell who identified an issue with our definition of dg multiplicative preprojective algebra. The anonymous referee caught a few errors and provided useful comments. Finally, special thanks to Sue Sierra for carefully reading a draft and providing detailed corrections and  suggestions.

\section{The multiplicative preprojective algebra} \label{s:MPA}

\subsection{Definitions}
Throughout the paper we fix an arbitrary field $k$. For each quiver (i.e. directed graph) $Q$, let $Q_{0}$ be the vertex set, $Q_{1}$ be the arrow set, and $h, t: Q_{1} \rightarrow Q_{0}$ the head and tail maps, respectively. We will assume that $Q_{0}$ and $Q_{1}$ are finite for convenience, but really only need finitely many arrows incident to each vertex. 

Let $Q^{\text{op}}$ denote the quiver with the same underlying graph of vertices and edges, but with every arrow in the opposite direction. $\overline{Q}$ denotes the quiver with the same vertex set as $Q$ and $Q^{op}$ and with arrow set $Q_{1} \sqcup Q^{op}_{1}$. For each arrow $a \in Q_{1}$, we write $a^{*}$ for the corresponding arrow in $Q_{1}^{op}$, and vice versa. In $\overline{Q}$ we distinguish between arrows in $Q$ and $Q^{op}$ using a function 
$$
\epsilon: \overline{Q}_{1} \rightarrow \{ \pm 1 \} \hspace{1cm} \epsilon(a) := \begin{cases} 
1 & \text{ if } a \in Q_{1} \\
-1 & \text{ if } a \in Q^{\text{op}}_{1}.
\end{cases}
$$

For a quiver $Q$, we denote the path algebra by $kQ$ and follow the convention that paths are concatenated from left to right. We have an inclusion $e_{(-)}: Q_{0} \rightarrow kQ$ in order to view a vertex $i \in Q_{0}$ as a length zero path $e_{i}$. 

For $a \in \overline{Q}_{1}$, define $g_{a} := 1+ a a^{*} \in k \overline{Q}$. Consider the localization $L_{Q} := k \overline{Q}[g_{a}^{-1}]_{a \in \overline{Q}_{1}}$. We write $L:=L_{Q}$, when the quiver is clear from context. 
Notice, for all $a \in \overline{Q}_{1}$:
\begin{equation} \label{identity: g}
g_{a} a = a + a a^{*} a = a g_{a^*}.
\end{equation}
This implies:
$$
g_{a^*} a^* = a^{*} g_{a} \hspace{2cm}
g_{a}^{-1} a = a g_{a^{*}}^{-1} \hspace{2cm}
g_{a^{*}}^{-1} a^{*} = a^{*} g_{a}^{-1}. 
$$
Fixing a total ordering $\leq$ on the set of arrows 
 $\overline{Q}_{1}$, one can make sense of a product over (subsets of) the arrow set. Using $\leq$ and $\epsilon$ we define:
 $$
 \rho := \prod_{a \in \overline{Q}_1} g_{a}^{\epsilon(a)}
 \hspace{2cm} l_{a} := \prod_{b \in \overline{Q}_1, b < a } g_{b}^{\epsilon(b)}
 \hspace{2cm} r_{a} := \prod_{b \in \overline{Q}_1, b>a} g_{b}^{\epsilon(b)}.
 $$
When we need to make the role of the total ordering $\leq$ more explicit, we write $\rho_{\leq}$ (respectively $l_{a, \leq}$ and $r_{a, \leq}$) for $\rho$ (respectively $l_{a}$ and $r_{a}$). By definition, $l_{a}$ and $r_{a}$ are the subproduct of $\rho$ to the left and right of $a$, respectively. Therefore,
\begin{equation} \label{identity: rho}
\rho = l_{a} g_{a}^{\epsilon(a)} r_{a}
\end{equation}
for all $a \in \overline{Q}_{1}$. 

\begin{defn}
Fix a quiver $Q$ and $q \in (k^\times)^{Q_{0}}$. Consider $\overline{Q}$ and fix an ordering $\leq$ on the arrows and a map $\epsilon$ as defined above. The \emph{multiplicative preprojective algebra}, $\Lambda^{q}(Q)$, is defined to be
$$
\Lambda^{q}(Q) := L/ J
$$
where $L = k \overline{Q} [ g_{a}^{-1} ]_{a \in \overline{Q}_{1}}$ is the localization and $J$ is the two-sided ideal generated by the element $\rho-q$. 
\end{defn}

Note that $q$ is viewed as an element of $k \overline{Q}$ via $\sum_{i \in Q_{0}} q_{i} e_{i} \in k Q_{0} \subset k \overline{Q}$, and as $\rho$ is invertible we need $q_{i} \neq 0$ so $e_{i} \Lambda^{q}(Q) \neq 0$, for all $i$. 

\begin{rem} \label{rem: independence}
The isomorphism class of $\Lambda^{q}(Q)$ is independent of both the orientation of the quiver and the choice of an ordering on the arrows, by Section 2 in \cite{Shaw}.
\end{rem}

In the multiplicative preprojective algebra, (\ref{identity: rho}) becomes the identity:
\begin{equation*} \label{identity: q}
l_{a} g_{a}^{\epsilon(a)} r_{a} = q.
\end{equation*}
Hence:
\begin{equation}
 r_{a} l_{a} = q g_{a}^{-\epsilon(a)}.  \label{identity: rl}
\end{equation}

As mentioned in the introduction, we say that a quiver is (extended) Dynkin, we mean that the underlying unoriented graph is an (extended) type ADE Dynkin diagram.  We don't consider non-simply laced types because, given a quiver, the associated Cartan matrix is  $2I-A$ where $A$ is the adjacency matrix of the underlying unoriented graph, which is symmetric.

\begin{exam} \label{exam: Dynkin} (Dynkin Case) \\
Let $Q$ be a Dynkin quiver and let $R$ be a commutative ring. Note that the definitions of (multiplicative) preprojective algebra make sense over $R$. In \cite[Section 5]{Kaplan21}, the first named author constructed explicit isomorphisms
\[
\Lambda^{1}(Q) \cong \Pi^{0}(Q) := \left. R \overline{Q} \middle/  \left ( \sum_{a \in Q_{1}} [a, a^*] \right ) \right.
\]
if $2$, $3$, and $5$ are invertible in $R$ (see also the earlier work \cite[Lemma 5.2.1]{Shaw_thesis}, \cite[Corollary 1]{CB2}, \cite[Theorem 13]{Lekili}, and \cite[Section 5]{LU21}). 
In particular, for a field $k$ of characteristic zero, we can work over $k[\![\hbar]\!]$ and set $q=e^{\hbar}$. Then $\Lambda^q$ is a formal deformation of $\Lambda^1$. Hence, by \cite[Proposition 5.0.2]{EE}, there exists some $\lambda \in k[\![\hbar]\!]$ such that there is a $k[\![\hbar]\!]$-linear algebra isomorphism
$$
\Lambda^{q}(Q) \cong \Pi^{\lambda}(Q) := \left. k\overline{Q} \middle/ \left (\sum_{a \in Q_{1}} [a, a^*]- \sum_{i \in Q_{0}} \lambda_{i} e_{i} \right ) \right. .
$$ 

In $A$ types, such an isomorphism holds  over any field $k$ and for any actual parameter $q$.
Namely, identifying $(A_{n})_{0} = \{ 1, 2, \dots, n \}$ and $(A_{n})_{1} = \{ a_{1}, a_{2}, \dots, a_{n-1} \}$ with tail $t(a_{i})=i$, the isomorphism is given by
$$
\Lambda^{q}(A_{n}) \cong \Pi^{\lambda}(A_{n}) \hspace{1cm} e_{i} \mapsto e_{i}, \ a_{i}  \mapsto a_{i} , \ a_{i}^*  \mapsto \left ( \prod_{j>i} q_{j} \right ) a_{i}^*
$$ 
where $\lambda_{i} := (q_{i}-1) \prod_{j >i} q_{j}$. Since $\Pi^{\lambda}(A_{n})$ is non-zero if and only if there exists $i, j$ with $i<j$ and $\sum_{\ell=i}^{j} \lambda_{\ell} = 0$, it follows that $\Lambda^{q}(A_{n})$ is non-zero if and only if $\prod_{\ell=i}^{j} q_{\ell} = 1$. 

\end{exam}


\subsection{The map \texorpdfstring{$\theta$}{theta}}

In \cite{Shaw}, the important map $\theta: \overline{Q}_{1} \rightarrow \Lambda^{q}(Q)$ is defined by $\theta(a) = q^{-1} l_{a} a r_{a^{*}}$ and extended to $k \overline{Q}$ by the identity on $Q_{0}$ and by requiring $\theta$ to be an algebra map. Then Lemma 3.3 in \cite{Shaw} shows that
\begin{align}  
\theta(g_{a}) &=  l_{a} g_{a} l_{a}^{-1}    \label{identity: Phi of g}   \\
&= r_{a}^{-1} g_{a} r_{a},  \label{identity: Phi of g 2}
\end{align}
so $\theta(g_a)$ is invertible. Hence $\theta$ factors through the localization $L:= k \overline{Q}[g_{a}^{-1}]_{a \in \overline{Q}_{1}}$. We will show $\theta$ descends to the quotient $\Lambda^{q}(Q)$, with the ordering of the arrows reversed, using the following result.

\begin{lem} \label{lem: Phi sends r to l}
Let $\leq$ denote a total order on $\overline{Q}_{1}$ and let $\geq$ denote its opposite ordering, i.e. $a \geq b$ if $b \leq a$. Such an order fixes a bijection $\overline{Q}_{1} \cong \{ a_{1}, a_{2}, \dots, a_{\mid \overline{Q}_{1} \mid} \}$. Then 
$$
\theta(r_{a_{j}, \geq}) = l_{a_{j}, \leq} =: l_{a_{j}} \hspace{1cm} \text{and} \hspace{1cm} \theta(l_{a_{j}, \geq}) = r_{a_{j}, \leq} =: r_{a_{j}}
$$ for any $a_{j} \in \overline{Q}_{1}$.
\end{lem}

\begin{proof}
We prove $\theta(r_{a_{j}, \geq}) = l_{a_{j}}$  by induction on $j$, where $j=1$ is the identity $\theta(1) =1$. Then,
\begin{align*}
\theta( r_{a_{j+1}, \geq} )  & \ =  \ \theta(g_{a_{j}}^{\epsilon(a_{j})}) \theta( r_{a_{j}, \geq} )  \\
&\overset{\text{(IH)}}{=} \theta(g_{a_{j}}^{\epsilon(a_{j})}) l_{a_{j}} \\
&\overset{(\ref{identity: Phi of g})}{=} l_{a_{j}} g_{a_{j}}^{\epsilon(a_{j})} l_{a_{j}}^{-1} l_{a_{j}}   \\
& \ = \ l_{a_{j}} g_{a_{j}}^{\epsilon(a_{j})} \\
& \ = \ l_{a_{j+1}}.
\end{align*}
The second identity is similar and one can formally obtain a proof from the above by exchanging the symbols $r$ and $l$, the identity \ref{identity: Phi of g} for \ref{identity: Phi of g 2}, and the order of the multiplication. 
\end{proof}

\begin{cor} \label{cor: Phi preserves rho}
$\theta( \rho_{\geq} ) = \rho$.
\end{cor}

This corollary implies $\theta$ descends to a map $\Lambda^{q}(Q, \geq) \rightarrow \Lambda^{q}(Q, \leq)$. Notice that we can similarly define $\theta_{\geq}: \Lambda^{q}(Q, \leq) \rightarrow \Lambda^{q}(Q, \geq)$. 

\begin{prop} \label{prop: Phi is an auto}
$\theta_{\geq} \circ \theta = \text{Id}_{\Lambda^{q}(Q, \leq)}$.
\end{prop}

\begin{proof}
It suffices to check $\theta_{\geq} \circ \theta$ is the identity on arrows in $Q_{1}$ and indeed:
\begin{align*}
\theta_{\geq}(\theta(a)) &= \theta_{\geq}( q^{-1} l_{a} a r_{a^{*}} ) \\
&= q^{-1} \theta_{\geq}(l_{a}) \theta_{\geq}(a) \theta_{\geq}( r_{a^{*}}) \\
(\ref{lem: Phi sends r to l}) &= q^{-1}  r_{a} \theta_{\geq}(a) l_{a^{*}} \\
&= q^{-1} r_{a} (q^{-1} l_{a} a r_{a^{*}}) l_{a^{*}} \\
(\ref{identity: rl})&= g_{a}^{-\epsilon(a)} a g_{a^{*}}^{-\epsilon(a^*)} \\
(\ref{identity: g}) &= a g_{a^{*}}^{-\epsilon(a)} g_{a^{*}}^{-\epsilon(a^*)} \\
&= a
\end{align*}
\end{proof}


\section{Calabi--Yau and free product properties} \label{s: 2CY-property}
The goal of this section is to prove that $\Lambda^{q}(Q)$ is 2-Calabi--Yau for $Q$ containing an unoriented cycle. We do so by exhibiting a length two, projective, $\Lambda^{q}(Q)$-bimodule resolution $P_{\bullet}$ of $\Lambda^{q}(Q)$, whose bimodule dual complex $P_{\bullet}^{\vee}$ is quasi-isomorphic to $\Lambda^{q}(Q)$. This resolution is due to Crawley-Boevey and Shaw, but they don't state nor prove that it is exact. The main new ingredient we provide is the injectivity of $\alpha$ which relies on the weak free product property. Hence we begin this section with a short digression explaining the strong and weak free product properties.

\subsection{Free product (complete intersection) properties} \label{section: definition of sffp}
Recall that, if $R$ is a commutative ring and $r \in R$ an element, then there is a dg analogue of the quotient $R/(r)$: the Koszul complex $(R[s]/(s^2), d)$ with $d|_R=0$ and $ds=r$, here $|s|=-1$. Note here that, in spite of the notation,  $R[s]/(s^2)$ is the graded-commutative algebra freely generated by $R$ and a single generator $s$ in degree $-1$.  The quotient map $(R[s]/(s^2), d) \to R/(r)$ is a quasi-isomorphism if and only if $r$ is a non-zerodivisor. 

Thus, in the commutative setting, the non-zerodivisor condition is the correct one for which  the Koszul complex (derived imposition of $r=0$) is equivalent to the quotient algebra. 

Now let us pass to the non-commutative setting. If $A$ is an algebra over a ring $S$
and $J=(r)$ an ideal generated by a single relation $r$, we can form a canonical algebra map,
\begin{equation}
\Phi: A/J *_S S[t] \to \gr_J A, \quad \Phi|_{A/J} = \text{Id}, \Phi(t)=r,
\end{equation}
where $\gr_J$ means the associated graded algebra with respect to the $J$-adic filtration.
\begin{defn} The pair $(A, r)$ satisfies the \emph{weak free product property} if $\Phi$ is an isomorphism.
\end{defn}
\begin{rem}
This condition is significantly more subtle than in the commutative case. In particular, it is insufficient for $r$ to be a non-zerodivisor. For example, if $A=k[x]$ and $r=x^2$, then we have $H^1(A * k[s], d) \ni [xs-sx] \neq 0$.  (Here $A$ is actually commutative, but we take the non-commutative construction; for a non-commutative example, simply replace $A$ with $k\langle x,y \rangle$.) 
\end{rem}
 The weak free product property is an analogue of a \emph{non-commutative complete intersection (NCCI)} \cite{Etingof-Ginzburg}, and closely matches the \emph{weak summand} property from \cite{Anick} (considered in the graded setting).   We have chosen this terminology to make the algebraic property we are using more evident.
 
 When the context is clear, we will sometimes abuse notion and say the quotient $A/J$, for $J=(r)$, satisfies the weak free product property, even though the choice of $A$ and $r \in A$ is important. 
 
Given an $S$-bimodule section $\sigma: A/J \to A$ of the quotient map $\pi: A \rightarrow A/J$, we can form an associated linear map,
\begin{equation}
\widetilde{\sigma}: A/J *_S S[t] \to A, \quad (a_0 t^{m_1} a_1 t^{m_2} \cdots t^{m_n} a_n) \mapsto  \sigma(a_0) r^{m_1} \sigma(a_1) r^{m_2} \cdots r^{m_n} \sigma(a_n),
\end{equation}
for $m_{i} > 0$, for all $i$. The existence of such a $\sigma$ (and hence $\tilde{\sigma}$) is automatic if $S$ is separable, as is the case when $S = kQ_0$ below.

By construction, this is $S[t]$-bilinear, where $t$ acts on $A$ by multiplication by $r$. It also reduces to the identity modulo $(t)$ on the source and $J$ on the target. If $(A,r)$ satisfies the weak free product property, then moreover the completion
\begin{equation}
\widehat{\widetilde{\sigma}}: \widehat{A/J *_S S[t]} \to \widehat{A},
    \end{equation}
    with respect to the $t$-adic and $J$-adic filtrations, is a linear isomorphism.

The goal of the strong free product property is to find a description of $A$ itself as a free product. The first version is the following:


\begin{defn} The triple
$(A,r,\sigma)$ satisfies the \emph{strong free product property} if $\widetilde{\sigma}: A/J *_S S[t] \to A$ is an $S$-bimodule isomorphism. 
\end{defn}

\begin{rem} \label{r: choice of sigma}
The choice of $\sigma$ is important. Let $A = k \langle x, y \rangle$ and $J=(y)$ so $A/J \cong k[x].$ Here $k[t]$ acts on $A$ via $tf:=yf$. Consider two different choices 
$$
\sigma_{1}, \sigma_{2}: k[x] \rightarrow k \langle x, y \rangle \hspace{1cm}
\sigma_{1}(x + (y)) = x,  \ \ \sigma_{2}(x + (y)) = x-xy.
$$ 
Then $\widetilde{\sigma}_{1}$ is a linear isomorphism, while $\widetilde{\sigma}_{2}$ is not surjective as 
$$x = \sigma_{2}(x + (y)) (1-y)^{-1} =  \sigma_{2}(x + (y)) \sum_{i \geq 0} y^{i} \notin \widetilde{\sigma}_{2}(k[x] *_{k} k[t]).
$$
\end{rem}


This property is too much to expect in many situations, such as in the presence of rational functions in $t$. To fix this, let $B=S[t,f^{-1}]$ be a localization of $S[t]$ obtained by inverting 
some 
$f \in S^\times + (t)$, such that the map $S[t] \to A$ extends to an algebra map $\tau: B \to A$ (such an extension is necessarily unique). Let $\overline{B} := tB$, so that we have an $S$-bimodule decomposition $B = S \oplus \overline{B}$.
Then $\widetilde{\sigma}$ extends to a map 
$\sigma': A/J *_S B \to A$, which has the form
\begin{equation} \label{eq: sigma' definition}
    a_0 b_1 a_1 \cdots b_n a_n \mapsto \sigma(a_0) \tau(b_1) \cdots \tau(b_n) \sigma(a_n), a_i \in A/J, b_i \in \overline{B}.
    \end{equation}
  
    \begin{defn} \label{def: sfpp}
The quadruple $(A,r,\sigma,B)$ satisfies the \emph{strong free product property} if $\sigma'$ is a linear isomorphism. 
\end{defn}
This definition reduces to the previous definition in the case $B = S[t]$, $\tau(t) = r$.

In this case, it follows by taking associated graded algebras that $(A,r)$ satisfies the weak free product property. Moreover, $A$ is Hausdorff in the $J$-adic filtration (because the source of $\sigma'$ is Hausdorff in the $t$-adic filtration), and $\sigma'$ is indeed a restriction of $\widehat{\widetilde{\sigma}}$.

\begin{rem}
It is important in the definition of $\sigma'$ to use the natural bimodule complement $\overline{B}=tB$.  Here is an example to show why (see also Remark \ref{r:overline-b-matters} for another one, which we naively ran into before realizing our mistake).  Let $A=k\langle x,y,z \rangle/(xyz-xz), r = y$, so that $A/J \cong k\langle x,z\rangle/(xz)$. A basis for $A/J$ is given by $\{z^i x^j\}_{i,j \geq 0}$. Let $\sigma: A/J \to A$ be the section preserving this, i.e.,  $\sigma(z^i x^j +(y,xyz-xz))=z^i x^j+(xyz-xz)$. Set $B:=k[t]$. Then $\sigma'=\widetilde{\sigma}:A/J *_k k[t] \to A$ is a linear isomorphism, so $(A,r,\sigma)$ is a strong free product.  However, if we were to instead choose a complement $\overline{B} = (t-1)B$, then we now have $\sigma'(x(t-1)z) = 0$,  so the map $\sigma': A/J *_k B \to A$ defined using $\overline{B}$ is not injective. (It is also not surjective, as $xz$ is not in the image.) On the other hand, for general weak free products, using the correct choice $\overline{B}=tB$, $\sigma'$ is always injective.

    \end{rem}  

Now we return to the setup of $Q$ a connected, non-Dynkin quiver and $q= (k^\times)^{Q_0}$. Let $B := kQ_0[t,(q+t)^{-1}]$ and $\overline{B} := t B = 
\text{Span}(t^m,(t')^m \mid m \geq 1\}$, for $t' := (q+t)^{-1}-q^{-1}$.
We conjecture that, for every such $Q$, there exists $\sigma$ such that the quadruple $(L_Q,r,\sigma,B)$ satisfies the strong free product property. Moreover, in Section \ref{s: free-product} we prove this conjecture in the case of quivers containing a cycle:

\begin{thm} [Proposition \ref{p:mpa-general-q} in Section \ref{s: free-product}] \label{thm:sfpp}
Let $Q$ be a connected quiver containing an unoriented cycle. Let $B = kQ_{0}[t, (q+t)^{-1}]$ and let $r$ denote the multiplicative preprojective relation. There exists a section $\sigma$ such that $(L_{Q}, r, \sigma, B)$ satisfies the strong free product property. 
\end{thm}

The proof of this theorem is technical and uses combinatorial algebraic techniques. Therefore we delay its proof until Section \ref{s: free-product}, which does not result in circular logic as that section does not depend on results after Section \ref{s:MPA}.



\begin{rem} \label{rem: weakening assumptions on quiver}
The connectedness assumption can be weakened as follows. If $Q = Q' \sqcup Q''$ then $L_{Q} = L_{Q'} \oplus L_{Q''}$ and $\Lambda^{q}(Q) = \Lambda^{q}(Q') \oplus \Lambda^{q}(Q'')$ and so by adding sections, the strong free product property for 
$$(L_{Q'} \oplus L_{Q''}, r' + r'', \sigma' +\sigma'', B' \oplus B'') = (L_{Q}, r, \sigma, B)$$
 follows from the strong free product properties for $(L_{Q'}, r', \sigma', B')$ and $(L_{Q''}, r'', \sigma'', B'')$. So one only needs the weaker assumption that $Q$ is a quiver with \emph{each} component containing an unoriented cycle. But we state results in the connected setting to simplify the hypotheses. 
\end{rem}

\begin{cor} \label{cor: weak free product property}
Let $Q$ be a connected quiver containing an unoriented cycle. Then $\Lambda^{q}(Q)$ satisfies the weak free product property. In particular, there exists an isomorphism of graded algebras:
$$
\sum_{i} \varphi_{i} : \gr( \Lambda^{q}(Q) *_{kQ_{0}} kQ_{0}[t]) \rightarrow \gr(L_{Q})
$$
where the associated graded algebras are taken with respect to the $t$-adic and $J_{Q}$-adic filtrations on $\Lambda^{q}(Q) *_{kQ_{0}} kQ_{0}[t]$ and $L_{Q}$ respectively. 
\end{cor}

\begin{rem} \label{r: Anick paper}
Note that for ordinary preprojective algebras, the free product
property was observed in \cite[Propositions 5.1.9,5.2.1]{Schedler_HH}.
In fact, as these algebras are non-negatively
graded with finite-dimensional subspaces in each degree, and one-dimensional in degree zero (connected), 
the strong and weak free product properties are
equivalent (and independent of the choice of graded section $\sigma$), as was already observed by Anick in \cite{Anick}.\footnote{Anick works in the graded context over a field rather than $kQ_{0}$, but his results generalize to this setting, see \cite{Etingof-Ginzburg, Schedler_HH}.} 
Moreover, if $A$ has global dimension at most two, then these conditions are also equivalent to
the condition that $A/(r)$ also has global dimension at most two.
In the case $A$ is a tensor or path algebra, such algebras were called
\emph{non-commutative complete intersections} in \cite{Etingof-Ginzburg} due to their close relationship to the condition that representation varieties be complete intersections. For a nuanced discussion of this relationship, including sufficient conditions for the representation variety to be a complete intersection, see the introduction and Theorem 24 in \cite{Berest}. 

However, in the ungraded case, we only have the implication that we need, that the free product property implies the existence of a length-two
projective bimodule resolution. Indeed, 
the latter property only depends on a piece of the associated graded algebra with
respect to the $(r)$-adic filtration, and in the ungraded case this filtration need not even be Hausdorff. In contrast, the strong free
product property implies the Hausdorff condition and gives information about the algebra itself.

Motivated by this, we believe that the strong free product property can be viewed as an ungraded analogue of the non-commutative complete
intersection property. It is an interesting question to investigate when their representation varieties are complete intersections.
\end{rem}

\subsection{A bimodule resolution of \texorpdfstring{$\Lambda$}{Lambda}}
\label{ss:bimodule-resolution}
In this subsection, we show that for \emph{any} quiver, the weak free product property for $\Lambda^{q}(Q)$ implies $\Lambda^{q}(Q)$ has a length two projective bimodule resolution. Consequently, since we establish the weak free product property for connected quivers containing a cycle, we prove $\Lambda^{q}(Q)$ has Hochschild dimension two for connected quivers containing a cycle. For ease of notation, write $\Lambda := \Lambda^{q}(Q)$.

Crawley-Boevey and Shaw build a chain complex of $\Lambda$-bimodules $P_{\bullet} = P_{2} \overset{\alpha}{\rightarrow} P_{1} \overset{\beta}{\rightarrow} P_{0}$ where,
$$
P_{2} = P_{0} := \Lambda \otimes_{kQ_{0}} kQ_{0} \otimes_{kQ_{0}} \Lambda = \langle \eta_{v} \rangle_{v \in Q_{0}} \hspace{1cm}
P_{1} := \Lambda \otimes_{kQ_{0}} k \overline{Q}_{1} \otimes_{kQ_{0}} \Lambda = \langle \eta_{a} \rangle_{a \in \overline{Q}_{1}} 
$$
and
$$
\alpha(\eta_{v}) := \sum_{a \in \overline{Q}_{1} : t(a) = v} l_{a} \Delta_{a} r_{a}
\hspace{.5cm} \text{where} \hspace{.5cm}
\Delta_{a} = 
\begin{cases}
\eta_{a} a^{*} + a \eta_{a^{*}} & \text{if } a \in Q_{1} \\
-g_{a}^{-1}( \eta_{a} a^{*} + a \eta_{a^{*}}) g_{a}^{-1} & \text{if } a 
\in Q_{1}^{op}. 
\end{cases}
$$
$$
\beta(\eta_{a}) := a \eta_{t(a)} - \eta_{h(a)} a.
$$

We claim that it is a resolution of $\Lambda$. To see this, following \cite{Shaw}, we first write down an explicit chain map of $\Lambda$-bimodule complexes $\psi: P_{\bullet} \rightarrow Q_{\bullet}$, where $Q_{\bullet}$ is quasi-isomorphic to $\Lambda$; we then prove it is an isomorphism. $Q_{\bullet}$ is the cotangent exact sequence in Corollary 2.11 in \cite{Cuntz-Quillen}, but in this context it was defined earlier (and shown quasi-isomorphic to $\Lambda$) by Schofield in \cite{Schofield}. So we have the maps:
$$
\xymatrix{
P_{\bullet} \ar[rr]^{\psi}_{\text{\cite{Shaw}}} && Q_{\bullet} \ar[rr]^{\text{quasi-iso}}_{\text{\cite{Schofield}}} && \Lambda.
}
$$

\begin{prop}[Lemma 3.1 in \cite{Shaw}] \label{prop: P is a resolution}
For \emph{any} quiver $Q$, the following diagram commutes:
$$
\xymatrix{
P_{0} \ar[rr]^{\alpha}   \ar[d]^{\psi_{2}} && P_{1} \ar[rr]^{\beta}  \ar[d]^{\psi_{1}}_{\cong} && P_{0} \ar[rr]^{\gamma}  \ar[d]^{\psi_{0}}_{\cong} &&  \Lambda \ar[d]^{id}_{=}  \\
J/J^{2} \ar[rr]^-{\kappa} && \Lambda \otimes_{L} \Omega_{kQ_{0}}(\Lambda) \otimes_{L} \Lambda \ar[rr]^-{\lambda} && \Lambda \otimes_{kQ_{0}} \Lambda \ar[rr]^{\mu} && \Lambda 
}
$$
where the vertical maps are $\Lambda$-bimodule maps defined on generators by,
\begin{align*}
\psi_{2}(\eta_{v}) &:= \rho e_{v} - q e_{v} &\psi_{1}(\eta_{a}) := 1 \otimes_{L} [a \otimes_{kQ_{0}} 1 - 1 \otimes_{kQ_{0}} a] \otimes_{L} 1 && \psi_{0}(\eta_{v}) = e_{v} \otimes e_{v}.
\end{align*}
Here $\alpha$ and $\beta$ are as defined above and $\gamma(\eta_v) := e_v$. The horizontal maps are defined by,
\begin{align*}
&\kappa(x + J^{2}) := 1 \otimes_{L} \delta(x) \otimes 1 \ \ \ \text{ where } \delta(x) := x \otimes 1 - 1 \otimes x \ \  \text{ for } x \in J  \\
 &\lambda(1 \otimes_{L} [ab \otimes_{kQ_{0}} c - a \otimes_{kQ_{0}} bc] \otimes_{L} 1) :=  ab \otimes_{kQ_{0}} c - a \otimes_{kQ_{0}} bc  \ \ \ \text{ for } a, b, c \in L \\
 &\mu( a \otimes b ) := ab \ \ \ \text{ for } a, b \in \Lambda.
\end{align*}
\end{prop}

Since $\psi_{0}$ and $\psi_{1}$ are $\Lambda$-bimodule isomorphisms, it remains to show $\psi_{2}$ is a $\Lambda$-bimodule isomorphism. We show this using the weak free product property.




\begin{prop} \label{prop: wfpp implies P resolution}
Suppose $\Lambda$ satisfies the weak free product property. Then $P_{\bullet}$ is a bimodule resolution of $\Lambda$. 
\end{prop}

\begin{proof}
Taking the $i=1$ piece of the graded isomorphism
$$
\gr(\varphi) = \sum_{i} \varphi_{i} : \gr(\Lambda *_{kQ_{0}} kQ_{0}[t, (t+q)^{-1}]) \longrightarrow \gr(L_{Q})
$$
gives an isomorphism of $\Lambda$-bimodules,
$$
\varphi_{1} : \Lambda \otimes_{kQ_{0}} kQ_{0} \cdot t \otimes_{kQ_{0}}
\Lambda \rightarrow J_{Q}/J_{Q}^{2}.
$$
Since $\varphi_{1}$ sends $t \mapsto r$, it sends $t e_{v} \mapsto r e_{v} = (\rho-q) e_{v}$ and hence $\varphi_{1} = \psi_{2}$. We conclude that $\psi_{2}$ is an isomorphism of $\Lambda$-bimodules and hence $\psi_{\bullet}: P_{\bullet} \rightarrow Q_{\bullet}$ is an isomorphism of $\Lambda$-bimodule complexes. In particular, $P_{\bullet}$ is a resolution since $Q_{\bullet}$ is a resolution.  
\end{proof}

For a complex $C_{\bullet}$ concentrated in non-negative degrees, define the length by 
\[
\text{len}(C_{\bullet}):= \sup \{ i \in \bN \mid C_{i} \neq 0 \}.
\]
For an algebra $A$, the Hochschild dimension of $A$ is $\text{HH.dim}(A) := \text{len}(\text{HH}_{\bullet}(A))$ and the global dimension of $A$, is 
$\text{gl.dim}(A) := \sup_{M \in A\text{-mod}} \inf_{P_{\bullet}} \{ \text{len}(P_{\bullet}) \}$ where the infimum is taken over all projective $A$-module resolutions of M.


\begin{cor} \label{cor: global dim 2}
Let $Q$ be a connected quiver containing a cycle. Then
\[
\text{gl.dim}(\Lambda) \leq \text{HH.dim}(\Lambda) = 2.
\]
\end{cor}

\begin{proof}
Use $P_{\bullet}$ to compute $\text{HH}_{i}(\Lambda)$; $\text{HH}_{i}(\Lambda) =0$ for $i>2$ while $\text{HH}_2(\Lambda) \neq 0$. Therefore $\text{HH.dim}(\Lambda)=2$. Every left $\Lambda$-module, $M$, has a length two projective left $\Lambda$-module resolution $P_{\bullet} \otimes_{\Lambda} M$,  and hence $\Lambda$ has global dimension at most two. 
\end{proof}

\begin{exam} \label{ex: quantum weyl algebra}
Note that the inequality in Corollary \ref{cor: global dim 2} may be strict. If $Q$ is the Jordan quiver (i.e the quiver with one vertex and one loop) then 
$$
\Lambda^{q}(Q) \cong k \langle a, a^* \rangle [ (1+a^*a)^{-1}] /(aa^*-qa^*a-(q-1)).
$$
The change of variables $x :=a$ and $y := a^*/(q-1)$ when $q \neq 1$, identifies $\Lambda^{q}(Q)$ with a localization of the first quantum Weyl algebra, $k \langle x, y \rangle/(xy-qyx-1)$, which has global dimension one. 
\end{exam}

\subsection{The dual complex} \label{ss: dual complex}
In this subsection, we show for \emph{any} quiver that if $P_{\bullet}$ is a resolution of $\Lambda^{q}(Q)$, then $\Lambda^{q}(Q)$ is 2-Calabi--Yau. Combining this with the previous subsection, we get that if $\Lambda^{q}(Q)$ satisfies the weak free product property then $\Lambda^{q}(Q)$ is 2-Calabi--Yau. In particular, this shows that $\Lambda^{q}(Q)$ is 2-Calabi--Yau for connected quivers containing a cycle. 

First we recall the notion of $d$-Calabi--Yau algebras \cite{Ginzburg}.
\begin{defn} $A$ is $d$-Calabi--Yau if: (a) $A$ has finite projective dimension as an $A$-bimodule; (b) $\Ext^i(A, A \otimes A) = 0$ for $i \neq d$; and (c) there exists
an $A$-bimodule isomorphism 
$$
\eta: \text{Ext}^{d}_{A-\text{bimod}}(A, A \otimes A) \rightarrow A.
$$ 
The map $\eta$ is called a $d$-Calabi--Yau structure.
\end{defn}

\begin{rem}
For perfect $A$-modules, $M$ and $N$, one has a quasi-isomorphism,
$$
\text{RHom}_{A-\text{bimod}}(M, N) \overset{\cong}{\longrightarrow}  \Hom_{A-\text{bimod}}(M, A \otimes A) \otimes^{\mathbb{L}}_{A \otimes A^{op}} N.
$$ 
Taking $M=A^{\vee}$ and $N=A$ gives $\text{RHom}_{A-\text{bimod}}(A^{\vee}, A) \cong A \otimes^{\mathbb{L}}_{A \otimes A^{op}} A.$
The isomorphism on the level of $d$th homology realizes
\[
 \eta \in \Hom_{A-\text{bimod}}(A^{\vee}, A[-d]) =: \Ext^{-d}_{A-\text{bimod}}(A^{\vee}, A) \cong \text{HH}_d(A)
 \]
as a class in $d$th Hochschild homology. 

For dg-algebras, one further equips this structure with a class in negative cyclic homology that lifts the Hochschild homology class of the isomorphism.  But, as shown in Proposition 5.7 and explained in Definition 5.9 of \cite{VdB-TV}, for ordinary algebras this additional structure exists uniquely.
\end{rem}

We have established $P_{\bullet}$ as a $\Lambda$-bimodule resolution of $\Lambda$, if $Q$ is connected and contains a cycle. To show $\Lambda$ is 2-Calabi--Yau, it suffices to show that its dual complex
\begin{align*}
\RHom_{\Lambda-\text{bimod}}( \Lambda, \Lambda \otimes \Lambda)
& :=
\Hom_{\Lambda-\text{bimod}}( P_{\bullet}, \Lambda \otimes \Lambda)  =: P^{\vee}_{\bullet} 
\end{align*}
is quasi-isomorphic to $\Lambda[-2]$.

Define $\eta_{v}^{\vee} \in \Hom_{\Lambda-\text{bimod}}(P_{0}, \Lambda \otimes \Lambda)$ and $\eta_{a}^{\vee} \in \Hom_{\Lambda-\text{bimod}}(P_{1}, \Lambda \otimes \Lambda)$ by,
$$
\eta_{v}^{\vee}(\eta_{w})  := \begin{cases}
e_{v} \otimes e_{v} & \text{ if } v = w \\
0 & \text{ otherwise } 
\end{cases} \hspace{1cm}
\eta_{a}^{\vee}(\eta_{b})  := \begin{cases}
e_{t(a)} \otimes e_{h(a)} & \text{ if } b= a^* \\
0 & \text{ otherwise }. 
\end{cases}
$$
These are generators of $P_{0}^{\vee}$ and $P_{1}^{\vee}$ respectively and give isomorphisms,
$$
P_{0}^{\vee} \cong \Lambda \otimes_{kQ_{0}} kQ_{0} \otimes_{kQ_{0}} \Lambda = \langle \eta_{v}^{\vee} \rangle, \hspace{1cm}
P_{1}^{\vee} \cong \Lambda \otimes_{kQ_{0}} k \overline{Q}_{1} \otimes_{kQ_{0}} \Lambda = \langle \eta_{a}^{\vee} \rangle. 
$$

Rather than directly study the dual complex $P^{\vee}_{\bullet}$, we modify the formulas for $\alpha^{\vee}$ and $\beta^{\vee}$ using the map $\theta$, in a way that doesn't affect the homology of the complex. Namely, after choosing generators $\{ \xi_{v} \}$ for $P_{0}^{\vee}$ and $\{ \xi_{a} \}$ for $P_{1}^{\vee}$, defined below, one can expand:
$$
\alpha^{\vee}(\xi_{a}) = \sum_{v \in Q_{0}} a'_{v} \xi_{v} a''_{v}
\hspace{1cm}
\beta^{\vee}(\xi_{v}) = \sum_{a \in \overline{Q}_{1} } b'_{a} \xi_{a} b''_{a},
$$
for some $a'_{v}, a''_{v}, b'_{a}, b''_{a} \in \Lambda$ and then define
$$
\alpha^{\vee}_{\theta}(\xi_{a}) := \sum_{v \in Q_{0}} \theta(a'_{v}) \xi_{v} \theta(a''_{v})
\hspace{1cm}
\beta^{\vee}_{\theta}(\xi_{v}) := \sum_{a \in \overline{Q}_{1} } \theta(b'_{a}) \xi_{a} \theta(b''_{a}).
$$
It suffices to show that
$$
\xymatrix{
(P_{\bullet}^{\vee})_{\theta} &:=&  P_{0}^{\vee} \ar[rr]^{-\beta^{\vee}_{\theta}} && P_{1}^{\vee} \ar[rr]^{ \alpha^{\vee}_{\theta}} && P_{0}^{\vee}   \\
}
$$
is quasi-isomorphic to $\Lambda[-2]$.\\
\\
We prove this by establishing an isomorphism of $\Lambda$-bimodule complexes
$\varphi_{\bullet} : P_{\bullet}[2] \rightarrow (P^{\vee}_{\bullet})_{\theta}$
following Crawley-Boevey and Shaw, so
$$
\xymatrix{
(P_{\bullet}^{\vee})_{\theta} \ar[rr]^{\varphi_{\bullet}^{-1}} && P_{\bullet}[2] \ar[rr]^{\psi_{\bullet}[2]} && Q_{\bullet}[2] \ar[rr]^{\text{quasi-iso}} & &\Lambda[2].
}
$$

\begin{thm} \label{thm: commuting diagram dual}
The following diagram commutes:
$$
\xymatrix{
P_{0}^{\vee} \ar[rr]^{- \beta^{\vee}_{\theta}} 
\ar@{}[drr] | {\textbf{(II)}} 
&& P_{1}^{\vee} \ar[rr]^{\alpha^{\vee}_{\theta}} 
\ar@{}[drr] | {\textbf{(I)}}
&& P_{0}^{\vee} \ar[rr]^{\gamma \circ \varphi_{0}^{-1}} 
&& \Lambda  \\
P_{0} \ar[rr]^{\alpha}  \ar[u]^{\varphi_{0}}_{\cong}  
&& P_{1} \ar[rr]^{\beta} \ar[u]^{\varphi_{1}}_{\cong} 
&& P_{0} \ar[rr]^{\gamma} \ar[u]^{\varphi_{0}}_{\cong} 
&& \Lambda \ar[u]^{\text{id}}_{=} 
}
$$
where the vertical maps are $\Lambda$-bimodule isomorphisms defined on generators by,
\begin{align*}
\varphi_{0}(\eta_{v}) &:= \xi_{v} :=  q \eta_{v}^{\vee}  &\varphi_{1}(\eta_{a}) := \xi_{a^*} := 
\begin{cases}
l_{a} \eta_{a^*}^{\vee} l_{a^{*}}^{-1} & \text{if } a \in Q_{1}^{\text{op}} \\
-r_{a^*}^{-1} \eta_{a^*}^{\vee} r_{a} & \text{if } a \in Q_{1}.
\end{cases} 
\end{align*}

\end{thm}

Note that $\varphi_{1}$ is an invertible map since $r_a$ and $l_a$ are invertible elements of $\Lambda$ for all $a \in \overline{Q}_{1}$. The commuting of {\bf(I)} becomes clear once we compute the maps $\alpha^{\vee}$, the content of the next lemma.

\begin{lem}[Lemma 3.2 in \cite{Shaw}] \label{lem: alpha vee}
\begin{align*}
 & \alpha^{\vee}(\eta_{a}^{\vee}) = \begin{cases}
a^* r_{a} \eta_{h(a)}^{\vee} l_{a} - g_{a^*}^{-1}r_{a^*} \eta_{t(a)}^{\vee} l_{a^*} g_{a^*}^{-1} a^* & \text{if } a \in Q_{1} \\
r_{a^*} \eta^{\vee}_{t(a)} l_{a^*} a^* - a^* g_{a}^{-1} r_{a} \eta_{h(a)}^{\vee} l_{a} g_a^{-1} & \text{if } a \in Q_{1}^{op}
\end{cases}\\
 & \alpha^{\vee}(\xi_{a}) = \theta(a^*) \xi_{t(a^*)} - \xi_{h(a^*)} \theta(a^*) \\
 & \alpha^{\vee}_{\theta}(\xi_{a^*}) = a \xi_{t(a)} - \xi_{h(a)} a \\
 & \beta = \varphi_{0}^{-1} \circ \alpha^{\vee}_{\theta} \circ \varphi_{1} 
\end{align*} 
So square \small{\textbf{(I)}} in Theorem \ref{thm: commuting diagram dual} commutes.
\end{lem}

\begin{proof}
The first two equalities are shown directly in \cite{Shaw} and the last two are clear from the definitions together with Proposition \ref{prop: Phi is an auto}.
\end{proof}

\begin{proof}[Proof of Theorem \ref{thm: commuting diagram dual}]
By Lemma \ref{lem: alpha vee}, it suffices to show that \small{\textbf{(II)}} commutes. While one can similarly compute $\beta^{\vee}$ directly, such a calculation is unnecessary as the commuting of \small{\textbf{(II)}} follows from that of \small{\textbf{(I)}}.

Indeed, dualizing and applying $(-)_{\theta}$ to the maps in \small{\textbf{(I)}}, produces a still commuting diagram:
$$
\xymatrix{
 P_{1} \ar[d]_{(\varphi_{1})^{\vee}_{\theta}} 
\ar@{}[drr] | {\text{\small{\textbf{(I)}}}_{\theta}^{\vee}}
&& P_{0} \ar[ll]_{(\alpha_{\theta}^{\vee})_{\theta}^{\vee}} 
\ar[d]^{(\varphi_{0})^{\vee}_{\theta}}
\ar@{}[drr] | {\text{=}}
&&
 P_{1} \ar[d]_{-\varphi_{1}} 
\ar@{}[drr] | {\text{\small{\textbf{(I)}}}_{\theta}^{\vee}}
&& P_{0} \ar[ll]_{\alpha} 
\ar[d]^{\varphi_{0}}  \\
P_{1}^{\vee}  
&& P_{0}^{\vee}  \ar[ll]^{\beta^{\vee}_{\theta}} 
&& P_{1}^{\vee}  
&& P_{0}^{\vee}  \ar[ll]^{\beta^{\vee}_{\theta}} 
}
$$
which shows $\varphi_{1} \circ \alpha = - \beta_{\theta}^{\vee} \circ \varphi_{0}$, i.e. \small{\textbf{(II)}} commutes. 

The equality of maps $(\alpha_{\theta}^{\vee})_{\theta}^{\vee} = \alpha$ follows from Proposition \ref{prop: Phi is an auto}, and $(\varphi_{0})^{\vee}_{\theta}= \varphi_{0}^{\vee} = \varphi_{0}$ follows from the definitions. For $(\varphi_{1})^{\vee}_{\theta} = -\varphi_{1}$, observe that it suffices to show $(\varphi_{1})_{\theta} = -(\varphi_{1})^{\vee}$ and indeed, 
\begin{align*}
(\varphi_{1})_{\theta}(\eta_{a^*}) = (\xi_{a})_{\theta}
&=\begin{cases}
\theta(l_{a^*}) \eta_{a}^{\vee} \theta(l_{a}^{-1}) & \text{ if } \epsilon(a) = 1 \\
-\theta(r_{a^*}^{-1}) \eta_{a}^{\vee} \theta(r_{a}) & \text{ if } \epsilon(a) = 1 \\
\end{cases} \\
& =
\begin{cases}
r_{a^*} \eta_{a}^{\vee} r_{a}^{-1}  & \text{if } \epsilon(a) = 1 \\
-l_{a^*}^{-1} \eta_{a}^{\vee} l_{a} & \text{if } \epsilon(a) = -1
\end{cases} \\
&= -(\varphi_{1})^{\vee}(\eta_{a^*}).\qedhere
\end{align*}
\end{proof}

So without conditions on the quiver, we have established:
\begin{cor}
If $P_{\bullet} \rightarrow \Lambda$ is exact then $(P_{\bullet}^{\vee})_{\theta} \rightarrow \Lambda[-2]$ is exact and $P^{\vee}_{\bullet} \rightarrow \Lambda[-2]$ is exact. 
\end{cor}
Therefore, the 2-Calabi--Yau property for $\Lambda$ follows from the a priori weaker Hochschild dimension two property. In the previous subsection, we showed that $\Lambda$ has Hochschild dimension two for $Q$ connected and containing a cycle. 

\begin{cor} \label{cor: 2-CY}
If $Q$ is connected and contains a cycle then $\Lambda^{q}(Q)$ is 2-Calabi--Yau.
\end{cor}

\section{Formality of dg multiplicative preprojective algebras} \label{s:formal}


%

In this section we show that if $Q$ satisfies the strong free product property, then the dg multiplicative preprojective algebra is formal. In particular this proves Conjecture \ref{conj:formal} in the case $Q$ is connected and contains a cycle. Moreover, it reduces Conjecture \ref{conj:formal} to the remaining extended Dynkin cases and Conjecture \ref{conj: sfpp}. 

If one views the dg multiplicative preprojective algebra as the central object of study, as in \cite{Lekili} and \cite{Lekili2}, then we are showing one can formally replace it by the non-dg version. \\ 

We begin with an elementary lemma. It is not strictly required, but it demonstrates more transparently the construction we will use.
\begin{lem}
Let $K$ be a commutative ring. 
Let $A$ be the dg-algebra defined as a graded algebra to be $K[r] * K[s]$ with $|r|=0$ and $|s|=-1$,
product given by concatenation of words, and differential extended as a derivation from the generators $d(s) =r$ and $d(r) =0$. Then $A$ is quasi-isomorphic to its cohomology $H^{*}(A) = K$ concentrated in degree zero. In fact, the identity map is chain homotopic to the augmentation map $A \to K$.
\end{lem}

\begin{proof}
Let $h: A \rightarrow A[-1]$ be the homotopy with the property $h(rf) = sf$ and $h(sf) = 0$ for all $f \in A$, and $h(K)=0$. Then $h \circ d + d \circ h - 1_{A}$ is the projection with kernel $K$ to the augmentation ideal of $A$.  Therefore, it defines a contracting homotopy from $A$ to $K$.
\end{proof}
In other words, the lemma is observing that $A$, as the tensor algebra on an acyclic complex $Kr \oplus Ks$, is itself quasi-isomorphic to $K$.


\begin{lem} \label{lem: homology of A}
The dg-algebra $A$ given by 
$$
\Lambda^{q}(Q) *_{kQ_{0}} kQ_{0}[r, (r+q)^{-1}] *_{kQ_{0}} kQ_{0}[s]
\hspace{.5cm} \text{with } |r | =0 \text{ and  } | s | = -1 
$$
and with differential determined by $d(s)=r$ is quasi-isomorphic to $\Lambda^{q}(Q)$ concentrated in degree zero. 
\end{lem}
\begin{proof}
Extending the preceding construction, define a homotopy $h: A \rightarrow A[-1]$ by:
$$
h(frg)=fsg, \quad h(fsg) = 0, \quad h(f(r+q)^{-1} g) = q^{-1}h(fg) - q^{-1} fs(r+q)^{-1} g
$$
for $f \in \Lambda^{q}(Q)$ and $g \in A$. The definition of $h(f(r+q)^{-1}g)$ is chosen to match the formula for $h(frg)$ in the $r$-adic completion.
There is an augmentation $A \twoheadrightarrow \Lambda^q(Q)$ with kernel $(r, s, r':=(r+q)^{-1} - q^{-1})$. 
Notice that $h \circ d + d \circ h$ is a homotopy from the identity on $A$ to the augmentation $\Lambda^{q}(Q)$, as it annihilates $\Lambda^q(Q)$ and is the identity on $s$, $r$, and $r'$.
\end{proof}


\begin{defn} \label{def: dg_mult_preproj}
The \emph{dg multiplicative preprojective algebra} is a dg-algebra over $kQ_{0}$ defined as a graded algebra by
$$
 \Lambda^{dg, q}(Q) :=  
 L_Q
 *_{kQ_{0}} kQ_{0}[s]  \hspace{1cm} |s| = -1, \ \ |\alpha| = 0 \text{ for } \alpha \in L_{Q}. 
$$
The differential, $d$, is defined by $d(s) = \rho -q$, $d(L_{Q}) \equiv 0$, and extended as a $kQ_{0}$-linear derivation to $L_{Q}*_{kQ_{0}} kQ_{0}[s]$.
\end{defn}


\begin{prop} \label{prop: formal}
If $\Lambda^{q}(Q)$ satisfies the strong free product property\footnote{meaning $(L_{Q}, r, \sigma, kQ_{0}[t, (t+q)^{-1}])$ satisfies the strong free product property for some choice of $\sigma$} then 
$$
H_{*}(\Lambda^{dg, q}(Q)) = H_{0}(\Lambda^{dg, q}(Q)) \cong \Lambda^{q}(Q)
$$ 
so in particular $\Lambda^{dg, q}(Q)$ is formal.  
\end{prop}
By Theorem \ref{thm:sfpp}, the proposition holds in particular if $Q$ contains a cycle.

\begin{rem} \label{rem:add_preproj_formal}
Note that for the ordinary preprojective algebra, $\Pi(Q)$, the Ginzburg dg-algebra has homology concentrated in degree zero for any non-Dynkin quiver: $\Pi(Q)$ has a length two bimodule resolution (see \cite{MV, BB} for the characteristic zero case, and \cite{EE} in general) which Anick shows is equivalent for \emph{graded} connected algebras in  \cite[Theorems~2.6 \& 2.9]{Anick}, and \cite{Etingof-Ginzburg} observes this extends to the quiver case.
\end{rem}

\begin{proof}
The strong free product property yields an isomorphism of graded vector spaces,
$$
L_Q \cong \Lambda^{q}(Q) *_{kQ_{0}} kQ_{0}[r, (r+q)^{-1}].
$$ 
Hence, as complexes
$$
\Lambda^{dg,q}(Q) \cong   
L_Q 
*_{kQ_{0}} kQ_{0}[s]
\cong \Lambda^{q}(Q) *_{kQ_{0}} kQ_{0}[r, (r+q)^{-1}] *_{kQ_{0}} kQ_{0}[s],
$$
which by Lemma \ref{lem: homology of A} is quasi-isomorphic to $\Lambda^{q}(Q)$, concentrated in degree zero. It follows that,
$$
\Lambda^{dg, q}(Q) \cong H_{*}(\Lambda^{dg, q}(Q)) \cong H_{0}(\Lambda^{dg, q}(Q)) \cong \Lambda^{q}(Q)
$$ 
as dg-algebras.
\end{proof}

\begin{rem}
In the presence of Conjecture \ref{conj: 2CY}, formality of $\Lambda^{dg,q}(Q)$ implies $\Lambda^{dg,q}(Q)$ is 2-Calabi--Yau. Hence by Theorem \ref{thm: 2CY}, we have shown that $\Lambda^{dg,q}(Q)$ is 2-Calabi--Yau, when $Q$ is connected and contains a cycle. One may be able to adapt the techniques in Section \ref{s: 2CY-property} to prove that $\Lambda^{dq,q}(Q)$ is 2-Calabi--Yau, in general. In more detail, writing $\Lambda^{dg} := \Lambda^{dg, q}(Q)$, the role of the $\Lambda^{q}(Q)$-bimodule resolution, $P_{\bullet}$, should now be played by the  $\Lambda^{dg}$-dg-bimodule given by the total complex of:
$$
\xymatrix{
\Lambda^{dg} \otimes_{kQ_{0}} kQ_{0} \cdot s \otimes_{kQ_{0}} \Lambda^{dg} \ar[r]^-{\alpha^{dg}} \ar@/_1pc/[rr]_-{\beta_1^{dg}} &
\Lambda^{dg} \otimes_{kQ_{0}} kQ_{1} \otimes_{kQ_{0}} \Lambda^{dg} \ar[r]^-{\beta_0^{dg}} &
\Lambda^{dg} \otimes_{kQ_{0}} \Lambda^{dg},
}
$$
where $\beta^{dg}_1(a \otimes s \otimes b) = as \otimes b - a \otimes sb$ and $\alpha^{dg}$ (respectively $\beta^{dg}_0$) has the same formula as $\alpha$ (respectively $\beta$). 
\end{rem}
\begin{rem}
We are grateful to Georgios Dimitroglou Rizell, who pointed out that our definition differs from that arising in symplectic geometry. Indeed in the derived multiplicative preprojective algebra, $\mathscr{L}_{\Gamma}$, in \cite[page 779]{Lekili2} Etg\"{u}--Lekili define additional variables $z_a, \zeta_a$ with $z_a$ invertible and $d(\zeta_a) = z_a -(1+a^*a)$, and hence $(1+a^*a)$ is invertible only after taking homology. In contrast, we invert $(1+a^*a)$ on the chain level in $\Lambda^{dg, q}(Q)$. However, for our main result, Proposition \ref{prop: formal}, this distinction is irrelevant, as we now explain. 

We claim that the dg algebra map $\alpha: \mathscr{L}_{\Gamma} \to \Lambda^{dg, q}(Q)$, given by $\alpha(z_a) = (1+a^* a)$,  $\alpha(\zeta_a) = 0$, and taking arrows to arrows, is a quasi-isomorphism.  To see this, note that $\mathscr{L}_{\Gamma}$ can be viewed as a bigraded dg algebra with two differentials: set $|\zeta_a| = (-1,0)$ and $|s|=(0,-1)$, with horizontal differential $d_H(\zeta_a)=z_a-(1+a^* a)$, $d_H(s)=0$, and vertical differential $d_V(\zeta_a)=0, d_V(s)=\rho-q$. We will show in the next paragraph that the map $\alpha$ induces an isomorphism on horizontal cohomology, that is,
$\alpha: (\mathscr{L}_{\Gamma}, d_H) \to (\Lambda^{dg, q}(Q), 0)$ is a quasi-isomorphism.
Therefore,  $\alpha$
is a morphism of bicomplexes (placing the target in horizontal degree zero), that induces an isomorphism on the first page of the associated spectral sequences. These sequences collapse on the second page. They collapse to the cohomology, since both sequences were third-quadrant (cohomologically) and hence convergent. This proves the claim.

It remains to show that $\alpha$ is an isomorphism on horizontal cohomology.
   More generally, let $A$ be a graded path algebra on the quiver $Q$ (arrows can be assigned any degrees), and let $S \subseteq A$ be a subset of homogeneous elements; in the case above we have $A:=k\overline{Q} *_{kQ_0} kQ_0[s]$ and $S := \{1+a^* a\}_{a \in \overline{Q}_1}$.  We wish to compare two localizations. The first is the naive one, $A[f^{-1}]_{f \in S}$.  The second is given by replacing $A$ by the quasi-isomorphic algebra 
   $\widetilde{A} := A\langle z_f, \zeta_f \rangle_{f \in S}$, with differential  $d(\zeta_f)=z_f-f, \  d(z_f) = 0, \ d(A) \equiv 0$.
   We then consider $\widetilde{A}[z_f^{-1}]_{f \in S}$. To compare these we use the technique of \emph{derived localization}, following \cite{BCL18}. 
   Since $A$ is hereditary with zero differential, by \cite[Corollary 4.20, Theorem 5.1]{BCL18}, 
  its derived localization by $S$ is $A[S^{-1}]$ (i.e., it is underived).   On the other hand, $\widetilde{A}$ is cofibrant in the category of dg algebras equipped with a morphism from $kQ_0\langle z_f\rangle_{f \in S}$, as it is given by cell attachment (although with nonzero differential).  So $\widetilde A[z_f^{-1}] = \widetilde{A} *_{kQ_0\langle z_f \rangle} kQ_0 \langle z_f, z_f^{-1} \rangle$ is also its derived localization.  Now the quasi-isomorphism $\widetilde A \to A$ is compatible with the morphisms from the path algebra $kQ_0\langle z_f \rangle_{f \in S}$, sending $z_f$ to $z_f$ and to $f$, respectively.  Thus the map $\widetilde{A}[z_f^{-1}]_{f \in S} \to A[f^{-1}]_{f \in S}$ is a quasi-isomorphism of derived localizations of $A$ at $S$. 
  
  Note that combining the two preceding paragraphs, in general, the quasi-isomorphism $A \langle z_f, z_f^{-1}, \zeta_f \rangle_{f \in S} \to A[f^{-1}]_{f \in S}$ induces a quasi-isomorphism $A \langle z_f, z_f^{-1}, \zeta_f, s_i \rangle \to A[f^{-1}] \langle s_i \rangle$ for any additional arrows $s_i$ and differential $d(s_i)$ compatible with the morphism (only assuming that $A$ is a graded path algebra with $S$ a collection of homogeneous elements). The same is true replacing $A\langle z_f, z_f^{-1}, \zeta_f \rangle$ by any other model of the derived localization of $A$ at $S$.
 \end{rem}

\begin{rem} \label{rem: dg MPA definition}
The dg multiplicative preprojective algebra is called the Legendrian cohomology dg algebra in \cite[3.2]{Lekili} where they establish that it is a multiplicative analog of Ginzburg's dg algebra for a quiver with zero potential defined in \cite[1.4]{Ginzburg}. It is called a capped Chekanov--Eliashberg algebra in \cite[Section 3.2]{DRET20} where they independently prove formality in the case $Q$ is the Jordan quiver and $q=1$ in \cite[Theorem 3.13]{DRET20}.
\end{rem}

\section{Local structure of multiplicative quiver varieties and moduli spaces attached to 2-Calabi--Yau algebras} \label{s:quiver_varieties}
In this section, we will assume that $k$ is an algebraically closed field of characteristic zero.

We will use our main result to prove, as anticipated in \cite[Section 7.5]{Schedler}, that multiplicative quiver varieties are \'etale-locally (or formally locally) isomorphic to ordinary quiver varieties. Our proof uses (a generalization of) a result of Bocklandt, Galluzzi, and Vaccarino in \cite{Bocklandt} for 2-Calabi--Yau algebras.  While our main result is only proved for quivers \emph{with} cycles, we are able to prove this result for \emph{all} quivers. The key idea is to embed any quiver into one containing a new vertex with a cycle, and put the zero vector space at this new vertex. This identifies every multiplicative quiver variety with one for a quiver containing a cycle.

We recall the definition of multiplicative quiver varieties \cite{Shaw, Yamakawa, Schedler}, beginning with King's notion of (semi)stability.  First, by an algebra over $kQ_0$, we mean a $k$-algebra which contains $kQ_0$ as a subalgebra.  Given a module $M$ over such an algebra $A$, its dimension vector is $\alpha \in \bN^{Q_0}$ given by $\alpha_i = \dim e_i M, i \in Q_0$. Given a $kQ_0$-module $V$, let $\Rep(A,V) := \Hom_{kQ_0-\text{alg}}(A, \End_k(V))$ be the set of $A$-module structures on $V$. Let $\Rep_{\alpha}(A) := \Rep(A,V)$ for $V := \bigoplus_{i \in Q_0}  k^{\alpha_i}$, called the representation space of dimension $\alpha$.
\begin{defn}\cite{King}  Let $Q$ be a finite quiver.
Let $A$ be an algebra over $kQ_0$, $\theta \in \bZ^{Q_0}$ a parameter, $\alpha \in \bN^{Q_0}$ a dimension vector.
Assume that $\theta \cdot \alpha = 0$. Then an $A$-module $M$ of dimension vector $\alpha$
is said to be $\theta$-\emph{semistable} if, for every submodule $N :=  \{ N_i \}_{i \in Q_{0}}$, with dimension
vector $\beta \in \bN^{Q_0}$, we have $\beta \cdot \theta \leq 0$.  
Furthermore $M$ is $\theta$-\emph{stable} if $\beta \cdot \theta < 0$ for all non-zero, proper submodules. 
Let
$Rep^{\theta\text{-ss}}(A,V)\subseteq \Rep(A,V)$ be the subset  of $\theta$-semistable module structures, and denote this by $\Rep^{{\theta\text{-ss}}}_\alpha(A)$ when $V := \bigoplus_{i \in Q_0} k^{\alpha_i}$. \end{defn}
\begin{defn}\cite{King}
Let $Q$, $\alpha$, and $\theta$ be as in the definition above, and let $A$ be an algebra over $kQ_0$.  Then the corresponding (semistable) moduli space is
\begin{equation} \mathcal{M}_\theta(A,\alpha) := \Rep^{\theta\text{-ss}}_\alpha(A) /\!/ \GL(\alpha).
\end{equation}
\end{defn}
In the case $A=\Lambda^q$, this is called a \emph{multiplicative quiver variety}, denoted $\mathcal{M}_{q,\theta}(Q,\alpha)$. In the case $A=\Pi^\lambda$ is a (deformed) preprojective algebra, it is called an ordinary quiver variety, denoted $\mathcal{M}_{\lambda,\theta}^{\text{add}}(Q,\alpha)$.



The main results of this section are the following:
\begin{thm}\label{t:fnt-general}
Let $A$ be a 2-Calabi--Yau algebra over $kQ_0$, and let
 $\rho$ be a $\theta$-semistable representation of $A$ of dimension $\alpha$. Then there exists $Q', \alpha'$ such that the formal neighborhood of $\rho$ of the moduli space $\mathcal{M}_\theta(A,\alpha)$
 is isomorphic to the formal neighborhood
of $\mathcal{M}^{\text{add}}_{0,0}(Q',\alpha')$ at the zero representation.
\end{thm}
\begin{thm}\label{t:fnt-mult}
At every point of a multiplicative quiver variety, a formal neighborhood is isomorphic to the formal neighborhood of zero of an ordinary quiver variety.
\end{thm}
In the case where the quiver contains an oriented cycle, Theorem \ref{t:fnt-mult} follows immediately from Theorem \ref{t:fnt-general} and our main result; in general, we need to enlarge the quiver: see Section \ref{ss:fnt-proof}.
Note that the corresponding result 
for formal neighborhoods of ordinary quiver varieties is known, see \cite[Corollary 3.4]{Bellamy-Schedler-quiver}.  

By Artin's approximation theorem \cite[Corollary 2.6]{Artin}, we can replace ``formal neighborhoods" in the preceding theorems by \'etale neighborhoods, since we are in the setting of varieties (by which we always mean of finite type) over a field.  

\begin{cor}\label{c:symp-sing}
Let $A$ be a 2-Calabi--Yau algebra over $kQ_0$. Then, all moduli spaces $\mathcal{M}_\theta(A,\alpha)$ are symplectic singularities. In particular, they are normal and have rational Gorenstein singularities. The same holds for all multiplicative quiver varieties.
\end{cor}

The proofs of these results are given in the final subsection. 

\subsection{Generalities on completions of $2$-Calabi--Yau algebras}
To prove Theorem \ref{t:fnt-general}, we will need the following results about the local structure of $n$-Calabi--Yau algebras at modules $M$, adapted from \cite{VdB-CalabiYau}.

\begin{defn}
Let $A$ and $B$ be $A_{\infty}$-algebras. $B$ is \emph{minimal} if $m^{B}_{1} = 0$. $B$ is further a \emph{minimal model} for $A$ if 
there exists an $A_{\infty}$-quasi-isomorphism $B \rightarrow A$ lifting the identity.  We make the same definitions for $L_\infty$ algebras.
\end{defn}
In particular, if $B$ is a minimal model for $A$ then $(B,m^B_2) \cong H^*(A)$ as graded algebras. Kadeishvili showed that every $A_{\infty}$-algebra has a minimal model:

\begin{thm}[Minimal Model Theorem \cite{Kadeishvili}] \label{thm:minimal model} 
Let $A$
be an augmented $A_{\infty}$-algebra over a semisimple $k$-algebra $S$. Then, $A$ admits an augmented 
$A_{\infty}$-algebra isomorphism $H^*(A)' \oplus C \to A$, where $H^*(A)'$ is an $A_{\infty}$-algebra which, as a dg algebra, is the cohomology $H^*(A)$ with zero differential, and $C$ is a contractible complex such that all $\geq 2$-ary multiplications involving it are zero. 

Similarly, if $\mathfrak{g}$ is an $L_\infty$-algebra over $k$, then there is an $L_\infty$-isomorphism $H^*(\mathfrak{g})' \oplus \mathfrak{c} \to \fg$, where the underlying ordinary dg Lie algebra of $H^*(\mathfrak{g})'$ is the cohomology of $\fg$ with zero differential, and $\mathfrak{c}$ is a contractible complex such that all $\geq 2$-ary  multiplications involving it are zero.
\end{thm}
Here, an $A_\infty$-algebra $A$ is \emph{augmented} over $S$ if it is of the form $S \oplus \overline{A}$ where $S$ is a subalgebra and $\overline{A}$ is a strict ideal, i.e., all multiplications with $\overline{A}$ as an input land in $\overline{A}$; moreover, we assume that the only nonzero multiplication between $S$ and $\overline{A}$ are the binary operations (i.e. the $S$-bimodule structure).  An \emph{augmented $A_\infty$-morphism} is an $A_\infty$-morphism which is the identity on $S$, preserves strictly the augmentation ideals, and all higher $A_\infty$-structure maps vanish when one of the inputs is in $S$.  

\begin{rem}
The map $A \to \overline{A}$ gives an equivalence between augmented $A_\infty$-algebras and non-unital $A_\infty$-algebras in the category of $S$-bimodules. This makes the statements for $A_\infty$ and $L_\infty$- algebras more symmetric. There are also $L_\infty$ analogues of working over a semisimple algebra: for example, we may work with representations of a reductive group. Given an augmented $A_\infty$-algebra over a matrix algebra, the augmentation ideal has an associated $L_\infty$-algebra which is a representation of the general linear group. 
\end{rem}

Kadeishvili's approach is direct and explicit: he constructs both the $A_{\infty}$-structure on $H^*(A)'$ and the $A_{\infty}$-algebra isomorphism $A \to H^*(A)' \oplus C$. For more conceptual treatments, see e.g.~Theorem 5.4 of \cite{Kajiura-nha} and Remark 4.18 in \cite{CL-Feynman}. For a sketch in the context of $L_{\infty}$-algebras see Lemma 4.9 of \cite{Kont-form}. 

\begin{rem}
Note that the Minimal Model Theorem is usually stated in the literature for fields, but it is known that the statement and proof generalizes to the case of semisimple algebras over a field.
\end{rem}

\begin{defn}
Let $A$ be an $A_{\infty}$-algebra. We say $A$ is \emph{formal} if there is an augmented $A_{\infty}$-isomorphism $H^*(A)' \to H^*(A)$, where $H^*(A)$ has zero $\ell$-ary multiplication for $\ell \geq 3$.
\end{defn}

\begin{defn}
Given a dg associative algebra $A$ with module $M$, define the \emph{derived Koszul dual algebra with respect to $M$} to be $E_M(A) := \REnd_A(M)$.
\end{defn}
This is only defined up to quasi-isomorphism, but it will not matter to us which model is chosen.
 Note that if $A$ is a Koszul algebra over $S$, with $S$ the augmentation module,
 then up to degree conventions,
 $E_S(A)$ is the completion of the Koszul dual algebra, $A^!$, with respect to the filtration by powers of the augmentation ideal. In this case, $A$ and $A^!$ have an additional weight grading, and $(A^!)^! \cong A$. \\
 


 Recall that, if $A$ is an $n$-Calabi--Yau algebra and $M$ a finite-dimensional module, then there is a trace $\lambda: \Ext^n(M,M) \to k$ such that the composition 
 \[ (-,-) :\Ext^i(M,M) \times \Ext^{n-i}(M,M) \overset{\circ}{\longrightarrow} \Ext^n(M,M) \overset{\lambda}{\longrightarrow} k \]  
 is a graded symmetric perfect pairing \cite[Lemma 3.4]{Keller-defCYc}. Since it is also graded commutative, this says that $\Ext^{\bullet}(M,M)$ is a \emph{symmetric dg Frobenius} 
 algebra. In the case that $R:=\End_A(M)$ is semisimple, this says that $\Ext^n(M,M) \cong R$ as $R$-modules. Moreover, if we realize $R$ as endomorphisms of a $kQ'$-representation (i.e $R \cong \prod_{i \in Q'_0} \End_k(k^{\alpha'_i})$ for some finite set $Q'_0$ and dimension vector $\alpha' \in \bN^{Q'_0}$) then we can write, for $V_i := k^{\alpha'_i}$,
 \[ \Ext^m(M,M) \cong \bigoplus_{i,j \in Q'_0} \Hom_k(V_i, V_j)^{c_{i,j}^m}, \]
 for some $c_{i,j}^m \in \bN.$
 
 Moreover,  in the case of interest, $n=2$, we only need to consider $m=1$. Then the pairing on $\Ext^1(M,M)$ is symplectic.  By picking an appropriate symplectic basis on $\Ext^1(M,M)$,
 we can write 
\[ \Ext^1(M,M) \cong T^* \left ( \bigoplus_{a \in Q'_1}  \Hom_k(V_{t(a)}, V_{h(a)}) \right ), \] 
with the standard symplectic structure on the cotangent bundle, for some set $Q'_1$ of arrows with vertex set $Q'_0$ (i.e., extending $Q'_0$ to a quiver $Q'=(Q'_0, Q'_1)$).  It turns out that the symplectic pairing on $\Ext^1(M,M)$, and hence the quiver data $(Q',\alpha')$, completely determines the dg algebra $\REnd(M)$ up to $A_{\infty}$-isomorphism. 

Continue to assume that $R:=\End_A(M)$ is semisimple. In this case, the image, call it $S$, of the action homomorphism $\rho_M: A \to \End_k(M)$  is also semisimple.  We could complete $A$ at $M$, meaning the completion with respect to the filtration by powers of $\ker \rho_M$.  This is not necessarily a quasi-isomorphism invariant, however.  A better way to take the completion is 
by a double Koszul duality, as $E_M (E_M A)$, where $M$ is viewed as an $E_M A$-module via the augmentation map $\REnd(M) \to \End_A(M)$.  The result is certainly complete, and in certain cases it is indeed the completion of $A$ (e.g., for $A = k[x]$ with $M=k$, one obtains $k[\![x]\!]$; see the proof of the next theorem for more cases).

Since $S$ is semisimple, it is Morita equivalent to a direct sum of copies of $k$, namely $kQ_0'$ for $Q_0'$ the set of isomorphism classes of indecomposable summands of $M$.  Then, we can replace the aforementioned ``completion'' of $A$  by a completed quiver algebra, by replacing $M$ by $M'$, the direct sum of one copy of each nonisomorphic indecomposable summand of $M$.  Then $E_{M'} E_M A$ is augmented over $kQ_0'$ and is Morita equivalent to the completion of $A$ at $M$. More precisely, if $V_i = k^{\alpha_i'}$ as before, so that $\End_A(M) = \bigoplus_i \End_k(V_i)$, then $M' = \bigoplus_i V_i$, viewed as an $E_M A$ module via the augmentation $E_M A \to \End_A(M)$.


\begin{thm}\label{thm: VdB-CalabiYau}
  Let $A$ be a 2-Calabi--Yau algebra over $kQ_0$ and $M$ a finite-dimensional module such that $\End_A( M)$ is semisimple.
 Then $E_M A$ is formal.
\end{thm}
\begin{proof} We deduce this result from \cite[Theorem 11.2.1, Corollary 9.3]{VdB-CalabiYau} as follows.
The latter gives a formal local characterization of $n$-Calabi--Yau  algebras (more generally for dg exact Calabi--Yau algebras concentrated in non-positive degrees) for $n\geq 3$. The proof there is valid also in the case $n=2$, where it yields that  the following are equivalent for a \emph{complete augmented} algebra $A$ over $kQ_0$:
\begin{itemize}
\item[(a)] $A$ is a $2$-Calabi--Yau algebra;
\item[(b)] $E_{kQ_0} A$ is formal and has a non-degenerate trace of degree $-2$.
\end{itemize}
In this case, $A$ itself is isomorphic to $E_{kQ_0} \REnd_A(kQ_0)$. 

Now, let $A$ be an \emph{ordinary} 2-Calabi--Yau algebra and $M$ a finite-dimensional module with $\End_A(M)$ semisimple.  Then $E_M A= \REnd_A(M)$ has a non-degenerate trace of degree $-2$.  We can now apply the aforementioned result to the dg algebra $A' := E_{M'} E_M A$, which is complete and augmented over $kQ_0$.  By construction, $E_{kQ_0} A' \cong E_M A$ (formally, this is because $B := E_M A$ is its own double Koszul dual, as it is augmented, finite-dimensional, and concentrated in positive degrees (\cite[Proposition A.5.4]{VdB-CalabiYau}).  Thus $E_M A$ is formal. \qedhere


\end{proof}
\begin{rem}
In fact, the proof shows that the following statements are equivalent for an ordinary algebra $A$ and module $M$ with $\End_A(M)$ semisimple:
\begin{itemize}
\item[(a)] $E_M A$ is formal and has a non-degenerate trace of degree $-2$;
\item[(a')] $E_M A$ has a non-degenerate trace of degree $-2$;
\item[(b)] The double dual $E_M E_M A$ is $2$-Calabi--Yau.
\end{itemize}
Since the double Koszul dual is Morita equivalent to the completed dg quiver algebra $E_{M'} E_M A$, these statements are also equivalent to this latter algebra being $2$-Calabi--Yau.
\end{rem}

\begin{rem} As stated, \cite{VdB-CalabiYau} actually deals with the case of Calabi--Yau dimension $n \geq 3$. In this case, one can also state a version of the theorem: instead of yielding that $E_M A$ is formal, one can 
only kill the higher $A_\infty$-structures of $\Ext^\bullet(M,M)$ which land in top degree $n$.
The main result of \emph{op.~cit.}~can then be stated as saying that the remaining structure of $E_M A$ is governed by a single cyclically symmetric element called the superpotential.  
\end{rem}
\begin{rem} Since submission of this article, Ben Davison has proved a more general formality result for 2-Calabi--Yau categories \cite[Theorem 1.2]{Davison-CY2}.  As he explains, the reason for the formality is quite simple: the Koszul dual $E_M A$ can be taken to be a cyclic $A_\infty$-algebra which is augmented over $\End_A(M)$. This means that, for $s \in \End_A(M)$, $\langle m_n(a_1,\ldots,a_n), s \rangle = \langle a_1, m_n(a_2,\ldots,a_n,s) \rangle = 0$.  This shows that all $A_\infty$-structures landing in top degree (here, degree two) vanish.
\end{rem}



Theorem \ref{thm: VdB-CalabiYau} implies that the formal moduli problem, based at $M$, of modules over a 2-Calabi--Yau algebra $A$ is equivalent to that of a dg preprojective algebra.  
Indeed, using the bar construction, one can realize $E_{M'} \Ext^{\bullet}(M,M)$ as the completed dg preprojective algebra of the quiver $Q'$, for $M'$ as above.
Note that the module $M'$
is a zero representation of this preprojective algebra: all arrows act by zero.

\subsection{The representation and moduli schemes}
We are interested rather in the ordinary representation moduli scheme of $A$, possibly using a nonzero stability condition. In this case, Theorem \ref{thm: VdB-CalabiYau} will imply that, when $A$ is $2$-Calabi--Yau the formal neighborhood
of this scheme at  $M$  will be isomorphic to that of the corresponding quiver variety. 

To prove this, we use the following generalization of \cite[Theorem 6.3]{Bocklandt}, describing the general structure of these schemes whenever $A$ is an algebra with $\End_A(M)$ semisimple. 
Given  (formal) 
schemes $X$, $Y$ with actions by a 
group $G$,
write $X \times^G Y := (X \times Y)/\!/G$ using the diagonal action. 

\begin{thm}\label{t:raf-rep}
Let 
$A$ be an algebra over $kQ_{0}$, and $\alpha \in \bN^{Q_0}$ a dimension vector. Suppose that $M \in \Rep_\alpha(A)$ is a representation whose $\GL_\alpha$-orbit is closed in some $\GL_\alpha$-stable affine open subset $U$ of $\Rep_\alpha(A)$. Let $A' := H^0 E_{M'}E_M A$
(for $M'$ as above).
Then:
\begin{enumerate}
    \item 
$\End_A( M)$ is semisimple;
\item There is a $\GL_{\alpha}$-equivariant isomorphism 
$\widehat{\Rep_\alpha(A)}_{\GL_{\alpha} \cdot M} \cong  \widehat{\Rep_{\alpha'}(A')}_{M'} \times^{\GL_{\alpha'}} \GL_{\alpha}$;
    \item The formal neighborhood of $[M]$ in $U/\!/\GL_\alpha$ is isomorphic to the formal neighborhood of $[M']$ in $\Rep_{\alpha'}(A')/\!/\GL_{\alpha'}$. 
\end{enumerate}
\end{thm}
Before we begin the proof of the theorem, as in 
\cite[\S 6]{Bocklandt}, we need to recall some of the formalism of Maurer--Cartan loci.  Let $\mathfrak{g}$ be a dg associative or Lie algebra. Then the Maurer--Cartan locus is
\[
\MC(\mathfrak{g}) := \left\{ a \in \mathfrak{g}^1 \  \middle | \ da + \frac{1}{2}[a,a] = 0 \right \}.
\]
Let $\wMC(\mathfrak{g})$ be its formal completion at $0 \in \fg^1$. More generally, given an $A_\infty$ or $L_\infty$-algebra, we can define
\[
\wMC(\mathfrak{g}) := Z \left (a \mapsto da + \frac{1}{2!}[a,a] + \frac{1}{3!}[a,a,a] + \cdots \right ) \subseteq \widehat{\fg^1},
\]
the formal subscheme of $\widehat{\mathfrak{g}^1}$ cut out by the Maurer--Cartan equation (now a power series).  

The algebra of functions on this formal scheme is the zeroth Lie algebra cohomology of $\mathfrak{g}^{>0}$,
$H^0 \CE(\mathfrak{g}^{>0}) = 
H^0 (\CE(\mathfrak{g}) /((\mathfrak{g}^0)^*))$. Here, the Chevalley--Eilenberg cochain complex is the completed dg symmetric algebra, $\CE(\mathfrak{g}) = (\widehat{\text{Sym}}(\mathfrak{g}^*[-1]),d_{\CE})$, 
equipped with the Chevalley--Eilenberg differential. For algebras $\mathfrak{g}$ concentrated in positive degrees, this does not depend on $A_\infty$ or $L_\infty$-quasi-isomorphisms. 

The Maurer--Cartan formal scheme has an infinitesimal action by the Lie algebra $\mathfrak{g}^0$, via gauge equivalence. 
The gauge action of an element  $\xi \in \mathfrak{g}^0$ is recorded by applying the differential and contracting with $\xi$.  
The categorical quotient of the Maurer--Cartan formal scheme by this action is defined, on the level of functions, by passing to $\mathfrak{g}^0$-invariant functions. The algebra of functions here is $H^0\CE(\mathfrak{g}^{\geq 0}) / ((H^0 \mathfrak{g})^*[-1])$. For algebras $\mathfrak{g}$ concentrated in non-negative degrees, 
this quotient does not depend on $A_\infty$ or $L_\infty$-quasi-isomorphisms.

Now let $A$ be a $kQ_0$-algebra and $M$ a module. Consider the non-negatively graded dg associative algebra of $kQ_0$-bilinear Hochschild cochains, 
\[
\mathfrak{g} := \text{HC}_{kQ_0}(A, \End_k(M)) := \bigoplus_{i \geq 0} \Hom_{kQ_0-\text{bimod}}(A^{\otimes_{kQ_0} i}, M),
\]
equipped with the usual differential and cup product structure. 
(We remark that this is well known to be quasi-isomorphic to the usual algebra of $k$-linear Hochschild cochains, since $kQ_0$ is semisimple.)  

Given $a \in 
\mathfrak{g}^1 = \Hom_{kQ_0-\text{bimod}}(A, \End_k(M))$, we can consider the deformation $(\rho_M+a): A \to \End_k(M)$, with $\rho_M$ the original module structure.  The condition for $\rho_M+a$ to be a module structure is the Maurer--Cartan equation, $da+a^2=0$.  Hence $\MC(\mathfrak{g}) = \Rep(A,M)$, with zero corresponding to $M$.  Thus $\wMC(\mathfrak{g}) = \widehat{\Rep(A,M)}_M$.




\begin{proof}[Proof of Theorem \ref{t:raf-rep}]  \hfill \\
First,  to show $\End_A(M)$ is semisimple, we will use  Matsushima's criterion \cite{Matsushima}: 
\begin{quote}
If $G$ is a reductive group acting on an affine variety $X$, then the stabilizer of a point in a closed orbit is reductive.
\end{quote}
 In the case at hand, $G=\GL_\alpha$ is acting on $X=U$, so the stabilizer $G_{M} \cong \Aut(M)$ is reductive. So any element $x \in N(\End_A(M))$, the nilradical of $\End_A(M)$, gives rise to an element $1+x$ in the unipotent radical, which is $\{ 1 \}$ as $\Aut(M)$ is reductive. So $N(\End_A(M))=0$, which implies, as $\End_A(M)$ is finite-dimensional, that the Jacobson radical $J(\End_A(M)) = 0$. We conclude that $\End_A(M)$ is semisimple, being Artinian with vanishing Jacobson radical. \\

To obtain (2), let $\mathfrak{g}$ be the dg algebra of $kQ_0$-bilinear cochains, $\text{HC}_{kQ_0}(A, \End_{k}(M))$ as before the proof.
As we explained,
the completed Maurer--Cartan subscheme $\wMC(\fg)=\wMC(\fg^{>0})$ is the same as for the minimal model
$H^*(\mathfrak{g}^{>0})$ of $\fg^{>0}$ (as these are concentrated in positive degrees). Next, let $\mathfrak{h} := Z^0(\mathfrak{g}) \cong \End_A(M)$, the zero-cycles of $\mathfrak{g}$, which is a reductive Lie subalgebra of $\mathfrak{g}^0$.  Its action integrates to the reductive group $H=\Aut_A(M) \cong \GL_{\alpha'}$, so it acts semisimply. Now, we apply
 Lemma \ref{lem: equivariant-minimal-models} below, to obtain a quasi-isomorphic $L_\infty$-algebra (in fact $A_\infty$-algebra, see Remark \ref{r: a-infinity-minimal-model}) $\mathfrak{g}' := \mathfrak{g}^0 \oplus Z^1(\mathfrak{g}) \oplus H^{>1}(\mathfrak{g})$.

Define $\tilde H^1(\mathfrak{g})$ to be an $H$-invariant complement to the one-coboundaries $B^1(\mathfrak{g})$ in $Z^1(\mathfrak{g})$. 
The $L_{\infty}$-structure maps $\tilde H^1(\mathfrak{g})^m \to H^2(\mathfrak{g})$ in $\mathfrak{g}'$ are the same as the ones on any minimal model $H^*(\mathfrak{g}')$ induced by transfer (as in the proof of Lemma \ref{lem: equivariant-minimal-models} below).
This gives an embedding of the Maurer--Cartan locus $\wMC(H^* \mathfrak{g}) = \wMC(H^{>1} \mathfrak{g})$ of the cohomology into the Maurer--Cartan locus of $\mathfrak{g}'$.  By Lemma \ref{lem: equivariant-minimal-models}, this inclusion is compatible with the $H$-action, which is linear.  It is also a formal slice to the infinitesimal $\mathfrak{g}^0$ action on $\wMC(\mathfrak{g}')$: the tangent space to this action is $B^1(\mathfrak{g})$, whereas the tangent space to $\wMC(\mathfrak{g}')$ is $Z^1(\mathfrak{g})$.

Next let us turn from the formal neighborhood of $M$ in $\Rep_\alpha(A)$ to a formal neighborhood of its $\GL_\alpha$ orbit. Luna's slice theorem \cite{Luna} implies that there is a $(\GL_\alpha)_{M} = \Aut_A(M)=H$-stable affine subset $V \subseteq U$, such that the action map $\phi: \GL_\alpha \times^{H} V \to \Rep_\alpha(A)$ induces a $\GL_\alpha$-equivariant isomorphism onto an \'etale neighborhood of the orbit $\GL_\alpha \cdot M$. 
Using the fact that $\Aut(M)$ is connected, we have the following identifications. For ease of reading let
$\text{FN}(X, Y) := \widehat{Y}_X$ denote the formal neighborhood of $X$ in $Y$. 
\begin{align*}
\text{FN}(\GL_\alpha \cdot M, U) 
& \cong \text{FN}(\GL_\alpha \times^{H} \{M\}, \GL_\alpha \times^{H} V) \\
& \cong 
\GL_\alpha \times^{H} \text{FN}( \{M \}, V). 
\end{align*}

Finally, we showed above that the slice $V$ can be taken to be the Maurer--Cartan locus of $H^{>0}(\mathfrak{g})$. This identifies with $\widehat{\Rep_{\alpha'}(A')}_{M'}$, since the latter is isomorphic to the Maurer--Cartan locus of the minimal model $H^*(\mathfrak{g})$. (Explicitly, since the augmentation ideal of $A'$ acts by zero on $M'$, $\End_{kQ_0}(M')=\End_{A'}(M')$ is the degree zero part of the Hochschild cochain complex of $M'$ with zero differential, so $\Ext^{>0}_{A'}(M',M')$ is quasi-isomorphic to $H^{>0}(\mathfrak{g})$.)  
This completes the proof of (2), as $H$ is identified with $\GL_{\alpha'}$  by definition of $\alpha'$. \\
 
 It remains to deduce (3) from (2). First note that, since $\GL_\alpha$ is reductive and the orbit $\GL_\alpha \cdot M \subseteq U$ is closed, by Hilbert's theorem, the ideal of $[M]$ in $\mathcal{O}(U/\!/\GL_\alpha)=\mathcal{O}(U)^{\GL_\alpha}$ is the set of
$\GL_\alpha$-invariant functions in the ideal of $\GL_\alpha \cdot M$ in $\mathcal{O}(U)$. Therefore, 
functions on $\text{FN}([M],U/\!/\GL_\alpha)$ are the $\GL_\alpha$-invariant functions in the  completion of $\mathcal{O}(U)$ at the fiber $F \subseteq U$ of the projection $U \to U/\!/\GL_\alpha$:
\begin{equation*}
\text{FN}([M], U/\!/\GL_\alpha) \cong \text{FN}(F, U)/\!/\GL_\alpha.
\end{equation*}
Note that  $\GL_\alpha \cdot M \subseteq F$, so we get a further map $\text{FN}(F,U)/\!/\GL_\alpha \to \text{FN}(\GL_\alpha \cdot M, U)/\!/\GL_\alpha$.  We claim that this is an isomorphism. Indeed, let $I_{\GL_\alpha \cdot M} \supseteq I_F$ be the ideals. Then we are considering two different completions of $\mathcal{O}(U/\!/\GL_\alpha)$ concentrated at $[M]$, by the systems $\{I_{\GL_\alpha \cdot M}^{\GL_\alpha}\}$ and $\{I_F^{\GL_\alpha}\}$.
Since $U$ is irreducible, by Krull's intersection theorem, $\bigcap_{n \geq 0} I_{\GL_\alpha \cdot M}^n = 0$.  Hence
 the systems  are both exhaustive. Since
$I_{[M]}^n/I_{[M]}^{n+1}$ is finite-dimensional for all $n$, both systems must yield the $I_{[M]}$-adic completion (equivalently, the completion by all finite-dimensional quotients supported at $[M]$).   We deduce that
\begin{equation}
\label{e:fn-invts-swap}
\text{FN}([M], U/\!/\GL_\alpha) \cong \text{FN}(\GL_{\alpha} \cdot M, U)/\!/\GL_\alpha.
\end{equation}



Applying (2), 
we have
\[
\text{FN}(\GL_\alpha \cdot M, U)/\!/\GL_\alpha 
\cong 
\text{FN}(\GL_{\alpha'} \cdot M', \Rep_{\alpha'}(A'))/\!/\GL_{\alpha'}.
\]
By \eqref{e:fn-invts-swap} applied to the first and last terms, we obtain finally the desired isomorphism.
\qedhere




\qedhere 

\end{proof}
\begin{rem}
Part of the proof is actually showing is that the derived formal moduli stack at $[M]$ of representations of $A$ is identified with the same for the dg algebra $E_{M'} E_M A$ at the zero representation $[M']$. This is true more generally, but under our hypotheses this  implies the stated result by taking a truncation and applying Luna's slice theorem. 
\end{rem}
\begin{rem} The second statement of the theorem is a strengthened version of the statement in \cite{Bocklandt}
that a formal neighborhood of $[M]$ in $\Rep_\alpha(A)$ identifies with that of $[M']$ in $\Rep_{\alpha'}(A')$ times a formal disc of dimension $\dim \GL_\alpha - \dim \GL_{\alpha'}$. This is because $\GL_\alpha$ is smooth, and taking the formal completion at the identity, the product construction here is multiplying by such a formal disc.
\end{rem}
The theorem above uses the following lemma:
\begin{lem} \label{lem: equivariant-minimal-models}
Suppose that $\mathfrak{h} \subseteq Z(\mathfrak{g}^0)$ acts on a dg Lie algebra $\mathfrak{g}$ concentrated in non-negative degrees. Suppose that all $\mathfrak{h}$-subrepresentations have complements 
(e.g., this is true if the $\mathfrak{h}$ action integrates to an action of a connected reductive group $H$ with Lie algebra $\fh$).
Then there is an  $L_\infty$-quasi-isomorphism 
$$
\phi: \fg' := \mathfrak{g}^0 \oplus Z^1(\mathfrak{g}) \oplus H^{>1}(\mathfrak{g}) \to \fg,
$$
where on the source,
all higher brackets
\begin{equation}\label{e:br-vanish}
\mathfrak{h} \times (\fg')^{\geq 2} \to \mathfrak{g}'
    \end{equation}
    vanish.  The bracket $\fg^0 \times \fg^0 \to \fg^0$ is the original one. 
        Moreover, the linear part $\phi^1: \fg' \to \fg$ is $\fh$-linear and induces the identity on $\fg^0\oplus Z^1(\fg)$, as well as on cohomology. Finally, $\phi^{\geq 2}$ vanishes on $\fh \times (\fg')^{\geq 1}$.
\end{lem}
\begin{proof}
We apply the homotopy transfer formulae  from \cite{Markl-transfer} (stated for $A_\infty$-algebras but easily adapted to the $L_\infty$ setting).  
To do this, for each $i$ we pick a decomposition
$\fg^i = B^i(\fg) \oplus \tilde H^i(\fg) \oplus \mathfrak{q}^i$, with $B^i(\fg)$ the $i$-coboundaries, $\tilde H^i(\fg)$ an $\fh$-linear complement to $B^i(\fg)$ in the $i$-cocycles $Z^i(\fg)$, and $\mathfrak{q}^i$ a $\fh$-linear complement to $Z^i(\fg)$ in $\fg^i$. We then define a homotopy $h: \fg^{>1} \to \fg^{> 0}$ via 
the projection $\mathfrak{g}^i \to B^i(\fg)$ followed by a $\fh$-linear isomorphism $B^i(\fg) \to \mathfrak{q}^{i-1}$, for $i > 1$, setting $h|_{\fg^{\leq 1}} = 0$.

The resulting homotopy is $\fh$-linear and has the property that $t:= \Id-(dh+hd)$ is a projection onto the subcomplex $\fg^0 \oplus Z^1(\fg) \oplus \tilde H^{> 1}(\fg)$, which is an $\fh$-subrepresentation.  Call this subcomplex $\fg'$. We have an $\fh$-linear decomposition $\fg = \fg' \oplus \mathfrak{c}$ as complexes, with $\mathfrak{c} = \imm(dh+hd)$ a contractible subcomplex (and $\fh$-subrepresentation).

Now use $h$ on all of $\mathfrak{g}$, as in
 the proof of Theorem \ref{thm:minimal model} (see the references above). 
 We obtain a new $L_\infty$-structure on $\fg$, which is $L_\infty$-isomorphic to the original one (with linear part the identity), so that all structures vanish on $\mathfrak{c}$ aside from the differential.
  The $L_\infty$-structures on $\fg'$ are linear combinations of expressions such as
$$t[a_1,h[a_2,[h[a_3,a_4],h[a_5,a_6]]]],$$
given by iteratively bracketing and applying $h$, except at the end where $t$ is applied.  

By $\fh$-linearity of $h$, if $x \in \fh$ and $a \in \fg' = \imm t$, then $h[x,a] = [x,ha] = 0$.  Similarly, $t[x,ha] = th[x,a]=0$.  Hence, all contributions to higher brackets $\fh \times \fg^{>1} \to \fg$ vanish.  Similarly, $\phi^{>1}$ vanishes on $\fh$ (since $h[x,a] = 0$). By construction $\phi$ is the identity on $\fg^0 \oplus Z^1(\fg)$ and on cohomology.
\qedhere
\end{proof}
\begin{rem}\label{r: a-infinity-minimal-model}
The lemma has an associative analogue with the same proof: let $\fg$
be a dg associative algebra and $\fh$ is a subalgebra for which every $\fh$-subbimodule of $\fg$ admits an $\fh$-complement (e.g., $\fg$ is augmented over a semisimple algebra $\fh$).
Then we obtain the same result with an $A_\infty$-quasi-isomorphism with higher order parts vanishing on $\fh$, and with higher multiplications on $\fg'$ vanishing on $\fh$.  This applies to the situation at hand, so that we could use an $A_\infty$-quasi-isomorphism in the proof of Theorem \ref{t:raf-rep}. However, it makes no difference for the Maurer--Cartan locus. (Actually, this says that the decomposition in Theorem \ref{t:raf-rep} enhances to a decomposition of \emph{non-commutative} representation schemes, meaning it describes representations with coefficients in non-commutative Artinian rings.)
\end{rem}

\subsection{Proof of main results}\label{ss:fnt-proof}
In the case where $A$ is 2-Calabi--Yau, we can use Theorem \ref{thm: VdB-CalabiYau} (which applies because of part one in Theorem \ref{t:raf-rep}) and the discussion following it, to refine part three of Theorem \ref{t:raf-rep}. Namely, we can identify the formal neighborhood of $[M']$ in $\Rep_{\alpha}(A') /\!/ GL_{\alpha'}$ with a formal neighborhood of the zero representation in a quiver variety.

\begin{cor} \label{c: FN_quiver_variety}
Let $A$ be a 2-Calabi--Yau algebra over $kQ_0$ for a quiver $Q$. Let $\alpha \in \bN^{Q_0}$ and let $M \in \Rep_{\alpha}(A)$, such that $\GL_{\alpha} \cdot M$ is closed in some $\GL_{\alpha}$-stable open affine subset, $U$.  Then a formal neighborhood of $[M]$ in $U/\!/\GL_\alpha$ is isomorphic to the formal neighborhood $\widehat{\mathcal{M}^{\text{add}}_{0,0}(Q',\alpha')}_0$ of the zero representation in a quiver variety. 
\end{cor}

Pick a stability parameter $\theta$. If $M \in \Rep_{\alpha}(A)$ is $\theta$-semistable, one has the open set $\Rep_{\alpha}(A)^{\theta\text{-ss}}$, which is a union of $GL_{\alpha}$-stable affine open subsets. As $M$ lies in one such affine open subset, $M$ satisfies the hypotheses of Theorem \ref{t:raf-rep} and Corollary \ref{c: FN_quiver_variety}. This implies the following corollary:

\begin{cor}
Let $Q, \alpha, A$ be as in Theorem \ref{t:raf-rep}, let $\theta \in \bZ^{Q_0}$.
Then for every $M \in \Rep_\alpha(A)^{\theta\text{-ss}}$, the conditions of Theorem \ref{t:raf-rep} are satisfied. So, the formal neighborhood of $[M]$ in $\mathcal{M}_\theta(A,\alpha)$
is isomorphic to that of zero in $\Rep_{\alpha'}(A')/\!/\GL_{\alpha'}$, for $A'$ as in the theorem.
\end{cor}

\begin{proof}[Proof of Theorem \ref{t:fnt-general}]
Let $Q, \alpha, A$ be as in Corollary \ref{c: FN_quiver_variety}, let $\theta \in \bZ^{Q_0}$, and $V := \Rep_\alpha(A)^{\theta\text{-ss}}$. 
For every $M \in V$ the conditions of Corollary \ref{c: FN_quiver_variety} are satisfied. So, the formal neighborhood of $[M]$ in $V /\!/\GL_\alpha$ is isomorphic to the formal neighborhood $\widehat{\mathcal{M}^{\text{add}}_{0,0}(Q',\alpha')}_0$ of the zero representation in a quiver variety.
\end{proof}

\begin{proof}[Proof of Theorem \ref{t:fnt-mult}]
If the quiver $Q$ contains a cycle, then Theorem \ref{t:fnt-mult} follows immediately from 
Theorem \ref{t:fnt-general} since $\Lambda^{q}(Q)$ is 2-Calabi--Yau, by Theorem \ref{thm: 2CY}.  

If $Q$ does not contain a cycle, then build $\widetilde{Q}$ from $Q$ by adding a new vertex $i_0$, an arrow from $i_0$ to itself, and an arrow from $i_0$ to any vertex of $Q$. If $\alpha \in \bN^{Q_{0}}$ is a dimension vector then define $\widetilde{\alpha} \in \bN^{\widetilde{Q}_{0}}$ such that $\widetilde{\alpha}|_{Q_0} = \alpha$ and $\widetilde{\alpha}_{i_0} = 0$. Note that $\Rep_\alpha(\Lambda^q(Q)) = \Rep_{\widetilde \alpha}(\Lambda^{\widetilde q}(\widetilde{Q}))$ where $\widetilde{q}$ is similarly such that $\widetilde{q}|_{Q_0} = q$ and $\widetilde{q}_{i_0} = 1$.

Under this identification, the $\GL_{\widetilde{\alpha}} = \GL_\alpha \times \GL_1$ action factors through the projection to $\GL_\alpha$, which identifies the actions on the two varieties. For every $\theta \in \bZ^{Q_0}$, extending by zero to $\widetilde{\theta}$, one also identifies $\theta$-semistable representations of $\Lambda^q(Q)$ of dimension $\alpha$ with $\widetilde{\theta}$-semistable representations of $\Lambda^{\widetilde q}(\widetilde{Q})$ of dimension $\widetilde{\alpha}$. Therefore, $\mathcal{M}_{q, \theta}(Q, \alpha) = \mathcal{M}_{\widetilde{q}, \widetilde{\theta}}(\widetilde{Q}, \widetilde{\alpha})$, i.e. the semistable moduli spaces in question are identical. So the result follows in general from the specific case where $Q$ contains a cycle.
\end{proof}

\begin{proof}[Proof of Corollary \ref{c:symp-sing}]
By \cite{Beauss}, a (normal) symplectic singularity is rational Gorenstein.  The latter is a formal local property. By \cite[Theorem 1.2]{Bellamy-Schedler-quiver}, ordinary quiver varieties are symplectic singularities. Thus, the moduli spaces in question
have rational Gorenstein singularities, and in particular are normal.

Next, thanks to Namikawa \cite[Theorem 4]{Nam-ext2form}, the property of being a (normal) symplectic singularity is a equivalent to having rational Gorenstein singularities and having a symplectic form on the smooth locus. It remains to check the last property.
(Note that this property is certainly known for many multiplicative quiver varieties: for instance,
Yamakawa \cite[Theorem 3.4]{Yamakawa} showed that the stable locus is smooth symplectic, and this is often the entire smooth locus.  For another example, character varieties of Riemann surfaces of genus $\geq 1$ (and many of genus zero) have symplectic smooth locus by \cite[\S 1.2]{Schedler}.)

To see that the smooth locus is symplectic in general, first we can assume that we are in the situation of a 2-Calabi--Yau algebra $A$ (in the case of multiplicative quiver varieties, the proof of Theorem \ref{t:fnt-mult} in Section \ref{ss:fnt-proof} identifies the moduli space with  one for a 2-Calabi--Yau algebra obtained by enlarging the quiver).
At a smooth point of the moduli space, Theorem \ref{t:fnt-general} endows the formal neighborhood of the point with a symplectic form,  given by
the canonical symplectic pairing $\Ext^1(M,M) \times \Ext^1(M,M) \to \Ext^2(M,M) \overset{\text{tr}}{\cong} k$ coming from the Calabi--Yau structure. 
This is functorial in the point of the moduli space: the Calabi--Yau structure furnishes a fixed $A$-bimodule isomorphism $A \cong \Ext^2(A, A^e)$. This induces a functorial isomorphism  
$$\Ext^2(M,M) \cong H^2(\RHom(A, A^e) \otimes^{\mathbb{L}}_{A^e} \End_k(M)) \to A \otimes_{A^e} \End_k(M) = \frac{\End_k(M)}{[A,\End_k(M)]}.
$$
Composing this with the trace map we obtain the functorial trace pairing.
\end{proof}
\begin{rem}Alternatively, one should be able to construct the symplectic structure on the smooth locus because the latter is
an open substack of 
the symplectic derived moduli stack of representations of $\Lambda^q(Q)$, shown to be symplectic in \cite{Brav}. 
\end{rem}

\section{The multiplicative preprojective algebra of the cycle is an NCCR} \label{s:NCCR}
The purpose of this section is to prove Conjecture \ref{conj: NCCR} in the case where $Q$ is a cycle. We begin with the necessary definitions.  Throughout this section, $Q$ denotes an extended Dynkin quiver (not necessarily a cycle).

According to the conjecture, the center of the multiplicative preprojective algebra is the ring of functions on the multiplicative quiver variety $\mathcal{M}_{1,0}(Q,\delta)$.  Here $\delta$ is the primitive positive imaginary root. In terms of the McKay correspondence, $Q$ is the McKay graph of a finite subgroup $\Gamma < \mathrm{SL}_2(\bC)$, which means that the vertices are labelled by the irreducible representations of $\Gamma$. In these terms, $\delta_v$ is the dimension of the irreducible representation of $\Gamma_Q$ attached to the vertex $v$. In particular, for the cycle with $n$ vertices, $\Gamma=\bZ/n\bZ$, and $\delta=(1,\ldots,1)$ is the all ones vector.


We next recall the notion of an NCCR.
Van den Bergh originally defined these in \cite[Appendix A]{VdB_Flop} to give an alternate proof of Bridgeland's theorem that a flop of three-dimensional smooth varieties induces an equivalence of their bounded derived categories. Van den Bergh later simplified and generalized the definition to the following:

 \begin{defn} \cite[Definition 4.1 and Lemma 4.2]{VdB_NCCR} 
 Let $R$ be an Gorenstein commutative integral domain. 
 An algebra $A$ is an \emph{NCCR} over $R$ if: 
 \begin{itemize}
 \item[(1)] $A$ is (maximal) Cohen--Macaulay,  
 \item[(2)] $A$ has finite global dimension, and
 \item[(3)] $A \cong \End_{R}(M)$ for some  reflexive\footnote{Recall an $R$-module $M$ is \emph{reflexive} if the natural map $M \rightarrow \Hom_{R}( \Hom_{R}(M, R), R)$ sending $m \in M$ to evaluation on $m$ (i.e. $m \mapsto [\varphi \in \Hom_{R}(M, R) \mapsto \varphi(m) \in R$]) is an isomorphism.}  module $M$.
 \end{itemize}
 \end{defn}
Note that, if $A$ is derived equivalent to a \emph{commutative} crepant resolution of $\text{Spec}(R)$, then it will have to satisfy these conditions by \cite[Corollary 4.15]{IyamaWemyss}. (However, in general, $R$ could admit a commutative crepant resolution but not a non-commutative one, and vice-versa). 

In our case, with $\dim R = 2$, it is convenient to observe that we don't have to check the Cohen--Macaulay condition:

\begin{lem}
Let $R$ be a normal Noetherian domain of dimension 2 over $k$. Let $M$ be a finitely-generated, reflexive $R$-module. Then $A := \End_{R}(M)$ is Cohen--Macaulay.
\end{lem}
\begin{proof}
Since $R$ is Noetherian and $M$ is finitely-generated and reflexive, \cite[Lemma 15.23.8]{Stacks} 
implies that $A$ is reflexive. Since $R$ is 2-dimensional and normal \cite[Corollary 3.9]{Burban-Drozd-survey} implies that $A$ is Cohen--Macaulay.
\end{proof}
\begin{rem}
Note that, in higher dimensions, while the Cohen--Macaulay property for $A$ is not automatic, it nevertheless can be deduced  from the Calabi--Yau property thanks to \cite[Theorem 3.2 (3)]{IR}. This gives an alternative way to handle condition (2) in our situation.
\end{rem}

\subsection{Shaw's results on the center}\label{s:shaw}


 
While the center of the multiplicative preprojective algebra is in general unknown, in Shaw's thesis, \cite{Shaw_thesis}, he proves the following. Let $v$ be an extending vertex.
 \begin{thm} [Theorem 4.1.1 \cite{Shaw_thesis}] \label{thm: shaw}
 $e_{v}\Lambda^{1}(Q)e_{v} \cong k[X, Y, Z]/(f(X, Y, Z))$ where $f$ has isolated singularity at the origin. Explicitly, 
 \[
 f(X, Y, Z) = 
 \begin{cases}
 Z^{n+1} + XY + XYZ & \text{if } Q = \widetilde{A_{n}}, \ n\geq 1 \\
 Z^2 - p_{n-4}(X)XZ + p_{n-5}(X) X^2Y - XY^2 - XYZ & \text{if } Q = \widetilde{D_{n}}, \ n \geq 4 \\
 Z^2 +X^2Z + Y^3 - XYZ & \text{if } Q = \widetilde{E_{6}} \\
 Z^2 + Y^3 + X^3Y - XYZ & \text{if } Q = \widetilde{E_{7}} \\
 Z^2-Y^3-X^5+XYZ & \text{if } Q = \widetilde{E_{8}}, 
 \end{cases}
\]
where $p_{-1}(X) := -1, p_{0}(X) := 0$, and $p_{i+1}(X) := X(p_{i-1}(X) + p_{i}(X))$ for $i \geq 1$.
 \end{thm}

 \begin{rem}\label{r:wemyss}
 Shaw expected that the singularities at the origin have the du Val type corresponding to the quiver.
 Over a field of characteristic zero,
  Michael Wemyss checked this in $E$ types via Magma. It is also clear that in $A$ types, the singularity is du Val of the same type as the quiver, by the rational substitution $y \mapsto y/(1+z)$. Presumably it can be checked that in type $D$ (over  characteristic not equal to two) the singularity also is the corresponding du Val one. 
  
  Note that having du Val singularities is 
 equivalent to the statement that the minimal commutative resolution is symplectic,  i.e., 2-Calabi--Yau. Thanks to \cite{StaffordVdB}, it is also true that if a Gorenstein surface admits an NCCR, then it has du Val singularities. This is another reason to believe Shaw's expectation.
 \end{rem}
\begin{rem} Suppose as expected that the singularities are du Val. Then, as in \cite{KVmck}, one may construct an NCCR from the minimal resolution. It seems an interesting question to show that this is Morita equivalent to  $\Lambda^1(Q)$. 
 \end{rem}
 
 This motivates the final statement in Conjecture \ref{conj: NCCR}, that the Satake map,  $Z(\Lambda^1(Q)) \to e_v \Lambda^1(Q) e_v$, given by $z \mapsto e_v z$, is an isomorphism. With this in place, the above translates into an explicit description of the center.

\subsection{Proof of Conjecture \ref{conj: NCCR} for a cycle} \label{s:NCCR_proof}
Fix $n \geq 1$. In the remainder of this section we prove Conjecture \ref{conj: NCCR} for $Q=\widetilde{A_n}$.  As a consequence, using Shaw's result, we conclude:
\begin{cor}\label{c:center}
The center of $\Lambda^1(\widetilde{A_n})$ is isomorphic to $k[X,Y,Z]/(Z^{n+1}+XY+XYZ)$.
\end{cor}
The steps of the proof of Conjecture \ref{conj: NCCR} for $\widetilde{A_n}$ are as follows:
\begin{enumerate}
    \item First we show that $\Lambda^1(\widetilde{A_n})$ is isomorphic to an NCCR over $e_{0}\Lambda^{1}(\widetilde{A_{n}}) e_{0}$;
    \item Then we use the preceding result to establish that the Satake map $Z(\Lambda^{1}(\widetilde{A_{n}})) \rightarrow e_{0}\Lambda^{1}(\widetilde{A_{n}}) e_{0}$, is an isomorphism;
    \item To complete the proof we consider the  canonical map  $Z(\Lambda^1(\widetilde{A_n})) \to k[\mathcal{M}_{0,1}(\widetilde{A_n},\delta)]$, given by associating to a central element and a simple representation the scalar by which the element acts in the representation. We show that this is an isomorphism.
    \end{enumerate}
We carry out these steps in the next subsections.

In the first step, we will make use of
the prime property for $\Lambda^{1}(\widetilde{A_{n}})$. We state the prime property now, but defer the proof until Section \ref{s: free-product}, as our proof uses an explicit basis produced in Proposition \ref{prop: basis for An}. 

\begin{rem} \label{rem:non-circular}
Note that there is no circular logic in the paper, as Section \ref{s: free-product} does not rely on any results after Section \ref{s:MPA}, and hence could instead fit logically between Sections \ref{section: definition of sffp} and \ref{ss:bimodule-resolution},
whereby every result would be proven in order. We decided that, due to the technical nature of   Section \ref{s: free-product}, whose methods are not used in the preceding material, it would be better to use its results as a black box in Sections \ref{ss:bimodule-resolution}--\ref{s:NCCR}. 
\end{rem}

 \begin{defn}
 Let $R$ be a ring. 
 We say $R$ is \emph{prime} if $r R r'=0$ implies $r =0$ or $r'=0$, for all $r, r' \in R$. 
 \end{defn}
 
For a commutative ring, this recovers the usual notion of an integral domain, i.e., that the zero ideal is a prime ideal. 

\begin{exam} \label{exam: prime}
For a non-example, take $B = \oplus_{n \in \bN} B_{n}$ to be a finite-dimensional $\bN$-graded algebra not concentrated in degree zero. Then there exists $N \in \bN$ such that $B_{m} =0$ for all $m>N$ but $B_{N} \neq 0$. Pick $b \in B_{N}$ nonzero and notice that $bBb \in \oplus_{m \geq 2N} B_{m} = \{0 \}$ since $2N>N$. Hence $B$ is not prime. 

In particular for $Q$ Dynkin and $k=\bC$, $\Lambda^{1}(Q) \cong \Pi^{0}(Q)$ is a finite-dimensional $\bN$-graded algebra and therefore not prime. However, for $Q = A_{2}$ and $q= (1/2, 2) \neq (1, 1)$, then 
$\Lambda^{q}(A_{2}) \cong \Pi^{(-1, 1)}(A_{2}) \cong \Mat_{2 \times 2}(k)$ is prime. 
\end{exam}

\begin{prop} [Proposition \autoref{prop: prime}] 
\label{prop:prime_earlier}
$\Lambda^{q}(\widetilde{A_n})$ is prime for all $n \geq 0$ and all $q \in (k^{\times})^{n+1}$.
\end{prop}
 
\subsubsection{The NCCR property}
We first show that the multiplicative preprojective algebra is an NCCR (Step 1).
\begin{prop} \label{prop: NCCR}
$\Lambda^{1}(\widetilde{A_{n}})$ is isomorphic to an NCCR over $e_{0}\Lambda^1(\widetilde{A_{n}}) e_0$.
\end{prop}

\begin{proof}
 Define $\Lambda := \Lambda^{1}(\widetilde{A_{n}})$ for ease of notation. Write the vertex set as $\{ 0, 1, \dots, n \}$ and the arrow set $\{ a_{0}, a_{0}^{*}, a_{1}, a_{1}^*, \dots, a_{n}, a_{n}^* \}$, with $t(a_{i}) = i = h(a_{i}^*)$ for $i<n$ but $t(a_{n}) = 0 = h(a_{n}^*)$. So the multiplicative preprojective relation at each vertex is:
 \[
 e_{i} (\rho - 1) =
 \begin{cases}
a_{0} a_{0}^* + a_{n} a_{n}^* + a_{0} a_{0}^*a_{n} a_{n}^* & \text{if } i = 0 \\
a_{n}^* a_{n} + a_{n-1}^* a_{n-1} + a_{n}^* a_{n} a_{n-1}^* a_{n-1} & \text{if } i = n \\
a_{i} a_{i}^* - a_{i-1}^* a_{i-1} & \text{otherwise.} \\
 \end{cases}
 \]
Shaw's isomorphism in Theorem \ref{thm: shaw} takes the form:
 \[
  a_{0} a_{0}^* \mapsto Z \hspace{1cm} a_{0} a_{1} \cdots a_{n-1}a_{n}^* \mapsto X \hspace{1cm} a_{n} a_{n-1}^* a_{n-2}^* \cdots a_{0}^* \mapsto Y.
 \]
 Define  $M :=e_{0} \Lambda$ and note that  $M= \oplus_{i=0}^{n} M_{i}$ where $M_{i} := e_{0} \Lambda e_{i}$. Observe that $M_{i} \cong (Z^{i}, Y)$, the two-sided ideal generated by $Z^{i}$ and $Y$ in $\Lambda$, as $e_{0} \Lambda e_{0}$-modules via a map,
 \begin{align*}
  a_{0} a_{1} \cdots a_{i-1} & \mapsto a_{0} a_{1} \cdots a_{i-1} a_{i-1}^* a_{i-2}^* \cdots a_{0}^* = (a_{0}a_{0}^*)^{i} = Z^{i} \\ 
  a_{n} a_{n-1}^* \cdots a_{i}^* & \mapsto a_{n} a_{n-1}^* \cdots a_{0}^*  = Y.
 \end{align*} 

Define the map
$$
\xymatrix{
\Lambda \ar[rr]^-{\phi} & &  \End_{e_{0} \Lambda e_{0} }(M),
}
$$
on generators by sending the idempotent $e_{i}$ at vertex $i$ to the projection map $M \rightarrow M_{i}$, and sending the arrows as follows:
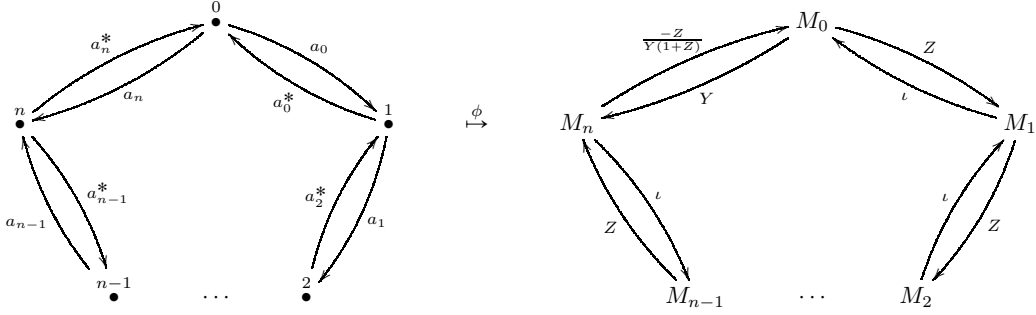
\[
\resizebox{\textwidth}{!}{
\xymatrix{
& & \overset{0}{\bullet} \ar@/^/[lld]^{a_{n}} \ar@/^/[rrd]^{a_{0}}  & &       &&        & & M_{0} \ar@/^/[lld]^{Y} \ar@/^/[rrd]^{Z}  & & \\
\overset{n}{\bullet} \ar@/^/[rdd]^{a_{n-1}^*} \ar@/^/[rru]^{a_{n}^*} & & & & \overset{1}{\bullet} \ar@/^/[ldd]^{a_{1}} \ar@/^/[llu]^{a_{0}^*}          & \overset{\phi}{\mapsto} &
                                                        M_{n} \ar@/^/[rdd]^{\iota} \ar@/^/[rru]^{\frac{-Z}{Y(1+Z)}} & & & & M_{1} \ar@/^/[ldd]^{Z} \ar@/^/[llu]^{\iota}\\
& & & &     &&       & & & & \\
& \overset{n-1}{\bullet} \ar@/^/[luu]^{a_{n-1}} & \cdots & \overset{2}{\bullet} \ar@/^/[ruu]^{a_{2}^*} &      &&      
                                                        & M_{n-1} \ar@/^/[luu]^{Z} & \cdots &  M_{2} \ar@/^/[ruu]^{\iota} &     
}
}
\]
where $\iota$ denotes the inclusion map. This map is well-defined at vertex $0$ and $n$ since
\[
Z + \frac{-YZ}{Y(1+Z)} + \frac{-YZ^{2}}{Y(1+Z)} = Z + \frac{-Z(1+Z)}{(1+Z)} = Z -Z = 0 
\]
and at vertex $i \neq 0, n$ since $Z-Z = 0$.

The surjectivity of $\phi$ follows from the observation that every $e_{0} \Lambda e_{0}$-module map of ideals 
is given by left multiplication by some element of the field of fractions of $e_{0} \Lambda e_{0}$.  The injectivity follows from the fact that $\Lambda$ is prime  (Proposition \ref{prop:prime_earlier}) and injectivity on $e_{0} \Lambda e_{0}$, as we now explain.

 By definition of primality, for any $a, c \in \Lambda$ both nonzero, there exists $b \in \Lambda$ such that $abc \neq 0$. Fix $\gamma \in \Lambda$ non-zero and take $a = e_0$ and $c = \gamma$ to get a non-zero path $\gamma' \in e_0 \Lambda$ containing $\gamma$ as a subpath. Then take $a = \gamma'$ and $c = e_0$ to get a non-zero path $\gamma'' \in e_0 \Lambda e_0$ containing $\gamma$ as a subpath. Since $\phi$ is injective on $e_0 \Lambda e_0$, $\phi(\gamma'') \neq 0$. Hence $\phi(\gamma) \neq 0$ and $\phi$ is injective.

To complete the proof that $\Lambda $ is an NCCR, we need to show that the module $M = e_{0}\Lambda $ is a reflexive $e_{0}\Lambda e_{0}$-module. The computation above shows that 
\[
\Hom_{e_0 \Lambda  e_0}(M_i, M_j) \cong e_i \Lambda  e_j \cong M_{j-i}
\]
as a module over $e_0 \Lambda e_0 \cong e_i \Lambda  e_i$, so in particular $\Hom_{e_0 \Lambda  e_0}(M, e_{0} \Lambda e_{0}) \cong  \oplus_{i} e_i \Lambda  e_0 \cong M.$
So $M$ is self-dual and hence reflexive as a $e_0 \Lambda e_0$-module. 
 \end{proof}

\subsubsection{The center}
Observe, if $A$ is an NCCR over some ring $R$, then the center $Z(A)$ is an $R$-algebra. Under suitable hypotheses, they are actually isomorphic. For example, this holds if $R$ an integrally closed Noetherian domain,
by Zariski's main theorem (as $\Spec Z(A) \to \Spec R$ is finite and birational). 

Instead of using this to establish our isomorphism, we consider an explicit map in the other direction.  More generally, suppose $A$ is a ring,  $e \in A$ is an idempotent, and $R := eAe$. Then we have a canonical map:
\begin{equation}\label{e:satake}
Z(A) \to R = eAe, \quad z \mapsto ez.
    \end{equation}
We call this the ``Satake map'' following the terminology for Hecke algebras, symplectic reflection algebras, etc.

Under natural conditions, the Satake map is well known to be an isomorphism.  Namely, note that $eA$ is an $(eAe)-A$ bimodule, and $\End_{A^{\op}}(eA) = eAe$. Then we have a natural map $A^{\op} \to \End_{eAe}(eA)$. 

\begin{lem}\label{lem: satake} 
Suppose that: (I) the natural map $A^{\op} \to \End_{eAe}(eA)$ is an isomorphism, and (II) $eAe$ is commutative. Then the Satake map \eqref{e:satake} is an isomorphism.
\end{lem}

\begin{proof}
We have an identification 
\[
Z(eAe) \cong \End_{eAe \otimes A^{\text{op}}}(eA) \cong Z(A) \hspace{1cm} z \mapsto ez.
\]
Since $eAe$ is commutative, $Z(A) \cong eAe$, via the Satake map.
\end{proof}
\begin{cor}\label{c:satake-mpa} The Satake map \eqref{e:satake} is an isomorphism for $A=\Lambda^1(\widetilde{A_n})$ and $e=e_v$, the idempotent at any vertex.
\end{cor}
\begin{proof}
This is a direct consequence of Lemma \ref{lem: satake}, once we
check the hypotheses: (I) and (II).
Thanks to
Proposition \ref{prop: NCCR}, $A \cong \End_{eAe}(eA)$ so (I) follows from $A \cong A^{op}$, a consequence of the independence of orientation established in \cite[Theorem 1.4]{Shaw}. By Shaw's Theorem \ref{thm: shaw} (II) holds (alternatively, the commutativity of the generators can be checked directly). 
\end{proof}

\begin{cor}
$\Lambda^1(\widetilde{A_n})$ is an NCCR over its center. 
\end{cor}

\begin{proof}
This follows immediately, provided we identify the $Z(\Lambda^1(\widetilde{A_n}))$-module structure on $\Lambda^1(\widetilde{A_n})$ with left multiplication. Indeed, given $z \in Z(\Lambda^1(\widetilde{A_n}))$ (by tracing through the above maps) its action on $\End_{e\Lambda^1(\widetilde{A_n})e}(M)$ via the Satake map is multiplication by $ez$.
\end{proof}
Note that Corollary \ref{c:satake-mpa} and Theorem \ref{thm: shaw} immediately imply Corollary \ref{c:center}.

\subsubsection{The center as functions on a quiver variety}
It remains to identify the center with the algebra of functions on the multiplicative quiver variety.  

In general, given a $kQ_0$-algebra $A$ and a finite-dimensional $kQ_0$-module $V$, we have a canonical algebra homomorphism $\ev: A \to k[\Rep(A,V)] \otimes \End_k(V)$, called ``evaluation'': $\ev(a)(\rho)=\rho(a)$. 

Suppose that $\rho: A \to \End(V)$ is an irreducible representation. Consider $z \in Z(A)$. If $k$ is algebraically closed, then by Schur's Lemma, $\rho(z)=\lambda_{\rho,z} \Id_V$ for some scalar $\lambda_{\rho,z}$. However, we don't assume here that $k$ is algebraically closed. We could fix this by passing to the algebraic closure, but this turns out to be unnecessary as follows.
\begin{lem} \label{lem: schur}
Suppose that 
$v \in Q_0$ is a vertex with $\dim V_v = 1$. Suppose $\rho$ is an irreducible representation. Then $\End(\rho)=k \cdot \Id_V$. 
\end{lem}
\begin{proof}
If $\phi \in \End(\rho)$, then $\rho(e_v) \phi = \phi \rho(e_v)$.
Therefore, $\phi$ preserves $\rho(e_v) V = V_v$. As this has dimension one, we have $\phi|_{V_v} = \lambda \Id_{V_v}$. Now, $\phi-\lambda \Id_V$ is not invertible. By Schur's Lemma over a general field, 
this implies that $\phi-\lambda \Id_V$ is zero.  So $\phi = \lambda \Id_V$.
\end{proof}
\begin{cor} Let $Q_0, A ,V, v$ and $\rho$ be as in Lemma \ref{lem: schur}. If $z \in Z(A)$, then $\rho(z) \in \End(V)$ is a scalar. 
\end{cor}
\begin{proof} Note that $\rho(z) \in \End(\rho)$. Then apply the lemma.
\end{proof}
\begin{cor} \label{cor: ev map}
Suppose that for some vertex $v$, we have $V_v = 1$, and moreover that there exists an irreducible representation $A \to \End(V)$. Then the restriction  $\ev|_{Z(A)}$ is an algebra map $Z(A) \to k[\Rep(A,V)] \cdot \Id_V$.
\end{cor}
\begin{proof}
Let $U \subseteq \Rep(A,V)$ be the locus of representations $\rho$ such that $\End(\rho)=k \cdot \Id_V$. This is a Zariski open subset, since $k \cdot \Id_V$ is always contained in $\End(\rho)$.  If $\rho \in \Rep(A,V)$ is irreducible, then by Lemma \ref{lem: schur}, $\rho \in U$. Thus, by our assumptions, $U$ is nonempty.
Since $\Rep(A,V)$ is a vector space, it is irreducible.  We conclude that $U$ is Zariski dense.


Now, for every $z \in Z(A)$, $\ev(z): \Rep(A,V) \to \End(V)$ is scalar-valued on $U$. As $U$ is dense, it is a scalar on all of $\Rep(A,V)$. Hence $\ev(z) \in k[\Rep(A,V)] \otimes \Id$. As $z$ was arbitrary, we obtain the result.
\end{proof}

Back to the situation at hand, for convenience let us orient $\widetilde{A_n}$ clockwise (note that the statement does not depend on orientation). We consider the vector space $V = k^{Q_0}$, which has the property $\dim V_v = 1$ for all $v \in Q_0$. Consider the representation on $V$
where each clockwise arrow is the identity (i.e., the one-by-one matrix $[1]$) and each counterclockwise arrow is zero. This defines a representation of the localization $L_{Q}$ that descends to an irreducible representation of $\Lambda^{1}(Q)$. Therefore, having satisfied the hypotheses of Corollary \ref{cor: ev map}, we obtain a canonical map
\begin{equation} \label{eq: can map}
    \ev_Z: Z(\Lambda^1(Q)) \to k[\mathcal{M}_{0,1}(Q,\delta)].
\end{equation}


\begin{prop} \label{prop: center-qv} The map $\ev_Z$ is an isomorphism.
\end{prop}
\begin{proof}
 To check surjectivity, let $f \in k[\mathcal{M}_{0,1}(Q,\delta)] = k[\Rep_\alpha(\Lambda^1(Q))]^{\GL_\alpha}$.  We wish to show that $f \in \ev_Z(Z(\Lambda^1(Q)))$. Note that $f$ is a polynomial in the matrix coefficient functions of the arrows (these are one by one matrices).  To be invariant under $\GL_\alpha$, the polynomial must in fact be a polynomial in the functions defined by closed paths in the quiver: each such closed path is canonically a scalar, as it is an endomorphism of a one-dimensional vector space. Thus it suffices to assume that there is a single closed path $a \in e_v \Lambda^1(Q) e_v$ such that $\rho(a) = f(\rho) \cdot \Id_{V_v}$ for all $\rho$.
 As the Satake map is an isomorphism (Corollary \ref{c:satake-mpa}), we must have $a = e_v z$ for some $z \in Z(\Lambda^1(Q))$.  Then, $\rho(a) = \ev_Z(z) \cdot \Id_{V_v}$. Hence $f(\rho) = \ev_Z(z)$ for all $\rho \in \Rep_\alpha(\Lambda^1(Q))$. This shows that $\ev_Z$ is surjective.
 
 By Corollary \ref{c:center} the source is an integral domain. Since we already proved surjectivity, injectivity will follow provided that the target also has dimension at least two.  This can be seen by constructing a two-parameter family of representations, e.g.,  we can take the representations with all clockwise arrows a matrix $(a)$ and all counterclockwise arrows a matrix $(b)$, with $ab \neq -1$. Alternatively, this statement follows from
 Theorem \ref{t:fnt-mult}. 
 \end{proof}

\section{The strong free product property}\label{s: free-product}
In this section, we prove the strong free product property for connected quivers containing a cycle. We first establish the strong free product property for the quivers $\widetilde{A_n}$ for $n \geq 0$ using the Diamond Lemma to build a section of the quotient map $\pi : L \rightarrow \Lambda^{q}(\widetilde{A_n})$. Then we establish the more general result using the corresponding result for \emph{partial} multiplicative preprojective algebras. See Subsection \ref{section: definition of sffp} for the prerequisite definitions.

 As results in previous Sections \ref{ss:bimodule-resolution}, \ref{ss: dual complex}, \ref{s:formal}, \ref{s:quiver_varieties}, \ref{s:NCCR} rely on results established in this section, the reader should note that we do not use any results beyond Section \ref{section: definition of sffp}, see Remark \ref{rem:non-circular}.

\subsection{The case of cycles}\label{ss:cycles}

Consider the quiver $\widetilde{A}_{n-1}$: with vertex set $(\widetilde{A}_{n-1})_{0} := \{ 0, 1, \dots, n-1 \}$ and arrow set $(\widetilde{A}_{n-1})_{1} =\{ a_{0}, a_{0}^*, a_{1}, a_{1}^*, \dots, a_{n-1}, a_{n-1}^* \}$ with $t(a_{i}) = i$ and $h(a_{i}) = i+1 \text{ (mod } n)$. Fix the ordering $a_{i} < a_{i+1} < a_{j}^* < a_{j+1}^*$ for all $i, j \in \{0, 1, \dots, n-2 \}$. The multiplicative preprojective algebra for this quiver, with respect to the ordering, is defined to be
$$
\Lambda^{q}(\widetilde{A}_{n-1}) := \frac{ k \overline{\widetilde{A}}_{n-1} [ (1+ a_{i} a_{i}^{*})^{-1}, (1+ a_{i}^* a_{i})^{-1}]_{i=0, \dots, n-1}  }
{  \left \langle \prod_{i=0}^{n-1} (1+ a_{i} a_{i}^*) \prod_{i=0}^{n-1} (1+ a_{i}^* a_{i})^{-1} - \sum_{i=1}^{n} q_{i} e_{i}  \right \rangle} =: \frac{L}{J}.
$$
Writing $a := \sum_{i} a_{i}$, $a^* := \sum_{i} a_{i}^*$, and $q = \sum_{i} q_{i} e_{i}$ since 
$$
1+a a^* = 1+\sum_{i} a_i a_i^{*} =  \prod_{i=0}^{n-1} (1+ a_{i} a_{i}^*)  \hspace{1cm} 
1+a^* a = 1+ \sum_{i} a_{i}^* a_{i} = \prod_{i=0}^{n-1} (1+ a_{i}^* a_{i})
$$
we have,
$$
\Lambda^{q}(\widetilde{A}_{n-1}) := \frac{ k \overline{\widetilde{A}}_{n-1} [ (1 + a a^*)^{-1}, (1+a^* a)^{-1}]
}
{ \left \langle  (1+ a a^*)(1+a^* a)^{-1} - q \right \rangle}.
$$
We write $r := (1+ a a^*)(1+a^* a)^{-1} - q$ for this relation, $S$ for the degree zero piece $k (\widetilde{A}_{n-1})_{0}$ of $\Lambda^{q}(\widetilde{A}_{n-1})$. As in Subsection \ref{section: definition of sffp}, let $B := S[t, (q+t)^{-1}]$ and $\overline{B} = tB$, spanned over $S$
by $t^m, (t')^m, m \geq 1$, for $t' := (q+t)^{-1}-q^{-1}$. Let $r' := (q+r)^{-1} - q^{-1}$.

We construct $\sigma: L/(r) *_{S} B \rightarrow L$ so that $(L, r, \sigma', B)$ satisfies the strong free product property using an explicit basis:
\begin{prop} \label{prop: basis for An}
$L$ is a free left $S$-module with basis consisting of $1$ together with all 
alternating products of elements of the following two sets, for $x := (1+aa^*)$:
$$
\fB := \left \{ x^{m} a^{\ell}, x^{m} (a^*)^{\ell} \mid m \in \bZ, \ell \in \bN  \right \}, \  \fR := \{r^m, (r')^m \mid m \in \bN\}.
$$
In particular, $\fB$ forms a basis for $\Lambda^q(\widetilde{A}_{n-1}) = L/(r)$, and 
$(L, r, \sigma, B)$ satisfies the strong free product property, with $\sigma$ induced from the inclusion of $\fB$ into $L$.
\end{prop}

\begin{proof}
Note that, for every vertex $i$, we have $e_i a = a e_j$ for a unique $j$, and similarly for the elements
$a^*, x, y := 1+ a^* a, x^{-1}, y^{-1}$, and by definition, $e_i r = r e_i$.
Therefore $L$ is spanned as a \emph{left} $S$-module by non-commutative monomials in $a, a^*, x, y, x^{-1}, y^{-1}, r$, and $r'$.
Define $\mathcal{M} := \langle a, a^*, x, y, x^{-1}, y^{-1}, r, r' \rangle$ the set of monomials and $\mathcal{P} := S\langle a, a^*, x, y, x^{-1}, y^{-1}, r, r' \rangle$ the set of non-commutative polynomials with coefficients in $S$. 

The set of relations, $R$, is the two-sided ideal generated by: 
\begin{gather} \label{relations: prop7.1}
    x x^{-1} = 1 = x^{-1} x, \quad y y^{-1} = 1 = y^{-1} y, \quad x = 1+aa^*, \quad y = 1+a^* a, \\ \label{relations_more: prop7.1}
    r = xy^{-1} - q, \quad r' = yx^{-1} - q^{-1}. 
\end{gather} 
So we have the presentation $L \cong \mathcal{P}/R$ and hence $\Lambda^{q}(\widetilde{A}_{n-1}) \cong \mathcal{P}/(R, r)$. \\
\\
The idea of the proof is to produce a basis of the \emph{quotient} $L = \mathcal{P} /R$ by realizing it as an $S$-module \emph{subspace} $\mathcal{P}_{\text{irr}} \subset \mathcal{P}$ spanned by \emph{irreducible} monomials, defined below. 

That is, we define an ordering, $\leq$, on the set $\mathcal{M}$. Then we use this ordering to build a system of reductions $\{ r_{i} \}$ from $R$ by reading each relation $R_{i} \in R$ as an $S$-module map, $r_{i}$, taking the leading term lt$(R_{i})$ to the smaller term lt$(R_{i}) - R_{i}$. We extend $r_{i}$ to $\mathcal{M}$ via $a$lt$(R_{i})b \mapsto a($lt$(R_{i}) - R_{i})b$ for $a, b \in \mathcal{M}$. We say $m \in \mathcal{M}$ is \emph{irreducible} (or in normal form) if every reduction is the identity on $m$ or, equivalently, if $m$ doesn't contain the leading term of any relation as a submonomial.

We will show that every $m \in \mathcal{M}$, reduces \emph{uniquely} to normal form, $m' \in \mathcal{P}_{\text{irr}}$, after applying \emph{finitely} many reductions. This implies the $S$-module map $\operatorname{r}: \mathcal{P} \rightarrow \mathcal{P}_{\text{irr}}$ given by $S$-linear extension of $m \mapsto m'$ is well-defined. Hence $\operatorname{r}$ splits the inclusion map $\mathcal{P}_{\text{irr}} \rightarrow \mathcal{P}$. As $\ker(\operatorname{r}) = R$, we conclude that $\operatorname{r}$ induces an $S$-module isomorphism $L \cong \mathcal{P}_{\text{irr}}$ and the set of irreducible monomials gives our desired basis.\\
\\
First we equip $\mathcal{M}$ with an ordering. Fix $w, z, z' \in \mathcal{M}$ and subsets $Z, Z' \subset \mathcal{M}$. Define,
\begin{align} \label{eq: number of occurences}
n_z(w) &:= \text{ the number of occurrences of } z \text{ in } w. \\
n_{z,z'}(w) &:= \text{ the number of occurrences of } z \text{ and } z' \text{ in } w \text{ with } z \text{ appearing before } z'  \label{eq: number of occurences 2}\\
n_{Z}(w) &:= \sum_{z \in Z} n_{z}(w) \hspace{1cm} \text{and} \hspace{1cm} n_{Z,Z'} := \sum_{z \in Z, z' \in Z'} n_{z,z'}.  \label{eq: number of occurences 3}  
\end{align}
Define a function $N: \mathcal{M} \rightarrow \bN^{5}$ taking $w$ to 
\begin{equation}
N(w) := (n_a(w), n_{\{a,a^*\},\{x,x^{-1},y,y^{-1}\}}(w), n_{\{ax, ax^{-1}\}}(w), n_{\{y,y^{-1}\}}(w), n_{\{r, r' \}}(w)) \in \bN^{5}.
\end{equation}
Define the ordering $w' \leq w$ in $\mathcal{M}$ if $N(w') \leq N(w)$ in the lexicographical ordering on $\bN^{5}$. This induces an ordering on $\mathcal{P}$, by extending $N$ to $\mathcal{P}$, via $N(\sum_{i} m_{i}) := \text{max}_{i} \{ N(m_{i}) \}$.\\
\\
Next, using this ordering, we define a system of reductions from the relations in \ref{relations: prop7.1}, \ref{relations_more: prop7.1}: 
\begin{itemize}
    \item Inverse Reductions: $x x^{-1}, \  x^{-1} x, \  y y^{-1}, \  y^{-1} y \mapsto 1$.
\item Short Cycle Reductions: $a a^* \mapsto x-1, \  a^* a \mapsto y-1$.
\item Reordering Reductions: $a^* x^{\pm 1} \mapsto y^{\pm 1} a^*, \  a y^{\pm 1} \mapsto x^{\pm 1} a$.
\item Substitution Reductions: $y^{-1} \mapsto x^{-1} (r +q), \ y \mapsto (r' + q^{-1}) x$ (if not preceded by $a$); \\ 
\textcolor{white}{Substitution Reductions:} $a x \mapsto a(r+q)y, \ ax^{-1} \mapsto a y^{-1} (r' + q^{-1})$.
\item Reductions in $B$: $rr', \ r'r \mapsto -q r' - q^{-1} r$.
\end{itemize} 
By design, if $w'$ is obtained from $w$ by applying a reduction, then $N(w') < N(w)$. This implies that any sequence of reductions terminates in finitely many steps, by the descending chain condition for the lexicographical ordering on $\bN^{5}$.

Next observe that under this reduction system $m \in \mathcal{M}$ is in normal form (or irreducible) if and only if it is alternating in $\fB$ and $\fR$. Therefore, the set of alternating words in $\fB$ and $\fR$ is a spanning set. It remains to show that $m \in \mathcal{M}$ reduces uniquely to normal form, which establishes linear independence. 






To prove uniqueness, we need to show whenever $w$ reduces to $r_{1}(w)$ and $r_{2}(w)$ that each further reduces to the same irreducible $w'$. Bergman's Diamond Lemma says to show uniqueness for all monomials $w$ it suffices to show uniqueness for specific $w =xyz$ where $xy$ and $yz$ are both leading terms for a relation in \ref{relations: prop7.1}, \ref{relations_more: prop7.1}, \cite[Theorem 1.2]{Bergman}. These $w$ are called \emph{overlap ambiguities}. If the two reduced expressions of $w = xyz$ (i.e $r_{1}(xy)z$ and $x r_{2}(yz)$) both further reduce to the same $w'$, we say the overlap ambiguity \emph{resolves}. To complete the proof it suffices to show all overlap ambiguities resolve.

Next, notice that any unresolvable ambiguity involving $y^{\pm 1}$ gives rise to an unresolvable ambiguity not involving $y^{\pm 1}$ by applying the Substitution or Reordering Reductions. So it suffices to check ambiguities in the following smaller system of reductions:
\begin{multicols}{3}
\noindent Inverse Reductions:
\begin{itemize}
\item[(1)] $x x^{-1} \xmapsto{r_{1}} 1$
\item[(2)] $ x^{-1} x \xmapsto{r_{2}} 1$
\end{itemize}
Short Cycle Reductions:
\begin{itemize}
\item[(3)] $a a^*\xmapsto{r_{3}} x-1$
\item[(4)] $a^* a \xmapsto{r_{4}} (r' + q^{-1}) x - 1$
\end{itemize} 
Reordering Reductions:
\begin{itemize}
\item[(5)] $a^* x \xmapsto{r_{5}} (r' + q^{-1}) x a^*$
\item[(6)] $ a x \xmapsto{r_{6}} qxa 
- q a r' x$
\item[(7)] $ax^{-1} \xmapsto{r_{7}}  \\ 
 \textcolor{white}{ax-1} x^{-1} a (r' + q^{-1})$
\item[(8)] $a^* x^{-1} \xmapsto{r_{8}}  x^{-1} (r+q) a^*$
\end{itemize}
 Substitution Reductions:
 \begin{itemize}
     \item[(9)]  $y^{-1} \mapsto x^{-1} (r + q)$
     \item[(10)] $y \mapsto (r' + q^{-1}) x$
 \end{itemize}
 Reductions in $B$:
 \begin{itemize}
     \item[(11)] $rr' \mapsto -q r' - q^{-1} r$
     \item[(12)] $r'r \mapsto -q r' - q^{-1} r$
 \end{itemize}
\end{multicols}
 

 The Substitution Reductions and Reductions in $B$ don't overlap with any others, so the only overlap ambiguities are amongst the (1)--(8), involving the generators $a,a^*,x,x^{-1}$ only. The Inverse, Short Cycle, and Reordering Reductions are quadratic in these generators giving rise to the following 12 cubic overlap ambiguities:
 



\begin{multicols}{4}
\begin{itemize}
    \item [(I)]  $x x^{-1} x$    \item[(II)]  $x^{-1} x x^{-1}$  \item[(III)]  $a a^* a$ 
    \item[(IV)]  $a^* a a^*$     \item[(V)]   $a^* x x^{-1}$     \item[(VI)] $a x x^{-1}$  \item[(VII)] $a^* x^{-1} x$  \item[(VIII)]$a x^{-1} x$       \item[(IX)]  $a a^* x$
    \item[(X)] $a^* a x$      \item[(XI)] $a a^* x^{-1}$       \item[(XII)] $a^* a x^{-1}$. 
\end{itemize}
\end{multicols}
The resolution of (I) and (II) are immediate (and are completely general, having to do with a basis for $k[x,x^{-1}]$). Here is a summary of the remaining resolutions of ambiguities:
$$
\begin{array}{ll}
\text{(III) } (r_3 - r_6 \circ r_4)(aa^* a) = 0 & \text{(VIII) } (r_{8} \circ r_{7} - r_{2})(a x^{-1} x) = 0 \\
\text{(IV) } (r_4 - r_{5} \circ r_{3})(a^* a a^*) = 0 & \text{(IX) } (r_3 - r_3 \circ r_6 \circ r_5)(a a^* x) = 0 \\
\text{(V) } (r_{8} \circ r_{5} - r_{1})(a^* x x^{-1}) = 0 & \text{(X) } (r_4 - r_4 \circ r_{4} \circ r_{5} \circ r_{6})(a^* a x) = 0 \\
\text{(VI) } (r_{7} \circ r_{6} - r_{1})(a x x^{-1}) = 0 & \text{(XI) } (r_3 - r_{3} \circ r_{7} \circ r_{8})(a a^* x^{-1}) = 0 \\
\text{(VII) } (r_{5} \circ r_{8} - r_{2})(a^* x^{-1} x) = 0 & \text{(XII) } (r_4 - r_{4} \circ r_{8} \circ r_{7})(a^* a x^{-1}) = 0.
\end{array}
$$
We explicitly demonstrate (X), one of the more involved resolutions:
\begin{align*}
a^* a x &= (a^* a) x \xmapsto{r_4} [(r'+q^{-1}) x - 1]x = (r'+q^{-1}) x^2 - x 
\intertext{and}  
a^* a x &= a^* (ax) \xmapsto{r_6} a^* (q x a - q a r' x) \\
&\xmapsto{r_4 \circ r_5} q  (r'+q^{-1}) x a^* a - q ((r'+q^{-1}) x - 1 ) r' x \\
&\xmapsto{r_4} q (r'+q^{-1}) x ((r' + q^{-1}) x - 1) - q ((r'+q^{-1}) x - 1) r' x \\
&= q (r'+q^{-1}) x (q^{-1} x - 1) + q r' x = (r'+q^{-1}) x^2 - x.
\end{align*}

\end{proof}
\begin{rem}\label{r:overline-b-matters}
The choice of $\overline{B}$ was important here. If we instead had defined it so that $(q+t)^{-1} \in \overline{B}$, i.e., if we replace $r'=(q+r)^{-1} - q^{-1} \in \fR$ by $(q+r)^{-1}$, then our desired basis would no longer be linearly independent. Indeed, reducing $a a^* a$ one way, we get $(x-1)a = xa - a$, which is irreducible, whereas the other way we get $a(y-1)=a(q+r)^{-1} x - a$, also irreducible. That is, $xa = a(q+r)^{-1} x$, an equality of two distinct irreducible elements.
\end{rem}


\begin{prop} \label{prop: prime}
$\Lambda^{q}(\widetilde{A_n})$ is prime for all $n \geq 0$ and all $q \in (k^{\times})^{n+1}$.
\end{prop}

\begin{proof}
We need to show, for every pair $f, g  \in \Lambda^{q}(\widetilde{A_n})$, both nonzero, there exists some $h \in \Lambda^{q}(\widetilde{A_n})$ such that $fhg \neq 0$. It suffices to show that there exists vertices $i, j$ and $h$ such that $e_{i}fe_{j} h g  \neq 0$, and hence we can take $f$ to be a linear combination of basis elements that all begin at $i$ and end $j$. By right multiplication by $a^{n-j}$ or $(a^*)^{j}$, one can take $f$ to be a linear combination of basis elements ending at vertex 0. By left multiplication by $a^{i}$ or $(a^*)^{n-i}$ and then applying Reordering Reductions---the $q$-commutator, $ax - qxa$ is zero, for instance, in $\Lambda^{q}(\widetilde{A_n})$---one can take $f$ to be a linear combination of basis elements starting and ending at vertex 0. In fact, $f$ is of the form $e_0 f_1(x,x^{-1})f_2(a^{n+1})$, where $f_1 \neq 0$ and $f_2$ has nonzero constant term. And similarly, we can take $g =e_0 g_1(x,x^{-1})g_2(a^{n+1}) $. Then their product has nonzero term $e_{0}f_1(x,x^{-1})f_{2}(a^{n+1})(0) g_1(x,x^{-1})g_{2}(a^{n+1})(0)$ and hence is nonzero.

\end{proof}

\subsection{Partial multiplicative preprojective algebras} \label{ss:pmpa}

First we define a \emph{partial multiplicative} preprojective algebra following the definition of a partial preprojective algebra by Etingof and Eu in \cite[Definition 3.1.1]{EE}.
\begin{defn}
Fix a quiver $Q$ and $q \in (k^{*})^{Q_{0}}$. Define a partition of the vertex set $Q_{0} = \cB \sqcup \cW$ into a set $\cB$ of \emph{black} vertices and a set $\cW$ of \emph{white} vertices. The \emph{partial multiplicative preprojective algebra} of $(Q, \cW)$ is 
\[
\Lambda^{q}(Q,\cW) := L / (r_{\cB}), \quad \text{where}\quad  r_{\cB} := 1_{\cB} r 1_{\cB}, \quad \text{for} \quad 1_{\cB} := \sum_{j \in \cB} e_j.
\]
\end{defn}
In words, we don't enforce the relations at the white vertices. 
Hence this algebra interpolates between $\Lambda^{q}(Q, Q_{0})= L$ and $\Lambda^q(Q,\emptyset) = \Lambda^q(Q)$. 

\begin{defn}
Let $Q$ be a quiver and let $\Gamma$ be its underlying graph. Fix $\mathcal{R} \subset Q_{0}$.
\begin{itemize}
\item A subgraph $T \subset \Gamma$ is a \emph{tree} if it is connected and acyclic. 
\item A tree $T \subset \Gamma$ is \emph{rooted in $\mathcal{R}$} if it has a single vertex, called the root, in $\mathcal{R}$.
\item A \emph{forest rooted in $\mathcal{R}$} is a disjoint union of trees \emph{rooted in $\mathcal{R}$}.
\item A subgraph $S \subset \Gamma$ is \emph{spanning} if the vertex set of $S$ is $Q_{0}$.
\end{itemize}
\end{defn}

Notice that every doubled quiver $\overline{Q}$ with $\cW \subset Q_{0}$ non-empty has a spanning forest, $F$, rooted in $\cW$. We view such an $F$ as a subquiver of $\overline{Q}$ by orienting the arrows \emph{towards} the roots, see Figure \ref{fig: spanning forest}. Since the isomorphism class of $\Lambda^{q}(Q)$ is independent of the orientation of $Q$, see Remark \ref{rem: independence}, we can assume that $F_{1} \subset Q_{1}$.

Let $B := \cB[t, (t+q)^{-1}]$. 
Each choice of spanning forest of $\overline{Q}$ rooted at $\cW$ gives rise to a linear isomorphism $\sigma':\Lambda^{q}(Q, \cW)*_{kQ_{0}} B \rightarrow L$ and hence a basis for $\Lambda^{q}(Q, \cW) =L / (r_{\cB})$.

\begin{prop}\label{p:pmpa-basis} 
Let $Q$ be a connected quiver and $Q_0 = \cB \sqcup \cW$ a decomposition into black and white vertices with $\cW \neq \emptyset$. Then
$(L,r_\cB, \sigma, B)$  satisfies the strong free product property for some choice of $\sigma$.

In more detail, let $F \subset \overline{Q}$ be a spanning forest rooted in $\cW$ with arrows $F_{1} \subset Q_{1}$ directed towards the roots. 

A basis for $L$ is given by concatenable words in the set,
\[
\{ a, \ x_{a}, \ x_{a}^{-1} \mid a \in \overline{Q}_1 \} \cup \{r_{\cB}, \  r'_{\cB}:=(q+r_{\cB})^{-1} - q^{-1}\},
\]
such that the following subwords do not occur:
\[
\begin{array}{cccccl}
x_a x_a^{-1}, &  x_a^{-1} x_a, &  a a^*, &  a x_{a^*}^{\pm 1}, && \text{ for } a \in \overline{Q}_1 \\  
x_{a}^{\pm 1}, & x_{a^*}^{-1}, & x_{a^*}a^*, & x_{a^*}^2, && \text{ for } a \in F_1 \\
r_{\cB}r_{\cB}', & r_{\cB}'r_{\cB} &&&&
\end{array}
\]
The words in which $r_{\cB}$ and $r_{\cB}'$ do not occur form a basis for $\Lambda^q(Q,\cW) = L/(r_{\cB})$, and
the section $\sigma$ is given by the inclusion of these elements.
\end{prop}

\begin{figure}

\begin{tikzpicture}[scale=.9]
\draw (0,0) circle [radius=0.125]; 
\draw (2,0) circle [radius=0.125]; 
\draw (4,0) circle [radius=0.125]; 
\draw[fill] (1,1) circle [radius=0.125]; 
\draw[fill] (3,1) circle [radius=0.125]; 
\draw[fill] (2,2) circle [radius=0.125]; 

\draw[->, thick] (0,.15) to [out=75,in=200] (.85, 1); 
\draw[gray, ->, thick] (1,.85) to [out=250, in=20] (.15, 0); 

\draw[->, thick] (1.85,.1) to [out=158, in=292] (1.15, .85); 
\draw[gray, ->, thick] (1.2,.95) to [out=338, in=112] (1.85, .20); 

\draw[->, thick] (2.15,0) to [out=22, in=158] (3.85, 0);  
\draw[gray, ->, thick] (3.85,-.15) to [out=202, in=338] (2.15, -.15); 

\draw[->, thick] (2,.2) to [out=68, in=210] (2.85, .9); 
\draw[gray, ->, thick] (3, .8) to [out=268, in=22] (2.15, .15); 

\draw[->, thick] (3,1.15) to [out=112, in=338] (2.15, 2);  
\draw[gray, ->, thick] (2.05, 1.8) to [out=292, in=150] (2.85, 1.1); 

\draw[thick,->] (2, 2.25) arc (-45:280:3mm); 
\draw[gray, thick,->] (1.75, 1.95) arc (45:370:3mm);

\end{tikzpicture} \hspace{1cm}
\begin{tikzpicture}[scale=.9]
\draw[forest] (0,0) circle [radius=0.125]; 
\draw[forest] (2,0) circle [radius=0.125]; 
\draw[forest] (4,0) circle [radius=0.125]; 
\draw[forest, fill=forest] (1,1) circle [radius=0.125]; 
\draw[forest, fill=forest] (3,1) circle [radius=0.125]; 
\draw[forest, fill=forest] (2,2) circle [radius=0.125]; 

\draw[->, thick] (0,.15) to [out=75,in=200] (.85, 1); 
\draw[forest, ->, thick] (1,.85) to [out=250, in=20] (.15, 0); 

\draw[->, thick] (1.85,.1) to [out=158, in=292] (1.15, .85); 
\draw[->, thick] (1.2,.95) to [out=338, in=112] (1.85, .20); 

\draw[->, thick] (2.15,0) to [out=22, in=158] (3.85, 0);  
\draw[->, thick] (3.85,-.15) to [out=202, in=338] (2.15, -.15); 

\draw[->, thick] (2,.2) to [out=68, in=210] (2.85, .9); 
\draw[forest, ->, thick] (3, .8) to [out=268, in=22] (2.15, .15); 

\draw[->, thick] (3,1.15) to [out=112, in=338] (2.15, 2);  
\draw[forest, ->, thick] (2.05, 1.8) to [out=292, in=150] (2.85, 1.1); 

\draw[thick,->] (2, 2.25) arc (-45:280:3mm);
\draw[thick,->] (1.75, 1.95) arc (45:370:3mm);

\end{tikzpicture} \hspace{1cm}
\begin{tikzpicture}[scale=.9]
\draw[forest] (0,0) circle [radius=0.125]; 
\draw[forest] (2,0) circle [radius=0.125]; 
\draw[forest] (4,0) circle [radius=0.125]; 
\draw[forest, fill=forest] (1,1) circle [radius=0.125]; 
\draw[forest, fill=forest] (3,1) circle [radius=0.125]; 
\draw[forest, fill=forest] (2,2) circle [radius=0.125]; 

\draw[->, thick] (0,.15) to [out=75,in=200] (.85, 1); 
\draw[->, thick] (1,.85) to [out=250, in=20] (.15, 0); 

\draw[->, thick] (1.85,.1) to [out=158, in=292] (1.15, .85); 
\draw[forest, ->, thick] (1.2,.95) to [out=338, in=112] (1.85, .20); 

\draw[->, thick] (2.15,0) to [out=22, in=158] (3.85, 0);  
\draw[->, thick] (3.85,-.15) to [out=202, in=338] (2.15, -.15); 

\draw[->, thick] (2,.2) to [out=68, in=210] (2.85, .9); 
\draw[forest, ->, thick] (3, .8) to [out=268, in=22] (2.15, .15); 

\draw[->, thick] (3,1.15) to [out=112, in=338] (2.15, 2);  
\draw[forest, ->, thick] (2.05, 1.8) to [out=292, in=150] (2.85, 1.1); 

\draw[thick,->] (2, 2.25) arc (-45:280:3mm);
\draw[thick,->] (1.75, 1.95) arc (45:370:3mm);

\end{tikzpicture}
\caption{The quiver on the left is a doubled quiver, obtained by adding the grey arrows. It has three white vertices and three black vertices. The middle and right diagrams show two inequivalent spanning forests, in green, with roots at the white vertices.} 
\label{fig: spanning forest}
\end{figure}
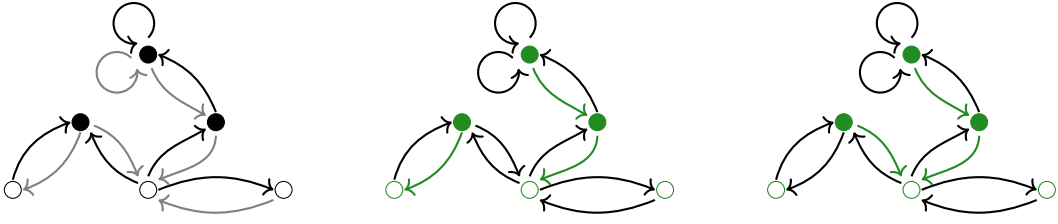

\begin{proof}
The proof parallels that of Proposition \ref{prop: basis for An}. 
Write $r:=r_{\cB}$ and $r':= r'_{\cB}$. 

Note that $L$ is spanned by the set, $\mathcal{M}$, of concatenable words in $\{ a, x_{a}, x_{a}^{-1}, r, r' \mid a \in \overline{Q}_{1} \}$. These words are subject to the following relations, depending on a choice of ordering $\leq$ on the arrows $a \in \overline{Q}_{1}$:
\begin{gather}
    x_{a} x_{a}^{-1} = 1 = x_{a}^{-1} x_{a}, 
    \quad  x_{a} = 1 + a a^*, \\
     r = \prod_{ \scalemath{0.65}{\begin{array}{c} a \in \overline{Q} \\ t(a) \in \cB \end{array}} } x_{a}^{\epsilon(a)} 
     - q, \quad r' = \prod_{\scalemath{0.65}{\begin{array}{c} a \in \overline{Q} \\ t(a) \in \cB \end{array}} } x_a^{-\epsilon(a)}
     - q^{-1} \\
     rr' = r'r = -q r' - q^{-1} r,
\end{gather}
where recall we write $t(a)$ for the \emph{tail} or source of $a$, not the target. Define 
\[
l_{a} := \prod_{  \scalemath{0.65}{\begin{array}{c} b \in \overline{Q} \\ b < a, \ t(b) \in \cB \end{array}} } x_{b}^{\epsilon(b)}  \quad \quad  \text{and} \quad \quad r_{a} := \prod_{ \scalemath{0.65}{\begin{array}{c} b \in \overline{Q} \\ b > a, \ t(b) \in \cB \end{array}}} x_{b}^{\epsilon(b)}.
\]
So for $a \in \overline{Q}_{1}$ with $t(a) \in \cB$ we have the relation 
$$
l_{a} (1+ a a^*)^{\epsilon(a)} r_{a} = (r+q)e_{t(a)} \  \implies \  x_{a}^{\epsilon(a)}= l_{a}^{-1} (r+q)(e_{t(a)}) r_{a}^{-1}.
$$
Hence in $L$, define $\red_{a}^{\epsilon(a)} := l_a^{-1} (r+q)(e_{t(a)}) r_a^{-1}$.
 
 We implement the above relations with the following reductions:
\begin{itemize}
    \item Inverse Reductions: $x_{a} x_{a}^{-1}, x_{a}^{-1} x_{a} \mapsto 1$ for $a \in \overline{Q}_{1}$.
    \item Short Cycle Reductions: $a a^* \mapsto x_{a}-1$ for $a \in \overline{Q}_{1}$.
    \item Reordering Reductions: $a^* x_{a}^{\pm} \mapsto x_{a^*}^{\pm} a^*$ for $a \in \overline{Q}_{1}$. 
    \item Substitution Reductions: $x_{a}^{\pm} \mapsto \red_{a}^{\pm}, \quad x_{a^*}^{-1} \mapsto 1- a^* \red_{a}^{-1} a$, \\ 
    \textcolor{white}{Substitution Reductions:} $x_{a^*}^{2} \mapsto x_{a^*} + a^* \red_{a} a, \quad x_{a^*} a^* \mapsto a^* \red_{a}, \text{ for } a \in F_{1}$ 
    \item Reductions in $B$: $rr', r'r \mapsto -q r' - q^{-1} r$.
\end{itemize}


For each word $w \in \mathcal{M}$, use the definition in (\ref{eq: number of occurences 2}) to define a weighted size,
$$
\varphi_a(w)  := n_{ \{ a, a^* \}}(w) + \frac{3}{2} n_{ \{ x_{a}, x_{a^*} \}}(w) + 3 n_{ \{ x_{a}^{-1}, x_{a^*}^{-1} \}}(w) 
$$
for each $a \in \overline{Q}_{1}$. Define a total ordering on the arrows $(\overline{Q}_{1}, \prec)$ such that,
\begin{align*}
a \prec a' &\text{ if } a \in F_1, a' \in \overline{Q}_{1} \setminus F_{1}, \\
&\text{ or if } a, a' \in F_{1} \text{ with } a' \text{ disconnected from } \cW \text{ in } F_{1} \setminus \{ a \}.
\end{align*}
Intuitively, we are saying that arrows in the spanning forest come before the rest in the ordering, with arrows closer to the white vertices coming first.  Using $\prec$, $\varphi_{a}$, and (\ref{eq: number of occurences 2}), (\ref{eq: number of occurences 3}) define 
\[
N': \mathcal{M} \rightarrow \bN^{(\overline{Q}_{1}, \prec)} \times \bN^{2}  \quad \quad
w \mapsto (2 \varphi_{a}(w), n_{ \{ a \mid a \in \overline{Q}_{1} \}, \{x_{a} \mid a \in \overline{Q}_{1} \}}(w), n_{\{r, r'\}}(w)),
\]
from which we say $w \leq w'$ if $N'(w) \leq N'(w')$ in the lexicographical ordering on $\bN^{|\overline{Q}_{1}| +2}$. 

Notice, as in Proposition \ref{prop: basis for An}, that $N'(r_{i}(w)) < N'(w)$ for any word $w$ and reduction $r_{i}$ with $r_{i}(w) \neq w$. First notice that, by design, $\varphi_{a}$ decreases under the following reductions:
\begin{itemize}
    \item Inverse Reductions: $\varphi_{a}(x_{a} x_{a}^{-1}) = \varphi_{a}( x_{a}^{-1} x_{a})= 3+ 3/2  > 0 = \varphi_{a}(1)$   
    \item Short Cycle Reductions: $\varphi_{a}(a a^*) = 2 > 3/2 = \varphi_{a}(x_{a})$ 
    \item Substitution Reductions: $\varphi_{a}(x_{a}) = 3/2  > 0 =\varphi_{a}(\red_{a})$,\\
    \textcolor{white}{Substitution:} 
    $\varphi_{a}(x_{a}^{-1}) = 3  > 0 =\varphi_{a}(\red_{a}^{-1})$, 
    \quad 
    $\varphi(x_{a^*}^{-1}) = 3 >  2 =\varphi_{a}(a^* \red_{a}^{-1} a)$, \\ 
    \textcolor{white}{Substitution:} 
    $\varphi(x_{a^*}^{2}) = 3 > 2 =\varphi_{a}(a^* \red_{a} a) \text{ and } 
    \varphi(x_{a^*}^{2}) = 3 > 3/2 = \varphi_{a}(x_{a^*}),$ \\
    \textcolor{white}{Substitution:}
    $\varphi_{a}(x_{a^*} a^*)= 5/2 > 1 = \varphi_{a}(a^* \red_{a}) $.
\end{itemize}
For the Substitution Reductions observe that $\red_{a}$ for $a \in F_{1}$ has subwords $x_{b}^{\pm 1}, x_{b^*}^{\pm 1}$ for only $b \in F_{1}$ which are necessarily \emph{farther} from the root than $a$, and the remaining arrows are not in the spanning forest. Consequently, $\varphi_{a}$ decreasing---despite $\varphi_{b}$ increasing for some $b \succ a$---implies that $N'$ decreases. The Reordering Reductions preserve all $\varphi_{a}$ but decrease $n_{ \{ a \mid a \in \overline{Q}_{1} \}, \{x_{a} \mid a \in \overline{Q}_{1} \}}$ by definition, and hence decrease $N'$. The Reductions in $B$ preserve all $\varphi_{a}$ and $n_{ \{ a \mid a \in \overline{Q}_{1} \}, \{x_{a} \mid a \in \overline{Q}_{1} \}}$ but decrease $n_{\{r, r'\}}$, hence $N'$.

We conclude that every $w \in \mathcal{M}$ reduces to a $kQ_{0}$-linear combination of words without subwords in the leading terms of the reductions:
$$
\{ x_{a} x_{a}^{-1}, x_{a}^{-1} x_{a}, a a^*, a x_{a^*}, a x_{a^*}^{-1} \mid a \in \overline{Q}_{1} \} \cup \{ x_{a^*}^{-1}, x_{a^*} a, x_{a^*}^{2} \mid a \in F_{1} \}
$$
after applying \emph{finitely} many reductions.

Note that some generators are nonreduced: $x_{a}$, $x_{a}^{-1}$, and $x_{a^*}^{-1}$ for $a \in F_{1}$.  Therefore, we can put in reductions for each of these and throw out all other reductions involving these generators, provided we check that all the defining relations still reduce to zero.
We have the reductions:
\begin{multicols}{2}
\begin{itemize}
    \item[(1)] $x_{a} x_{a}^{-1} \xmapsto{r_{1}} 1$ for $a \in \overline{Q}_{1}$
    \item[(2)] $x_{a}^{-1} x_{a} \xmapsto{r_{2}} 1$ for $a \in \overline{Q}_{1}$
    \item[(3)] $a a^* \xmapsto{r_{3}} \red_{a}-1$ for $a \in F_{1}$ 
    \item[(4)] $a a^* \xmapsto{r_{4}} x_{a} -1$ for $a \notin F_{1}$ 
\end{itemize}
\begin{itemize}
    \item[(5)] $a x_{a^*}^{\pm} \xmapsto{r_{5}} x_{a}^{\pm} a$ for $a \notin \overline{F}_{1}$
    \item[(6)] $a x_{a^*} \xmapsto{r_{6}} \red_{a} a$ for $a \in F_{1}$
    \item[(7)] $x_{a^*}^{2} \xmapsto{r_{7}} x_{a^*} + a^* \red_{a} a$ for $a \in F_{1}$
    \item[(8)] $x_{a^*} a^* \xmapsto{r_{8}} a^* \red_{a}$  for $a \in F_{1}$
\end{itemize}
\end{multicols}
\noindent
which don't overlap with the remaining reductions:
\begin{itemize}
    \item[] Substitution Reductions: $x_a^{\pm 1} \mapsto \red_a^{\pm 1}, x_{a^*}^{-1} \mapsto 1-a^* x_a^{-1} a, \quad a \in F_1$; 
     \item[] Reductions in B: $rr', r'r \mapsto -q r' - q^{-1} r$.
\end{itemize}

As before, reductions (3) and (4) imply the relations $x_a = 1+aa^*$, whereas the Substitution Reductions imply the defining relations for $r, r'$. So this is a valid reduction system.

This reduction system has thirteen ambiguities: 
\begin{multicols}{3}
\begin{itemize}
    \item[(I)] $x_{a} x_{a}^{-1} x_{a}$ for $a \notin \overline{F}_{1}$
    \item[(II)] $x_{a}^{-1} x_{a} x_{a}^{-1}$ for $a \notin \overline{F}_{1}$
    \item[(III)]  $a x_{a^*} x_{a^*}^{-1}$ for $a \notin \overline{F}_{1}$
    \item[(IV)]  $a x_{a^*}^{-1} x_{a^*}$ for $a \notin \overline{F}_{1}$
    \item[(V)] $a x_{a^*}^{2}$ for $a \in F_{1}$
    \item[(VI)] $x_{a^*}^{2} a^*$ for $a \in F_{1}$
    \item[(VII)] $x_{a^*} a^* a$ for $a \in F_{1}$ 
    \item[(VIII)] $a x_{a^*} a^*$ for $a \in F_{1}$
    \item[(IX)] $a^* a x_{a^*}$ for $a \in F_{1}$ 
    \item[(X)] $a^* a x_{a^*}$ for $a \in \overline{Q} \backslash F_{1}$
    \item[(XI)]  $a a^* a$ for $a \in \overline{Q}_{1} \backslash \overline{F}_{1}$ 
    \item[(XII)] $a a^* a$ for $a \in F_{1}$
    \item[(XIII)] $a a^* a$ for $a^* \in F_{1}$
    \item []
    \item [] 
\end{itemize}
\end{multicols}
which all resolve by the resolutions
\begin{multicols}{2}
\begin{itemize}
    \item[(I)] $(r_{1} - r_{2})(x_{a} x_{a}^{-1} x_{a})=0$
    \item[(II)] $(r_{2} - r_{1})(x_{a}^{-1} x_{a} x_{a}^{-1}) = 0$
    \item[(III)] $(r_{1}- r_{1} \circ r_{5} \circ r_{5})(a x_{a^*} x_{a^*}^{-1}) = 0$
    \item[(IV)] $(r_{2}- r_{2} \circ r_{5} \circ r_{5})(a x_{a^*}^{-1} x_{a^*}) = 0$
    \item[(V)] $(r_{3} \circ r_{6} \circ r_{7} - r_{6} \circ r_{6})(a x_{a^*}^{2}) = 0$
    \item[(VI)] $(r_{3} \circ r_{8} \circ r_{7} - r_{8} \circ r_{8})(x_{a^*}^{2} a^*) = 0$
    \item[(VII)] $(r_{7} \circ r_{4} - r_{8})(x_{a^*} a^* a) = 0$
    \item[(VIII)] $(r_{3} \circ r_{6}  - r_{3} \circ r_{8})(a x_{a^*} a^*) = 0$
    \item[(IX)] $(r_{7} \circ r_{4} - r_{6})(a^* a x_{a^*}) = 0$
    \item[(X)] $(r_{4} \circ r_{5} \circ r_{5} - r_{4})(a^* a x_{a^*})=0 $
    \item[(XI)] $(r_{4} - r_{5} \circ r_{4})(a a^* a) = 0$
    \item[(XII)] $(r_{6} \circ r_{4} -r_{3})(a a^* a)=0$
    \item[(XIII)] $(r_{8} \circ r_{4} -r_{3})(a a^* a)=0$.
\end{itemize}
\end{multicols}

The resolutions of the ambiguities (I)--(IV) and (X)--(XIII) are quick, leaving the computational heart of the calculations with the five resolutions (V)--(IX). Note that the resolutions for (V) and (VI) are identical after swapping the roles of reductions $r_{6}$ and $r_{8}$, and similarly for (IX) and (VII), leaving three calculations: (V), (VIII), and (IX). These ambiguities express the overlap of $r_{6}$ with $r_{7}$, $r_{8}$, and $r_{4}$ respectively and further reduce uniquely to $\red_{a}^{2} a$, $\red_{a}(\red_{a} -1)$, and $a^* \red_{a} a$.   
\end{proof}

\subsection{A convenient substitution} \label{ss:substitution}
It will be convenient for us to make the substitutions:
\begin{equation}
\overline{x_a^{\pm}} := x_a^{\pm 1} - 1,
\end{equation}
motivated as follows:

Let $A \cong \Lambda^{q}(Q, \cW)$ for $Q$ connected, and $\cW$ possibly empty. Let $I$ be the ideal generated by all paths beginning and ending at vertices having either $q=1$ or in $\cW$ (if non-empty). Then $A/I$ is nonzero, and we can make use of the $I$-adic filtration.  The modified generators $\overline{x_a^{\pm}}$, for $a$ an arrow in $I$, have the advantage of lying in the ideal $I$. As we will show, in the cases $Q$ contains a cycle and $\cW \neq \emptyset$, the $I$-adic filtration is Hausdorff. 

Thus, we get an embedding of $A$ into the completion $\widehat A_I$, realizing $\overline{x_a^{\pm}}$ as power series with zero constant term. In the special case where $q=1$ at all black vertices, this embedding sends every modified generator, $\overline{x_a^{\pm}}$, to a non-commutative power series in arrows \emph{with} zero constant term. This completion is closely related to the completion of (partial) additive preprojective algebras with $\lambda=0$ at all black vertices.


Practically speaking, we only require the above substitution at white vertices to obtain a basis for quivers containing cycles, see Subsection \ref{ss: cont cycle}. But theoretically, we advocate for this substitution at any vertex where we think of $q$ as a deformation parameter based at $q=1$.

Let us explain how this substitution works in the case of the cycle $\widetilde{A_n}$ (although we do not strictly need it in that case).
We formally set $\overline{x^{\pm}} := x^{\pm 1} - 1$ and
$\overline{y^{\pm}} := y^{\pm 1} - 1$; then the modified reductions from Section \ref{ss:cycles} are the following ones: 
\begin{itemize}
    \item Inverse Reductions: $\overline{x^+} \overline{x^-}, \overline{x^-} \overline{x^+} \mapsto -\overline{x^+} - \overline{x^-}$ and 
    $\overline{y^+} \overline{y^-}, \overline{y^-} \overline{y^+} \mapsto -\overline{y^+} - \overline{y^-}$
\item Short Cycle Reductions: $a a^* \mapsto \overline{x^+}, a^* a \mapsto \overline{y^+}$.
\item Reordering Reductions: $a^* \overline{x^{\pm}} \mapsto \overline{x^{\pm}} a^*, a \overline{y^{\pm}} \mapsto \overline{y^{\pm}} a$.
\item Substitution Reductions: $\overline{y^-} \mapsto \overline{x^-}(r +q) + r + (q-1)$, (if not preceded by $a$); \\
\textcolor{white}{Substitution Reductions:} $\overline{y^+} \mapsto (r' + q^{-1}) \overline{x^+} + r' + (q^{-1}-1)$ (if not preceded by $a$); \\ 
\textcolor{white}{Substitution Reductions:} $a \overline{x^+} \mapsto a(r+q) \overline{y^+} + ar + (q-1)a$;  \\
\textcolor{white}{Substitution Reductions:} $a \overline{x^-} \mapsto a \overline{y^-} (r'+q^{-1}) + ar' + (q^{-1}-1) a$.
\end{itemize} 
This produces the same ambiguities as before, which resolve in the same way after eliminating the nonreduced
generators $\overline{y^{\pm}}$ (another way to say this is that the reductions are the same up to the change of variables, so ambiguities resolve if and only if they did before). The modified ordering function,
$$
N^{z}(w) := (n_a(w), n_{\{a,a^*\},\{\overline{x^+},\overline{x^-},\overline{y^+}, \overline{y^-}\}}(w), n_{\{a \overline{x^+}, a \overline{x^-}\}}(w), n_{\{\overline{y^+}, \overline{y^-}\}}(w)),
$$
is strictly decreasing under applications of reductions and hence every term reduces after applying finitely many reductions. So we have proven the following variant of Proposition \ref{prop: basis for An}:

\begin{prop} \label{p:basis_for_cycle_w/sub}
Let $Q \cong \widetilde{A}_{n}$ be a cycle. Then $L_{Q}$ is a free left $kQ_{0}$-module with basis given by alternating words in $\mathfrak{R}$ and $\mathfrak{B}' := \{ (\overline{x^{\pm}})^{m} a^{\ell}, (\overline{x^{\pm}})^{m} (a^*)^{\ell} \mid m \in \bN, \ell \in \bN \}.$
Hence $\mathfrak{B}'$ is a basis for $\Lambda^{q}(Q)$.
\end{prop}

In the case of the partial multiplicative preprojective algebra, the modified reductions are as follows:
\begin{itemize}
    \item Inverse Reductions: $\overline{x_{a}^+} \overline{x_{a}^{-}}, \overline{x_{a}^{-}} \overline{x_{a}^+} \mapsto -\overline{x_a^+} - \overline{x_a^-}$ for $a \in \overline{Q}_{1}$.
    \item Short Cycle Reductions: $a a^* \mapsto \overline{x_{a}^{+}}$ for $a \in \overline{Q}_{1}$.
    \item Reordering Reductions: $a^* \overline{x_{a}^{\pm}} \mapsto \overline{x_{a^*}^{\pm}} a^*$ for $a \in \overline{Q}_{1}$. 
    \item Substitution Reductions: $\overline{x_{a}^{\pm}} \mapsto \red_{a}^{\pm}-1, \quad \overline{x_{a^*}^{-}} \mapsto - a^* \red_{a}^{-1} a$, \\ 
    \textcolor{white}{Substitution Reductions:} $\overline{x_{a^*}^+}^{2} \mapsto -\overline{x_{a^*}^{+}} + a^* \red_{a} a, 
    \quad \overline{x_{a^*}^{+}} a^* \mapsto a^* (\red_{a}-1), \text{ for } a \in F_{1}$. 
\end{itemize}
Again, the same ordering function applies here and strictly decreases under these reductions. The ambiguities must resolve since they did before.

\begin{prop} \label{p:basis_for_pmpa_w/sub}
Let $Q, \cB, \cW$ be as in Proposition \ref{p:pmpa-basis}. 
Then $L_{Q}$ is a free left $kQ_{0}$-module with basis given by concatenable words in the set,
\[
\{ a, \ \overline{x_{a}^+}, \ \overline{x_{a}^{-}} \mid a \in \overline{Q}_1 \} \cup \{r_{\cB}, \  r'_{\cB}\},
\]
such that the following subwords do not occur:
\[
\begin{array}{cccccl}
\overline{x_a^+} \overline{x_a^{-}}, &  \overline{x_a^{-}} \overline{x_a^+}, &  a a^*, &  a \overline{x_{a^*}^{\pm}}, && \text{ for } a \in \overline{Q}_1 \\  
\overline{x_{a}^{\pm}}, & \overline{x_{a^*}^{-}}, & \overline{x_{a^*}^+}a^*, & \overline{x_{a^*}^+}^2, && \text{ for } a \in F_1 \\
r_{\cB}r_{\cB}', & r_{\cB}'r_{\cB}. &&&&
\end{array}
\]
A basis for $\Lambda_Q^q$ as a free $kQ_0$-module  is given by those words above not containing $r_{\cB}, r_{\cB}'$.
\end{prop}

\begin{rem}
Note that, for the following subsection, we only require the substitutions $\overline{x_a}$ in the case where the arrow $a$ begins at a white vertex (which in particular implies that $a \notin F_1$, although it could be that $a^* \in F_1$).  If we only make these substitutions, it is similarly easy to write the above reductions in the case where for certain arrows $\overline{x_a^{\pm}}$ appears and for others $x_a^{\pm1}$ appears; we leave this to the reader.
\end{rem}

The only thing that we require from the above in the next subsection is the following observation: 
\begin{equation} \label{eq:reductions_preserve_augmentation} \text{\emph{reductions on} } 1_{\cW} L_Q 1_{\cW}  \text{  \emph{preserve the augmentation ideal,} } \ker(\Lambda^{q}(Q, \cW) \rightarrow k\cW). 
\end{equation}
In other words, any monomial of positive length beginning and ending at white vertices reduces to a linear combination of other such monomials. This was not true with the original generators (e.g., looking at the inverse reductions).

\subsection{Quivers containing cycles} \label{ss: cont cycle}
In this subsection, we prove the strong free product property for a connected quiver containing a cycle, along with providing a natural decomposition and basis for its multiplicative preprojective algebra.
In more detail, the multiplicative preprojective algebra decomposes (as a vector space) into a free product of the multiplicative preprojective algebra for the cycle and a \emph{partial} multiplicative preprojective algebra for the complement of the cycle. This technique should extend to the case of general extended Dynkin quivers, hence reducing Conjecture \ref{conj: 2CY} to the extended Dynkin case.

Let $Q$ be a connected quiver containing a cycle $Q_E$, with complement $Q' := Q \backslash Q_E$. Let $\cW := (Q_{E})_{0}$, so the vertices of the cycle are white. Fix $q \in (k^*)^{Q_{0}}$ and a decomposition $q = (q_{E}, q')$. There is a linear isomorphism:
\begin{equation} \label{e:iso-general-q}
\Psi: \Lambda^{q_{E}}(Q_{E}) *_{kQ_{0}} \Lambda^{q'}(Q', \cW) \to \Lambda^{q}(Q).
\end{equation}
We prove this by producing a basis of $\Lambda^{q}(Q)$ of alternating words in $\Lambda^{q_{E}}(Q_{E})$ and $\Lambda^{q'}(Q', \cW)$.

 \begin{rem}
 For the (deformed) \emph{additive} preprojective algebra, the analogous map,
 \[
 \Psi_{\text{add}}: \Pi^{\lambda_{E}}(Q_{E}) *_{kQ_{0}} \Pi^{\lambda'}(Q', \cW) \to \Pi^{\lambda}(Q),
 \]
is an isomorphism for all connected quivers $Q$ containing an extended Dynkin quiver $Q_{E}$. This follows from the proof of \cite[Theorem 3.4.2]{EE} (see also \cite[Section 5]{Schedler_HH}, particularly Corollary 5.2.9.(ii)).  
 \end{rem}

As before, let $B := kQ_0[t,(q+t)^{-1}]$ and $\overline B = tB$, which is spanned by elements $\{ t^m, (t')^m \mid m \geq 1 \}$ where $t' := (q+t)^{-1}-q^{-1}$.

\begin{prop}\label{p:mpa-general-q}
Let $Q$ be a connected quiver containing a cycle $Q_E \subseteq Q$ ($Q_E \cong \widetilde{A}_{n-1}$).
Then there exists a section $\sigma: \Lambda^{q}(Q) \rightarrow L$ such that $(L, r, \sigma, B)$ satisfies the strong free product property. 

In more detail, $L_Q$ is a free left $kQ_{0}$-module with basis given by concatenable alternating products in the bases of  $\Lambda^{q_E}(Q_E)$  given by Proposition \ref{prop: basis for An} or Proposition \ref{p:basis_for_cycle_w/sub}, of $\Lambda^{q'}(Q',\cW)$   given by Proposition \ref{p:basis_for_pmpa_w/sub}, and $r^m, (r')^m$  ($m\geq 1$).

\end{prop}
\begin{cor}\label{cor: flat} Let $Q$ be as in Proposition \ref{p:mpa-general-q}.
A basis for $\Lambda^q(Q)$ is given by concatenable alternating words in the mentioned bases of $\Lambda^{q_E}(Q_E)$ and $\Lambda^{q'}(Q',\cW)$. In particular, the family $\Lambda^q(Q)$ defines a free $k[q_i,q_i^{-1}]_{i \in Q_0}$-module, and hence is flat over $(k^\times)^{Q_0}$.
\end{cor}
\begin{rem} \label{rem:substitution}
Note in Proposition \ref{p:mpa-general-q} that we only need to replace  $x_a^{\pm1}$ for $\overline{x_a^{\pm}}$ if $a \in \overline{Q'}_1$ begins at a vertex of $Q_E$. Moreover, making this change to the statement does not affect the proof. On the other hand, we could freely replace $x_a^{\pm1}$ by $\overline{x_a^{\pm}}$ for \emph{all} arrow in $\overline{Q}_{1}$, again without changing the proof.
\end{rem}

\begin{proof}[Proof of Proposition \ref{p:mpa-general-q}]
First we will establish that our proposed basis for $L$ implies the strong free product property. To see this, observe that the set of subwords not containing $r_{i}, r'_{i}$ for $i \in Q_{0}$ form a basis for $\Lambda^{q}(Q)$. The inclusion of basis elements $\Lambda^{q}(Q) \rightarrow L$ defines a section $\sigma$. Using $\sigma$ define $\sigma': \Lambda^{q}(Q) *_{kQ_{0}} kQ_{0}[t, (q+t)^{-1}] \rightarrow L$ to be the extension of the map taking $t \mapsto r$, $(q+t)^{-1} \mapsto r'+q, \text{and } p \mapsto \sigma(p)$ for $p \in \Lambda^{q}(Q)$. Then $\sigma'$ is clearly a  $kQ_{0}$-linear isomorphism, and hence $(L, r, \sigma, B)$ satisfies the strong free product property.

Next we will show that the proposed basis for $L$ implies that there exists a $kQ_{0}$-linear isomorphism: $\Psi: \Lambda^{q_{E}}(Q_{E}) *_{kQ_{0}} \Lambda^{q}(Q', \cW) \rightarrow \Lambda^{q}(Q)$. For this, identify:
\begin{itemize}
    \item $\Lambda^{q}(Q)$ as the span of words in $L$ without the subwords $r_{i}, r_{i}'$,
    \item $\Lambda^{q_{E}}(Q_{E})$ as the span of words in $\Lambda^{q}(Q)$ without the subwords $a, \overline{x_{a}^{\pm}}$ for $a \in \overline{Q'}_{1}$, and
    \item  $\Lambda^{q}(Q', \cW)$ as the span of words in $\Lambda^{q}(Q)$ without the subwords  $b_{i}, x_{b_{i}}^{\pm1}$ for $b_{i} \in \overline{Q_{E}}_{1}$. 
\end{itemize}
 Hence there exists $kQ_{0}$-linear maps $\iota_{1}, \iota_{2}: \Lambda^{q_{E}}(Q_{E}), \Lambda^{q}(Q', \cW) \rightarrow \Lambda^{q}(Q)$ defined by the inclusion of basis elements. These maps determine a unique injective $kQ_{0}$-linear map $\Psi:= \iota_{1} *_{kQ_{0}} \iota_{2}: \Lambda^{q_{E}}(Q_{E}) *_{kQ_{0}} \Lambda^{q}(Q', \cW) \rightarrow \Lambda^{q}(Q)$, which is clearly surjective, hence an isomorphism. 
 
It remains to establish that the given set is indeed a basis for $L$. By Proposition \ref{prop: basis for An} we have a basis $\mathfrak{B}_{Q_{E}}$ for $L_{Q_{E}}$ and by Proposition \ref{p:basis_for_pmpa_w/sub} we have bases $\mathfrak{B}_{Q'}$ for $L_{Q'}$. Therefore we have a basis of alternating words in $\mathfrak{B}_{Q_{E}}$ and $\mathfrak{B}_{Q'}$ for $L =L_{Q_{E}} *_{kQ_{0}} L_{Q'}$.  


However this basis gives rise to a basis for the quotient $L/(\rho_{Q_{E}}+\rho_{Q'}1_{\cB}-q)$. So we need to show that $L/(\rho_{Q_{E}}+\rho_{Q'}1_{\cB}-q)$ is isomorphic to $ L/(\rho_{Q_{E}}\rho_{Q'}-q)=:\Lambda^{q}(Q)$ as $kQ_{0}$-modules. Hence we consider the system of reductions combining the systems of reductions from Proposition \ref{prop: basis for An} and Proposition \ref{p:basis_for_pmpa_w/sub}. Crucially, we perturb the system of reductions by perturbing the relation $r^{\text{pre}} :=\rho_{Q_{E}}+\rho_{Q'}1_{\cB}-q$ to $r =  \rho_{Q_E}\rho_{Q'} - q$. 



First observe that this change does nothing to the reductions for $L_{Q'}$, since the transformation is the identity on black vertices. That is, $r^{\text{pre}} 1_{\cB} = \rho_{Q'}1_{\cB}-q 1_{\cB} = r1_{\cB}$. 

For $L_{Q_E}$, notice $r^{\text{pre}} 1_{\cW} = \rho_{Q_{E}}-q_{E}$ while $r 1_{\cW} = \rho_{Q_{E}}\rho_{Q'} - q_{E}$. So we alter each reduction involving $\rho_{Q_{E}}^{\pm}$ by the transformation:
\[
\rho_{Q_{E}} \mapsto \rho_{Q_{E}} \rho_{Q'} = \rho_{Q_{E}} (\rho_{Q'}-1)+ \rho_{Q_{E}} \quad \quad \rho_{Q_{E}}^{-1} \mapsto  \rho_{Q'}^{-1} \rho_{Q_{E}}^{-1} =  (\rho_{Q'}^{-1}-1)\rho_{Q_{E}}^{-1}+ \rho_{Q_{E}}^{-1}.
\]
Note that we choose this form for the transformation to emphasize that the new relation splits as a sum of (I) the old relation and (II) a piece in the ideal generated by $\overline{x_{a}^{\pm}}$ for $a \in \overline{Q'}_{1}$.
 This transformation only effects the Substitution Reductions in the original reduction system for the cycle, see the proof of Proposition \ref{prop: basis for An}. The Substitution Reductions become (after applying a Reordering Reduction) the following: 
\begin{align*} 
y^{-1} &\mapsto x^{-1}(\rho_{Q'} - 1)(r +q) + x^{-1}(r+q), \quad \text{  (if not preceded by }a),  \\
y &\mapsto (r' + q^{-1})(\rho_{Q'}^{-1} - 1) x + (r'+q^{-1})x \quad \text{  (if not preceded by }a), \\
a x &\mapsto a(\rho_{Q'}-1) (r+q) y + ary + qya, \\
 ax^{-1} &\mapsto ay^{-1} (r'+q^{-1})(\rho_{Q'}^{-1}-1) + ay^{-1}r'+q^{-1}y^{-1} a.
\end{align*} 



Order monomials in $L$ lexicographically in the orderings $N$ and $N'$ of Propositions \ref{prop: basis for An} and \ref{p:pmpa-basis}. Then the above reductions
strictly decrease the ordering. Here we are using \ref{eq:reductions_preserve_augmentation} from the previous subsection to deduce that the ideal of positive-length monomials beginning and ending at vertices of $Q_E$ is preserved under reductions.

All ambiguities lie either entirely in $L_{Q_{E}}$ or entirely in $L_{Q'}$. Hence the ambiguities in $L_{Q'}$ resolve as before. The ambiguities in $L_{Q_{E}}$ still resolve using the same reductions as before perturbing. To see this, note that we have replaced the formal variables $\rho_{Q_{E}}^{\pm1}$ (which do not interact with $a, a^*, x^{\pm1}, y^{\pm1})$ with the new formal variables $(\rho_{Q_{E}} \rho_{Q'})^{\pm1}$. 

Since the perturbed system of reductions has all the same leading coefficients as the original, we conclude that $L$ has the desired basis.
\end{proof}

\section{The center and primality of multiplicative preprojective algebras} \label{s:center}
Let $Q$ be a connected quiver strictly containing a cycle. 
The goal of this section is to complete the proof of Theorem \ref{thm: 2CY} by first establishing that $\Lambda^{q}(Q)$ is prime and then that $Z(\Lambda^{q}(Q)) = k$ and hence the Calabi--Yau structure is unique up to rescaling.

\subsection{Primality of  multiplicative preprojective algebras}
We will show $\Lambda^{q}(Q)$ is prime by first showing that left multiplication by certain elements is injective on the subspace of concatenable elements. 

\begin{lem} \label{lem: left mult by a is inj}
Let $a$ denote the sum of all the positively oriented arrows of the cycle in $\overline{Q}_{1}$. Then left multiplication by $a$, $L_{a}: 1_{\cW} \Lambda^{q}(Q) \rightarrow 1_{\cW} \Lambda^{q}(Q)$, is injective. 
\end{lem}

\begin{proof}
Decompose the vertices $Q_{0} = \cB \sqcup \cW$ where the white vertices are in the cycle. Decompose the arrows in $\overline{Q}_{1} = \overline{Q_{E}}_{1} \sqcup \overline{Q'}_{1}$. Define
\[
A_{+} := \ker( \epsilon_{A}: \Lambda^{q_E}(Q_E) \rightarrow k \cW )
\hspace{1cm}
B_{+} := \ker( \epsilon_{B}: \Lambda^{q}(Q', \cW) \rightarrow k \cW ).
\]
Then one can define a descending filtration by $\cF_{0} = \Lambda^{q}(Q)$ and $\cF_{m} := \text{Span}( B_{+} (A_{+}B_{+})^{\geq m}) )$ for $m > 0$. Notice $a \in \cF_{m}, b \in \cF_{\ell}$ implies $ab \in \cF_{m+\ell}$, so this is an \emph{algebra} filtration.

Consider the exact sequence $B_{+} \overset{\iota}{\hookrightarrow} \Lambda^{q}(Q) \overset{\pi}{\twoheadrightarrow} \Lambda^{q_E}(Q_E)$. The basis of Proposition \ref{p:mpa-general-q} realizes an inclusion $i: \Lambda^{q_E}(Q_E) \rightarrow \Lambda^q(Q)$, a $kQ_{0}$-module splitting. So for $\alpha, \beta \in \Lambda^{q_E}(Q_E)$, $i(\alpha) \cdot i(\beta) \equiv  \alpha \cdot \beta$ modulo the two-sided ideal generated by $B_{+}$. Therefore, in the associated graded algebra $gr_{\cF}(\Lambda^q(Q)) :=\oplus_{m=0} \cF_{m}/ \cF_{m+1}$,
\[
i(\alpha) \cdot i(\beta) = \alpha \cdot \beta + \alpha \cdot \beta b'' + \alpha b' \beta  + b \alpha \cdot \beta 
\]
for $b, b', b'' \in B_{+}$. Therefore, for $b_1 \in B_{+}$, there exists $b_2 \in B_+$ such that
\[
i(\alpha) \cdot i(\beta) b_1 = \alpha \cdot \beta (1+b_2)b_1. 
\]

Recall $A_{+}$ has $kQ_{0}$-module basis given by $\{ x_{a}^{p},  \ x_{a}^{m}a^\ell, \ x_{a}^{m} (a^*)^{\ell} \mid m, \ell, p \in \bZ, \ p \neq 0, \ \ell > 0 \}$ by Proposition \ref{prop: basis for An}. In the associated graded algebra $gr_{\cF}(\Lambda^q(Q))$, $L_{a}$ acts on $A_{+} B_{+}$ as follows:
\begin{align*}
    a & (x_{a}^{m} a^{\ell}) b = q^{m} x_{a}^{m} a^{\ell+1} b \\
    a & (x_{a}^{m} (a^*)^{\ell}) b = q^{m} x_{a}^{m} (x_{a} -1)(a^*)^{\ell-1} b \\
    a & (x_{a}^{p}) b = q^{p} x_{a}^{p}(x_a -1)(\rho_{Q'}^{-1}) b
\end{align*}
for $b \in B_{+}$. Since $L_{a}$ is injective on $A_{+}$, by Proposition \ref{prop: prime}, we conclude that $L_{a}$ is injective on the right ideal generated by $A_{+}$. 

Consider the basis of Proposition \ref{p:mpa-general-q}, and write $b \in 1_{\cW} B_+$ in this basis. Then $ab$ is again a basis element, and hence $L_{a}$ takes basis elements injectively to basis elements. We conclude that $L_{a}$ is injective on the right ideal generated by $1_{\cW} B_+$, and therefore on all of $1_{\cW} \Lambda^{q}(Q)$.
\end{proof}

\begin{lem} \label{lem: right mult by a is inj}
Right multiplication by $a$, $R_a: \Lambda^{q}(Q) 1_{\cW} \rightarrow \Lambda^{q}(Q) 1_{\cW}$ is injective. 
\end{lem}

The proof is completely analogous, using the same filtration, together with the calculations:
\begin{align*}
    b (x_{a}^{m} a^{\ell}) & a = b x_a^m a^{\ell+1} \\
    b (x_{a}^{m} (a^*)^{\ell}) & a = b  q^{-\ell+1} x_{a}^{m+1} (a^{*})^{\ell-1} - b x_{a}^{m} (a^{*})^{\ell-1} \\
    b (x_{a}^{p}) & a = b x_{a}^p a. 
\end{align*}

\begin{lem} \label{lem: right mult by gamma is inj}
Let $v \in Q_{0}$. There is unique path $\gamma_{v, w}$ in the spanning forest from $v$ to a white vertex $w \in \cW$. Right multiplication by $\gamma_{v, w}$, $R_{\gamma_{v, w}}:  \Lambda^{q}(Q)e_{v} \rightarrow \Lambda^q(Q)e_{w}$, is injective. 
\end{lem}

\begin{proof}
We need to show $\alpha \gamma_{h(\alpha), w} \neq 0$ for $\alpha \neq 0$. Consider the basis in Proposition \ref{p:basis_for_pmpa_w/sub}, consisting of words in $a, \overline{x_{a}^{\pm}}$ for $a$ an arrow, without certain disallowed subwords, e.g. $aa^*$ for $a \in \overline{Q}_{1}$. Note that $\gamma_{h(\alpha), w}$ is a basis element as $aa^*$ cannot appear in a shortest path. Write $\alpha$ as a linear combination of basis elements. Notice $\alpha \gamma_{h(\alpha), w}$ is a linear combination of basis elements \emph{unless} the disallowed subword $a^*a$ is created for some arrow $a \in F_{1}$. This disallowed subword reduces to $\overline{x_{a^*}^+}$ (which is not itself disallowed since $a^* \not \in F_{1}$, as $a \in F_{1}$.) Furthermore, the appearance of $\overline{x_{a^*}^+}$ for $a \in F_{1}$ cannot create the disallowed subwords: 
\[
\text{(I) } \overline{x_{a^*}^+} \overline{x_{a^*}^-}, \quad
\text{(II) } \overline{x_{a^*}^-} \overline{x_{a^*}^+}, \quad
\text{ (III) }a \overline{x_{a^*}^{\pm}}, \quad
\text{ (IV) } \overline{x_{a^*}^+}^2, \quad
\text{ (V) } \overline{x_{a^*}^+}a^*, 
\]
for $a \in F_{1}$, as in each case $\alpha$ or $\gamma_{h(\alpha), w}$ would itself contain a disallowed subword: 
\[
\text{(I) } \overline{x_{a^*}^-}, \quad
\text{(II) } \overline{x_{a^*}^-}, \quad 
\text{ (III) }aa^*, \quad
\text{ (IV) }a \overline{x_{a^*}^+} \text{ or  } \overline{x_{a^*}^+} a^*, \quad
\text{ (V) } aa^*,
\]
each a contradiction. 
We conclude that right multiplication by $\gamma_{h(\alpha), w}$ takes basis elements injectively to basis elements and hence is injective.
\end{proof}

\begin{lem} \label{lem: left mult by gamma is inj}
Let $v \in Q_{0}$. There is unique path $\gamma_{w, v}$ in the opposite of the spanning forest, $F_{1}^{\text{op}}$, from $w \in \cW$ to $v$. Left multiplication by $\gamma_{w, v}$, $L_{\gamma_{w, v}}:  e_{v} \Lambda^{q}(Q) \rightarrow e_{w} \Lambda^q(Q)$, is injective. 
\end{lem}

The proof is identical, and follows from the isomorphism $\Lambda^{q}(Q) \cong \Lambda^{q}(Q)^{\text{op}}$. 

\begin{prop} \label{prop: prime for quivers containing cycle}
$\Lambda^{q}(Q)$ is prime, for $Q$ connected and containing a cycle. 
\end{prop}

\begin{proof}
Let $\alpha, \beta\in \Lambda^{q}(Q)$ be nonzero. We will show $\alpha \Lambda^{q}(Q) \beta \neq 0$ by building an explicit element $\gamma = \gamma_{1} \gamma_{2}$ so that $\alpha \gamma \beta \neq 0$. That is, define
\[
\gamma_1 := \gamma_{h(\alpha), w} x_{a}^{M} a^{N}  \hspace{1cm} \gamma_2 := a^{N'} x_{a}^{M'} \gamma_{w, t(\beta)}
\]
where $M, M', N, N' \in \bN$ are sufficiently large (depending on $\alpha$ and $\beta$) and where $\gamma_{h(\alpha), w}$ and $\gamma_{w, t(\beta)}$ are as defined in Lemma \ref{lem: right mult by gamma is inj} and Lemma \ref{lem: left mult by gamma is inj} respectively.

We will first show that right multiplication by $\gamma_{1}$ is injective on concatenable paths to conclude $\alpha \gamma_{1} \neq 0$. Then we will argue that left multiplication by $\gamma_{2}$ is injective on concatenable paths to conclude $\gamma_{2}\beta \neq 0$. Finally, we will show that $\alpha \gamma_{1} \gamma_{2} \beta \neq 0$.

To show $R_{\gamma_{1}} : \Lambda^{q}(Q) e_{h(\alpha)} \rightarrow  \Lambda^{q}(Q)e_{h(\gamma_{1})}$ is injective, it suffices to show that right multiplication by each piece, $\gamma_{h(\alpha), w}$,  $x_{a}^{M}$, and  $a^{N}$, is injective. $R_{\gamma_{h(\alpha), w}}$ is injective by Lemma \ref{lem: right mult by gamma is inj}, $R_{x_{a}^{M}}$ is injective since $x_a$ is invertible, and $R_{a^{N}}$ is injective by Lemma \ref{lem: right mult by a is inj}.

Similarly, $L_{\gamma_{2}}: e_{t(\beta)} \Lambda^{q}(Q) \rightarrow e_{t(\gamma_{2})}\Lambda^{q}(Q)$ is injective since $L_{a}, L_{x_{a}},$ and $L_{\gamma_{w, t(\beta)}}$ are injective by Lemma \ref{lem: left mult by a is inj}, invertibility of $x_a$, and Lemma \ref{lem: left mult by gamma is inj}, respectively. 

Finally notice that $\alpha \gamma_{1} \neq 0$ and $\gamma_{2} \beta \neq 0$ implies $\alpha \gamma_{1} \gamma_{2} \beta \neq 0$. To see this, consider the filtration $\cF$ defined in the proof of Lemma \ref{lem: left mult by a is inj}. It suffices to show $\alpha \gamma_{1} \gamma_{2} \beta \neq 0$ in $\text{gr}_{\cF}(\Lambda^{q}(Q))$. Write $\alpha \gamma_{1}$ and $\gamma_{2} \beta$ in the basis of Proposition \ref{p:mpa-general-q} (see the basis in Proposition \ref{prop: basis for An}). By design $\alpha \gamma_{1}$ ends with a basis element of the form $x_{a}^{m} a^{n}$ for $m, n >0$ and $\gamma_{2} \beta$ begins with a basis element of the form $a^{n'} x_{a}^{m'}$ for $m', n' >0$. Their product in $gr_{\cF}(\Lambda^{q}(Q))$ is the scaled basis element $q^{nm'} x_{a}^{m+m'} a^{n+n'}$. So $\alpha \gamma_{1} \gamma_{2} \beta \neq 0$ in $\text{gr}_{\cF}(\Lambda^{q}(Q))$ and hence in $\Lambda^{q}(Q)$, completing the proof. 

\end{proof}

\subsection{The center of multiplicative preprojective algebras}

The center of $\Lambda^{1}(Q)$ depends dramatically on the taxonomy of quiver $Q$ into Dynkin, extended Dynkin, and others. 
\begin{itemize}
\item For $Q$ Dynkin and $k$ characteristic not 2, 3, or 5, one can compute the center using the isomorphism $\Lambda^{1}(Q) \cong \Pi(Q)$, see Example \ref{exam: Dynkin}.
\item For $Q$ extended Dynkin, Conjecture \ref{conj: NCCR} predicts $Z(\Lambda^{1}(Q)) \cong e_{v} \Lambda^{1}(Q) e_{v}$, which is proven in Section \ref{s:NCCR_proof} 
in the case $Q=\widetilde{A_n}$. 
\item In the remaining cases, Conjecture \ref{conj: 2CY} predicts $Z(\Lambda^{q}(Q))=k$, for any $q \in (k^*)^{Q_{0}}$. 
\end{itemize}
The goal of this section is to establish the conjecture in the case $Q$ contains a cycle. 

\begin{prop} \label{prop: general center trivial}
Let $Q$ be a connected quiver strictly containing a cycle and fix $q \in (k^{\times})^{Q_{0}}$.  Then $Z(\Lambda^{q}(Q)) = k$.
\end{prop}

\begin{proof}
Let $z \in Z(\Lambda^{q}(Q))$. Decompose $z = z_{0} + z_{+}$ into a sum of length zero and positive length paths.  First suppose that $z_+ = 0$. Then $z = \sum_{i \in Q_0} c_i e_i$. Note that every individual arrow forms a basis element of Proposition \ref{p:mpa-general-q}. 
Then $z a = a z$ for every arrow implies that all $c_i$ are equal, as $Q$ is connected.
 


Now assume $z_+ \neq 0$. Expanding $z_+$ in the basis of Proposition \ref{p:mpa-general-q}, we write $z_+ =\sum_{i} c_{i} z_{i}$, where each $z_i$ is a positive-length alternating word in the cycle and the complement. We claim that each $z_{i}$ has an arrow not in the cycle. Suppose, by contradiction, there exists $j$ such that $z_{j}$ consists of only arrows in the cycle. Since $Q$ strictly contains the cycle, there exists an arrow $b \in Q_{1}$ not in the cycle. And as $z_+$ commutes with each arrow $a_{i}$ in the cycle, there exists $l$ such that $z_l$ consists of only arrows in the cycle that ends at $t(b)$. Then $z_+  b = b z_+$.
But $z_+ b$ contains a term beginning with $x_a^{m} a^{j}$ for some $m, j$ with $(m, j) \neq (0,0)$. However, $b z_{+}$ has no term beginning $x_a^{m} a^{j}$ unless $(m, j) = (0,0)$. This contradicts the existence of $z_{j}$ consisting of only arrows in the cycle. 

Since $z_+ \neq 0$, thanks to
Lemma \ref{lem: right mult by gamma is inj},  there exists a vertex $i$ and a path $b =\gamma_{h(z_+), i}$ such that $z_+ b e_i \neq 0$.  Therefore also $b z_+ e_i \neq 0$, so $z_+ e_i \neq 0$.  By Lemma \ref{lem: right mult by a is inj}, we then have $z_+ a^n \neq 0$ for all $n$. Hence also $a^n z_+ \neq 0$.  Now, for sufficiently large $N \gg 0$, $a^{N} z_{+}$ contains basis elements beginning with an arbitrarily high power of the cycle. However,  terms of $z_{+} a^{N}$ begin only with powers of the cycle appearing in $z_+$, 
since every $z_j$ has a term not in the cycle. These powers are bounded, so this contradicts the assumption that $z_+ \neq 0$. We conclude that $z$ is a scalar multiple of the identity.
\end{proof}

\begin{cor} \label{c: unique CY}
If $Q$ is connected and properly contains a non-oriented cycle, then $\Lambda^{q}(Q)$ has a unique, up to scaling, Calabi--Yau structure. 
\end{cor}

\begin{proof}
Write $\Lambda := \Lambda^{q}(Q)$. Any two Calabi--Yau structures  differ by an invertible map in $\Hom_{\Lambda-\text{bimod}}(\Lambda, \Lambda)$, which is determined by the image of the unit, a central invertible element. So the set of Calabi--Yau structures on $\Lambda$, when non-empty, is a $Z(\Lambda)^{\times}$-torsor. By Proposition \ref{prop: general center trivial}, $Z(\Lambda)^{\times} = k^\times$, so any two Calabi--Yau structures differ by an invertible scalar.
\end{proof}

This completes the proof of Theorem \ref{thm: 2CY}.

 
\bibliographystyle{alpha}
\bibliography{MPA-Final-2ndVersion}


\Addresses
\end{document}